\documentclass[a4paper,11pt]{article}
\usepackage[dvips]{epsfig}
\usepackage{amsmath,amssymb,euscript,graphicx,amsfonts}
\usepackage{enumerate,color}
\usepackage{graphicx}
\usepackage{hyperref}

\newtheorem{theorem}{Theorem}[section]
\newtheorem{proposition}[theorem]{Proposition}
\newtheorem{lemma}[theorem]{Lemma}

\newtheorem{problem}[theorem]{Problem}
\newtheorem{example}[theorem]{Example}
\newtheorem{remark}[theorem]{Remark}

\newcommand{\Qed}{\rule{2.5mm}{3mm}}
\newenvironment{proof}{{\noindent \sc Proof.}}{\hfill $\Qed$ \\}
\renewcommand{\H}{\mathcal{H}}
\renewcommand{\P}{\mathcal{P}}

\newcommand{\W}{\mathcal{W}}
\newcommand{\B}{\mathcal{B}}
\renewcommand{\S}{\mathcal{S}}
\newcommand{\la}{\langle}
\newcommand{\ra}{\rangle}
\newcommand{\ZZ}{\mathbb{Z}}

\newcommand{\OO}{\mathcal{O}}
\newcommand{\mb}{\mathbf}

\newcounter{case}
\newenvironment{case}[1][\unskip]{\refstepcounter{case}\sc
\medskip \noindent Case \thecase\ #1.\ }{\unskip\upshape}
\renewcommand{\thecase}{\arabic{case}}

\newcounter{subcase}

\numberwithin{subcase}{case}

\def\PG{\hbox{\rm PG}}
\def\PGL{\hbox{\rm PGL}}
\def\GL{\hbox{\rm GL}}
\def\GF{\hbox{\rm GF}}
\def\SL{\hbox{\rm SL}}

\def\PSL{\hbox{\rm PSL}}

\def\di{\bigm|} \def\lg{\langle} \def\rg{\rangle}

\def\a{\alpha}    \def\e{\varepsilon}
    
\def\th{\theta} \def\ld{\lambda} 
\def\D{\S}  
 \def\O{\Omega} 
\def\Aut{\hbox{\rm Aut\,}}
\def\ow{\overline W}  

\def\di{\bigm|} \def\lg{\langle} \def\rg{\rangle}

\newcommand{\mm}[4]{
\left[ \begin{array}{cc}
#1 & #2 \\
#3 & #4
\end{array} \right]
}

\newenvironment{proofT}{{\noindent \sc Proof of Theorem~\ref{the:polynomial-degree-4}.}}{\hfill $\Qed$ \\}
\newenvironment{proofTT}{{\noindent \sc Proof of Theorem~\ref{the:main}.}}{\hfill $\Qed$ \\}

\usepackage{minitoc}

\begin{document}


\begin{center}
{\bf\large
HAMILTON CYCLES IN VERTEX-TRANSITVE \\
GRAPHS OF ORDER A PRODUCT OF TWO PRIMES}
\end{center}


\begin{center}
Shaofei Du{\small $^{a,}$}\footnotemark,
Klavdija Kutnar{\small $^{b,c,}$}\footnotemark \ and
Dragan Maru\v si\v c{\small $^{b,c,d,}$}\addtocounter{footnote}{0}\footnotemark$^{,*}$  \\

\medskip
{\it {\small
$^a$Capital Normal University, School of Mathematical Sciences,
Bejing 100048, People's Republic of China\\
$^b$University of Primorska, UP FAMNIT, Glagolja\v ska 8, 6000 Koper, Slovenia\\
$^c$University of Primorska, UP IAM, Muzejski trg 2, 6000 Koper, Slovenia\\
$^d$IMFM, Jadranska 19, 1000 Ljubljana, Slovenia\\
}}
\end{center}

\addtocounter{footnote}{-2}
\footnotetext{This work is supported in part  by the National Natural Science Foundation of China (11671276).}
\addtocounter{footnote}{1}
\footnotetext{This work is supported in part by the Slovenian Research Agency (research program P1-0285 and research projects
N1-0038, N1-0062, J1-6720, J1-6743, J1-7051, and J1-9110).}
\addtocounter{footnote}{1}
\footnotetext{
This work is supported in part by the Slovenian Research Agency
(I0-0035, research program P1-0285
and research projects N1-0038, N1-0062, J1-6720, and J1-9108),
and in part by H2020 Teaming InnoRenew CoE (grant no. 739574).

~*Corresponding author e-mail:~dragan.marusic@upr.si}

\vspace*{-15pt}

\begin{abstract}
A step forward is made in a long standing Lov\'{a}sz's problem
regarding hamiltonicity of vertex-transitive graphs
by showing that every connected vertex-transitive
graph of order a product of two primes, other than
the Petersen graph, contains a Hamilton cycle.
Essential tools used in the proof range from classical results on existence of Hamilton cycles, such as Chv\'atal's theorem and Jackson's theorem, to  certain results on polynomial representations of quadratic residues at primitive roots in finite fields.
\end{abstract}

\vspace*{-15pt}

\begin{quotation}
\noindent {\em Keywords: vertex-transitive graph, Hamilton cycle, automorphism group, orbital graph, finite field, polynomial.}\\
Math. Subj. Class.: 05C25, 05C45.
\end{quotation}


\section{Introduction}
\label{sec:intro}
\indent

The following  question asked by  Lov\'{a}sz \cite{L70} in 1970
tying together traversability and symmetry,
two seemingly unrelated graph-theoretic concepts,
remains unresolved after all these years.

\begin{problem}
\label{Lovasz}
{\rm \cite{L70}}
Does	every finite connected vertex-transitive  graph have a Hamilton	path?
\end{problem}

No	connected vertex-transitive graph without a Hamilton	path -- a simple path
containing all	vertices of the graph -- is known to exist. Moreover, only four
connected vertex-transitive  graphs	on at least three	vertices not having a
Hamilton	cycle	-- a simple cycle containing	 all	vertices of the graph -- are	
known	so	far: the Petersen graph, the Coxeter	graph,	and the two graphs	obtained from them	by replacing each vertex with a triangle \cite{TRG}.	
None of these four exceptional graphs is a Cayley graph, that is, a
vertex-transitive graph admitting a regular subgroup of automorphisms. This has led to a folklore conjecture that every connected Cayley graph possesses a Hamilton cycle.

Problem~\ref{Lovasz}, together with its Cayley graph variant, has spurred quite a bit of interest in the mathematical community,	resulting in a	number of papers affirming the existence of Hamilton	paths and in some cases even Hamilton cycles.

Such is the case for instance for connected vertex-transitive graphs of orders	
$kp, k \leq 6$, $10p, p\ge 11$, $p^j, j \leq 4$ and $2p^2$, where	$p$ is	a prime	
\cite{A79,C98,KM08,KMZ12,KS09,DM85,DM87,MP82,MP83}.
Furthermore, for all of these families, except for the graphs of orders $6p$ and $10p$
(and of course for the Petersen graph, truncation of the Petersen graph and the Coxeter graph), it is also known that they contain a Hamilton cycle, see \cite{A79,C98, KM09, KS09,DM85, DM87,MP82,MP83, T67}. With the obvious exception of the Petersen graph, Hamilton cycles are also known to exist in connected
vertex-transitive graphs whose automorphism groups contain a transitive subgroup with a cyclic commutator subgroup of prime-power order. This result proved in~\cite{DGMW98} uses results from a series of papers dealing with the same group-theoretic restrictions in the context of Cayley graphs \cite{D83,DM83,DW85}.

As	for Cayley graphs, most	of the results proved thus	far	 depend on	 various restrictions imposed on the corresponding Cayley groups (see \cite{A89,ACD10,AZ,CG96,DGMW98,D83,D85,
GWM11,GWM14,GKMM12,GKM09,GM07,KW85,KM09,KMMMS12,
DM83,DW85,W82,W86,WM15,WM18}).
Among them we would like to single out
Witte-Morris's proof that a connected Cayley	digraph of
an arbitrary $p$-group has a Hamilton cycle (see \cite{W86}),
Alspach's proof that connected Cayley graphs on dihedral groups of order divisible by $4$ are hamiltonian (see \cite{ACD10}), and
Ghaderpour and Witte-Morris's completion of the proof of existence of Hamilton cycles in Cayley graphs arising from nilpotent 	
and odd order groups with cyclic commutator subgroups 	
 (see \cite{GWM11,GWM14}).
And finally, with a combination of algebraic and topological
methods a Hamilton path and in some cases even a Hamilton cycle was proved to exist in
cubic Cayley graphs arising from $(2,s,3)$-generated
groups (see \cite{GKMM12,GKM09,GM07}).

Coming back to the general context of vertex-transitive  graphs, the main obstacle to making a substantial progress with regards to Problem~\ref{Lovasz} is a lack of structural results for such graphs. Despite the fact that, somewhat paradoxically, it is precisely the above problem that is responsible for much of the work directed towards obtaining such structural results by
opening up new research directions.
Such is the case for example with the so-called polycirculant
conjecture which states that every vertex-transitive graph has a
fixed-point-free automorphism of prime order
(see \cite{bcc15,seven,FKS81,DMMN07,
GX07,M81,M2,V,X}).
Such automorphisms have been of great practical use
in constructions of Hamilton cycles in vertex-transitive  graphs via the
so-called lifting cycle technique \cite{A89,DM83} thus far.
They are also an important ingredient in the proof
of the main theorem of this paper.

The aim of this paper is to make a
step forward in the above hamiltonicity problem
by giving a complete solution for
graphs of order $pq$, where $p$ and $q$ are primes.

\begin{theorem}
\label{the:main}
With the exception of the Petersen graph, a connected vertex-transitive graph of order $pq$, where $p$ and $q$ are primes,
contains a Hamilton cycle.
\end{theorem}

Since vertex-transitive graphs of prime-square order are necessarily
Cayley graphs  of abelian groups and thus
hamiltonian \cite{DM83,DM85},
the primes $p$ and $q$ may be assumed to be distinct.

The proof of Theorem~\ref{the:main} depends heavily
on the classification of vertex-transitive graphs of order $pq$
from \cite[Theorem~2.1]{MS5}
(see also \cite{MS2,MS3,MS1,MS4}),
based on whether these graphs do or do not admit a particular imprimitive subgroup of automorphisms.
Three mutually disjoint classes are identified.
The first class consists of graphs admitting an imprimitive
subgroup of automorphisms with blocks of size $p$
(the larger of the two primes). The second class consists
of graphs which admit an
imprimitive subgroup of automorphisms with blocks of size $q$,
but no subgroups of automorphisms with blocks of size $p$.
Finally, the remaining graphs are characterized by the fact that every
transitive subgroup of automorphisms is primitive.
The first and the second
of these three classes are exhausted, respectively,
by the so-called metacirculants and Fermat graphs.
In short, a {\em metacirculant} is a graph with
a transitive cyclic or metacyclic subgroup, and a {\em Fermat graph}
is a particular $q$-fold cover of a complete graph $K_p$
associated with the action of
$\SL(2,p-1)$ on $\PG(1,p-1)$ for Fermat prime $p$
and a prime $q$ dividing $p-2$ (see \cite{MS2} and
Subsection~\ref{ssec:semiregular} for a detailed description).
The graphs in the third class arise as
generalized orbital graphs associated with one of the groups
listed in Table~\ref{tab:groups}.  For the purpose of this paper
we will refer to these graphs as {\em primitive graphs}.
We would like to remark that  vertex-transitive
graphs of order $pq$, more precisely those with a primitive
automorphism group and those with an imprimitive
automorphism group with blocks of size $q$,
were also characterized by
Praeger, Wang and Xu \cite{PWX,PX,WX}.
In fact, for the analysis of hamiltonian properties of the generalized
orbital graphs arising from Row 5 of
Table~\ref{tab:groups}, a description from \cite{PX} will be used
(see Section~\ref{sec:PGL}).

Hamiltonicity of metacirculants and Fermat graphs of order
$pq$ has already been established, respectively,
in \cite{AP82a,M81a} and \cite{DM92}.
For the completion of the proof of Theorem~\ref{the:main} we
need to prove the existence of Hamilton cycles for all the remaining
graphs, that is, for primitive graphs of order $pq$.
Here we will use an approach which combines a variety
of graph-theoretic and number-theoretic tools.
For example, we will use the fact that
the polycirculant conjecture has been settled for
certain vertex-transitive graphs. In particular, a vertex-transitive graph
of order $pq$, $p>q$, contains a fixed-point-free automorphism of
order $p$. This allows an application of the lifting
cycle technique providing the quotient graph with respect
to this automorphism of order $p$ admits a Hamilton cycle.
When appropriate and when such automorphism exists, however,
this technique will be applied
to the quotient graph relative to a semiregular
automorphism of order $q$.
The existence of such a Hamilton cycle in the quotient
is established using certain  classical theorems on
Hamilton cycles, such as the well-known
Chv\'atal's theorem \cite{C72} and
Jackson's theorem \cite{BJ78}, and also using some certain
properties of
finite fields. In particular, 
we obtain a novel result on polynomials of degree
$4$ over finite fields of prime order with regards to a polynomial
representation of quadratic residues at primitive roots,
thus  refining results from \cite{MV82}
(see Theorem~\ref{the:polynomial-degree-4}).
This result will be used in Section~\ref{sec:PSL}
in all those cases for which Chv\'atal's theorem
does not suffice to prove existence of a Hamilton cycle in the
corresponding quotient.

The proof of Theorem~\ref{the:main} is a lengthy
analysis, covered in Sections~\ref{sec:strategy},~\ref{sec:small},~\ref{sec:PGL} and~\ref{sec:PSL},
of hamiltonian properties
of all possible generalized orbital graphs arising from
group actions in Table~\ref{tab:groups}.
In the sections preceding
this analysis we fix the terminology and notation,
gather same useful results and
tools, and prove the above mentioned property
of polynomials of degree $4$ over finite fields.
To be more concrete, the outline of this paper is given
in the list of sections given below:

\renewcommand*\contentsname{}
\vspace*{-15pt}
\begin{center}
\begin{minipage}[t][9.5cm][b]{0,8\textwidth}
{\scriptsize
\tableofcontents
}
\end{minipage}
\end{center}


\newpage

\section{Terminology, notation and some useful results}
\label{sec:pre}
\noindent


\subsection{Basic definitions and notation}
\label{ssec:definition}
\noindent

Throughout this paper graphs are finite, simple and undirected,
and groups are finite, unless specified otherwise.
Furthermore, a {\em multigraph} is a generalization
of a graph in which we allow multiedges and loops.
Given a graph $X$ we let $V(X)$ and $E(X)$ be the
vertex set and the edge set of $X$, respectively.
For adjacent vertices $u,v \in V(X)$ we write
$u \sim v$ and denote the corresponding edge by $uv$.
The valency of a vertex $u\in V(X)$ is denoted by $val_X(v)$
(or $val(v)$ in short). If $X$ is regular then its valency
is denoted by $val(X)$.
Let $U$ and $W$ be disjoint subsets of $V(X)$.
The subgraph of $X$ induced by $U$ will be denoted by $X\la U \ra$.
Similarly, we let $X[U,W]$  denote the bipartite subgraph
of $X$ induced by the edges having one endvertex in $U$
and the other endvertex in $W$.

Given a transitive group $G$ acting on a set $V$,
we say that a partition $\B$ of $V$ is $G$-{\em invariant}
if the elements of $G$ permute the parts, the so-called
{\em blocks} of $\B$, setwise.
If the trivial partitions $\{V\}$ and $\{\{v\}: v \in V\}$ are the only
$G$-invariant partitions of $V$, then $G$ is {\em primitive},
and is {\em imprimitive} otherwise.

A graph $X$ is {\em vertex-transitive}
if its automorphism group, denoted by $\Aut X$, acts transitively on $V(X)$.
A vertex-transitive graph is said to be {\em primitive} if
every transitive subgroup of its automorphism group
is primitive, and is said to be
{\em imprimitive} otherwise.

A graph containing a Hamilton cycle will be sometimes referred as
a {\em hamiltonian} graph.


\subsection{Generalized orbital graphs}
\label{ssec:orbital}
\noindent

In this subsection we recall the orbital graph construction which
is used throughout the rest of the paper.
Orbital graphs can be constructed for any group
action but in view of the fact that 
any transitive action of a group $G$ is isomorphic to 
the action of $G$ on the coset space of
a subgroup of $G$, we give this construction in the context
of group actions on coset spaces.
(In Section~\ref{sec:PGL}) we will, however, consider orbital
graphs with respect to the action of 
$\PSL(2,q^2)\cong P\O ^{-}(4, q)$
on particular non-singular $1$-dimensional vector spaces
over $\GF(q)$.)

An action of a group $G$ on the coset space $G/H$ with respect to
a subgroup $H\le  G$ gives rise to an action of $G$ on $G/H\times G/H$.
Its orbits are called {\em orbitals}.
Of course, the diagonal $D=\{(x,x) \colon x\in G/H\}$ is always an orbital.
Its complement, $G/H\times G/H - D$ is an orbital if and only if $G$ is doubly transitive. Unless specified otherwise, we only consider simply transitive actions.

An orbital is said to be {\em self-paired} if it simultaneously contains or does not contain ordered pairs $(x,y)$ and $(y,x)$, for $x,y\in G/H$.
For an arbitrary union $\cal{O}$ of orbitals (not containing the diagonal $D$), the {\em generalized orbital (di)graph} $X(G/H,\cal{O})$
of the action of $G$ on $G/H$
with respect to $\cal{O}$ is a simple (di)graph with vertex set $G/H$ and
edge set $\cal{O}$. (For simplicity reasons we will refer to any such (di)graph as an orbital (di)graph of $G$.)
It is an (undirected) graph if and only if $\cal{O}$ coincides with its symmetric closure, that is, $\cal{O}$ has the property that $(x,y)\in\cal{O}$ implies $(y,x)\in\cal{O}$.
Further, the generalized orbital graph $X(G/H,\cal{O})$ is
said to be a {\em basic orbital graph} if $\cal{O}$ is
a single orbital or a union of a single orbital and its symmetric closure.

In terms of symmetry, the group $G$ acts transitively on the vertices of $X(G/H,\cal{O})$, and hence orbital (di)graphs are vertex-transitive. In fact, every vertex-transitive (di)graph can be constructed in this way.

The orbitals of the action of $G$ on $G/H$ are in 1-1 correspondence with the orbits of the action of $H$ on $G/H$, called {\em suborbits} of $G$.
A suborbit corresponding to a self-paired orbital is said to be
{\em self-paired}.
When presenting the (generalized) orbital (di)graph $X(G/H,\cal{O})$
with the corresponding (union) of suborbits $\S$ the (di)graph
$X(G/H,\cal{O})$ is denoted by $X(G,H,\S)$. 

In the example below the Petersen graph (the exceptional graph
from Theorem~\ref{the:main}) is described as a
basic orbital graph arising from the alternating group $A_5$.

\begin{example}
Let $G=A_5$ be the alternating group
and let $H=\la (1\ 2), (1\ 2\ 3)\ra$ be its subgroup.
In the action of $G$ on $G/H\times G/H$ there
are two non-diagonal orbitals, both self-paired.
The corresponding
nontrivial suborbits of $H$ are, respectively, of length $3$ and $6$,
giving the Petersen graph and its complement.
\end{example}


\subsection{Semiregular automorphisms and quotient (multi)graphs }
\label{ssec:semiregular}
\noindent

Let $m\geq 1$ and $n\geq 2$ be integers. An automorphism $\rho$
of a graph $X$ is called $(m,n)$-{\em semiregular}
(in short, {\em semiregular})
if as a permutation on $V(X)$ it has a cycle decomposition consisting
of $m$ cycles of length $n$.
If $m=1$ then $X$ is called a {\em circulant}; it is
in fact a Cayley graph of a cyclic group of order $n$.
Let $\P$ be the set of orbits of $\rho$, that is,
the orbits of the cyclic subgroup $\langle \rho \rangle$ generated by $\rho$.
Let $A, B \in \P$.
By $d(A)$ and $d(A,B)$ we denote the valency of
$X\la A\ra$ and $X[A,B]$, respectively.
(Note that the graph $X[A,B]$ is
regular.)
We let the {\em quotient graph corresponding to
$\P$} be the graph $X_\P$ whose vertex set
equals $\P$ with $A, B \in \P$ adjacent if there
exist vertices $a \in A$ and $b \in B$, such that $a \sim b$ in $X$.
We let the {\em quotient multigraph corresponding to
$\rho$} be the multigraph $X_\rho$ whose vertex set is $\P$ and in
which $A,B \in \P$ are joined
by $d(A,B)$ edges.
Note that the quotient graph $X_\P$
is precisely the underlying graph of $X_\rho$.

The question whether all vertex-transitive graphs admit a semiregular
automorphism  is a famous open problem in algebraic graph theory
(see, for example, \cite{seven,DMMN07,G1,G2,GX07,M81,M2,V,X}).

A graph $X$ admitting an $(m,n)$-semiregular
automorphism $\rho$ can be represented,
following the terminology established  in \cite{MMSF07},
by an $m\times m$ array of subsets of
$H=\la\rho\ra$ as well as with
the well-known Frucht's notation \cite{RF70}.
For the sake of completeness we include both definitions.
Let $\P=\{S_i\colon i\in\ZZ_m\}$ be the set of $m$ orbits
of $\rho$, let $u_i \in S_i$ and let
$S_{i,j}$ be defined by $S_{i,j} = \{t \in H \colon u_i \sim
\rho^t(u_j)\}$. Then the $m\times m$ array
$(S_{i,j})$ is called the {\em symbol} of
$X$ relative to $(\rho;u_0,\dots,u_{m-1})$.
To give a precise definition of Frucht's notation
let $S_i = \{v_i^j\ |\ j \in \ZZ_n\}$ where $v_i^0=u_i$
and $v_i^j = \rho^j(v_i^0)$.
Then $X$ may be represented in this notation by
emphasizing the $m$ orbits of $\rho$ in the following way.
The  $m$ orbits of $\rho$ are  represented by $m$ circles.
The symbol $n/R$, where $R \subseteq\ZZ_n\setminus \{0\}$,
inside the circle corresponding to
the orbit $S_i$ indicates that for each $j\in \ZZ_n$,
the vertex $v_i^j$ is adjacent to all the vertices $v_i^{j+r}$, where $r \in R$.
When $X\la S_i \ra$ is an independent set of vertices
we simply write $n$ inside its circle.
Finally, an arrow pointing from  the circle representing
the orbit $S_i$ to the circle representing the
orbit $S_k$, $k \ne i$, labeled by the set $T \subseteq \ZZ_n$
indicates that for each $j\in\ZZ_n$, the vertex $v_i^j\in S_i$
is adjacent to all the vertices $v_{k}^{j+t}$, where $t \in T$.
An example illustrating this notation is given in Figure~\ref{fig:ch}.

We end this subsection with two important examples of
graphs admitting semiregular automorphisms which
arise in the classification of vertex-transitive graphs
of order $pq$, see Section~\ref{sec:strategy}.
The first class consists of the so-called  metacirculant graphs,
already mentioned in the introduction as the class of
vertex-transitive graphs of order $pq$ admitting an
imprimitive subgroup of automorphisms with blocks of size $p$,
where $p$ is the largest of the two primes $p$ and $q$.
A formal definition, first given in \cite{AP82}, goes as follows.
An $(m,n)$-{\it metacirculant} is a graph of order $mn$ admitting
an $(m,n)$-semiregular automorphism $\rho$
and an automorphism $\sigma$ normalizing $\la\rho\ra$ which
cyclically permutes the orbits of $\rho$.
In particular, let the vertices of an $(m,n)$-metacirculant $X$
with respect to an $(m,n)$-semiregular automorphism $\rho$ be
denoted as in the previous paragraph where Frucht's notation is defined.
Then $X$ is uniquely
determined by the array $(m,n,\alpha,T_0,\ldots, T_\mu)$
where $\alpha\in\ZZ_n^*$, $\mu$ is the integer part of $m/2$,
the sets $T_i\subset \ZZ_n$ satisfy the following conditions
$$
0\notin T_0=-T_0, \ \alpha^mT_i=T_i \textrm{ for $0\le i\le \mu$,
and if $m$ is even, then $\alpha^\mu T_\mu=-T_\mu$},
$$
the edge set of $X$ is given by
$$
v_i^r\sim v_j^s \textrm{ if and only if } s-r\in \alpha^i T_{j-i},
$$
and the automorphism $\sigma$ of $X$ that normalizes
$\la\rho\ra$ is defined by
$\sigma(v_i^j)=v_{i+1}^{\alpha j}$,
where $i\in\ZZ_m$ and $j\in \ZZ_n$.
It can be easily seen that the semidirect
product $\la\rho\ra\rtimes\la\sigma\ra$ is a transitive subgroup of
the automorphism group of the metacirculant $X$.
For example, a metacirculant given by the array
$(2,5,2,\{\pm 1\},\{0\})$ is isomorphic to
the Petersen graph.

The second class consists of the so-called  Fermat graphs \cite{MS2},
mentioned in the introduction as the class of
vertex-transitive graphs of order $pq$ admitting an
imprimitive subgroup of automorphisms with blocks of size $q$
and no imprimitive subgroup of automorphisms with blocks of size $p$,
where $p$ is the largest of the two primes $p$ and $q$.
In particular, let $p=2^{2^s}+1$ be a Fermat prime and
let $q$ be a prime dividing $p-2$.
Let $w$ be a fixed generator of the multiplicative group
$\GF(p-1)^*=\GF(p-1)\setminus\{0\}$ of the Galois field $\GF(p-1)$.
For a symmetric subset $S$ of $\GF(q)^*=\GF(q)\setminus\{0\}$ and a
non-empty proper subset $T$ of $\GF(q)^*$ we let
the {\em Fermat graph} $F(p,q,S,T)$ be the graph
with vertex set $\PG(1,p-1)\times\GF(q)$ such that,
for each point $v$ of the projective line
$\PG(1,p-1)$ and each $r\in \GF(q)^*$,
the neighbors of $(\infty,r)$ are all the vertices of the form
\begin{eqnarray*}
 (\infty,r+s) \ (s\in S) \textrm{ and } (y,r+t) \ (y\in \GF(p-1), t\in T),
\end{eqnarray*}
and the neighbors of $(v,r)$, $v\ne \infty$,
are all the vertices of the form
\begin{eqnarray*}
 (v,r+s) \ (s\in S) \textrm{ and } (\infty,r-t) \ (t\in T)
\textrm{ and } (v+w^i,-r+t+2i) \ (i\in\GF(q), t\in T).
\end{eqnarray*}
From this list of adjacencies one can easily
see that $F(p,q,S,T)$ admits a $(p,q)$-semiregular automorphism,
and that the quotient graph with respect to this
automorphism is isomorphic to the complete
graph $K_p$. The smallest Fermat graph is the line
graph of the Petersen graph.
(Let us also mention that the class of connected Fermat graphs
is disjoint from the  class of metacirculants of order $pq$, see \cite{MS2}.)


\subsection{Some group-theoretic terminology}
\label{ssec:groups}
\noindent

For group-theoretic terms not defined here
we refer the reader to \cite{W66}.
The following classical groups appear  in the description of
primitive group actions, of degree
a product of two distinct primes,
which do not have  imprimitive subgroups
(see Table~\ref{tab:groups}):
\begin{enumerate}[(i)]
\itemsep=0pt
\item $P\Omega^\epsilon(2d,2)$:  the projective orthogonal group of a vector space of dimension $2d$ over a finite field $\GF(2)$,
\item $\PSL(n,q)$: the projective special linear group on $n$-dimensional vector space over finite field of order $q$,
\item $\PGL(n,q)$: the projective general linear group on $n$-dimensional vector space over finite field of order $q$,
\item $M_{22}$: the Mathieu group (a sporadic simple
group of order $443520$),
\item $A_n$: the alternating group of degree $n$,
\item $D_{2n}$: the dihedral group of order $2n$.
\end{enumerate}


\subsection{Useful number theory facts}
\label{ssec:numbers}
\noindent


For a prime power $r$ a finite field $\GF(r)$ of order $r$
will be denoted by $F_r$,
with the subscript $r$ being omitted whenever the order of the field
is clear from the context.
The set of nonzero
quadratic residues modulo $r$, that is, elements of $F^*=F\setminus\{0\}$
that are  congruent to a perfect square modulo $r$, will be denoted by
$S^*$. The elements of $S^*$ will be called {\em squares},
the elements of $F^*$ not belonging to $S^*$
will be called {\em non-squares},
and the set of all non-squares will be denoted by
$N^*$, that is, $N^*=F^*\setminus S^*$.

The following basic number-theoretic results
will be needed throughout the paper.

\begin{proposition}
\label{pro:-1}
{\rm \cite[Theorem~21.2]{S12}}
Let $F$ be a finite field of odd prime order $p$.
Then $-1\in S^*$ if $p\equiv {1\pmod 4}$,
and $-1\in N^*$ if $p\equiv {3\pmod 4}$.
\end{proposition}

\begin{proposition}
\label{pro:2}
{\rm \cite[Theorem~21.4]{S12}}
Let $F$ be a finite field of odd prime order $p$.
Then $2\in S^*$ if $p\equiv {1,7\pmod 8}$, and
$2\in N^*$ if $p\equiv {3,5\pmod 8}$.
\end{proposition}

\begin{proposition}
\label{pro:preseki}
{\rm \cite[p. 167]{MS3}}
Let $F$ be a finite field of odd prime order $p$. Then
$$
|S^*+1 \cap (-S^*)|=\left\{
\begin{array}{ll}
(p-5)/4 & p\equiv {1\pmod 4},\\
(p+1)/4 & p\equiv {3\pmod 4}.
\end{array}\right.
$$
In particular, if $p\equiv {1\pmod 4}$ then
$|S^*\cap S^*+1|=(p-5)/4$,
$|N^*\cap N^*+1| = (p-1)/4$, and
$|S^*\cap N^*\pm 1| = (p-1)/4$.
\end{proposition}

Using Proposition~\ref{pro:preseki} the following
result may be easily deduced.

\begin{proposition}
\label{num}
Let $F$ be a finite field of odd prime order $p$.
Then for any $k\in F^*$,
the equation $x^2+y^2=k$ has
$p- 1$ solutions if $p\equiv {1 \pmod  4}$,
and $p+ 1$ solutions if
$p\equiv  {3 \pmod  4}$.
\end{proposition}

\begin{proposition}
\label{pro:AB}
Let $F$ be a finite field of odd prime order $p\equiv {1\pmod 4}$,
and let $A=S^*\cap S^*+1$ and
$B=S^*\cap S^*-1$.
Then $|(A\setminus B) \cup (B \setminus A)|\ge 2$, that is,
$|A\cup B|\ge |A|+2$.
\end{proposition}

\begin{proof}
First observe that there must exist three consecutive elements
$s-1, s, s+1$ of the field $F$
such that $s-1$ and $s$ are squares but $s+1$ is not.
Therefore  $s\in S^*\cap S^*+1$ but $s\notin S^*\cap S^*-1$,
and so $s\in A\setminus B$. But then, since $-1\in S^*$,
we have $-s\in B\setminus A$,
and thus $|(A\setminus B) \cup (B \setminus A)|\ge 2$.
\end{proof}

\begin{proposition}
\label{pro:ABn}
Let $F$ be a finite field of odd prime order $p\equiv {1\pmod 4}$,
and let $A=S^*\cap N^*+1$ and $B=S^*\cap N^*-1$.
Then $|(A\setminus B) \cup (B \setminus A)|\ge 2$, that is,
$|A\cup B|\ge |A|+2$.
\end{proposition}

\begin{proof}
First observe that there must exist three consecutive elements
$s-1, s, s+1$ of the field $F$
such that $s-1$ and $s$ are squares but $s+1$ is not.
Therefore  $s\in S^*\cap N^*-1$ but $s\notin S^*\cap N^*+1$,
and so $s\in B\setminus A$. But then, since $-1\in S^*$
we have $-s\in A\setminus B$,
and so $|(A\setminus B) \cup (B \setminus A)|\ge 2$.
\end{proof}

\begin{proposition}
\label{pro:edges}
Let $F$ be a finite field of odd prime order
$p\equiv {1\pmod 8} $, let $i=\sqrt{-1}\in F^*$, and let
$\mathcal{T}=\{x \in F^*\colon 1+x^2\in S^*\}$.
Then for each $x\in F^*$ at least one of the three elements
$x$, $ix$ and $ix^2$ belongs to $\mathcal{T}$.
\end{proposition}

\begin{proof}
If neither $x$ nor $ix$ are contained in $\mathcal{T}$ then
both $1+x^2$ and $1+(ix)^2=1-x^2$ are non-squares, and consequently
$1-x^4=(1+x^2)(1-x^2)=1+(ix^2)^2$ is a square, and so
$ix^2$ belongs to $\mathcal{T}$.
\end{proof}


\subsection{Theorems about  existence of Hamilton cycles}
\label{ssec:ham}
\noindent

In this subsection we list three results about existence of Hamilton
cycles in particular graphs that will prove useful in the subsequent sections.
A recent result about the existence of
Hamilton paths in those generalized Petersen graphs which do
not have a Hamilton cycle
is also given as it will be needed in Subsection~\ref{subsec:89}.

\begin{proposition}
\label{pro:chvatal}
{\rm \cite{C72} (Chv\'atal's Theorem)}
Let $X$ be a graph of order $n$. For a given positive integer $i$ let
$S_i=\{ x\in V(X)\colon \mathrm{deg}(x) \le i\}$.
If for every $i < n/2$
$$
\textrm{either } |S_i|\le i-1 \ \textrm{ or } \ |S_{n-i-1}|\le n-i-1
$$
then $X$ contains a Hamilton cycle.
\end{proposition}

\begin{proposition}
\label{pro:jack}
{\rm \cite[Theorem~6]{BJ78} (Jackson's Theorem)}
Every $2$-connected regular graph of order $n$ and valency at least $n/3$
contains a Hamilton cycle.
\end{proposition}

The generalized Petersen graph $GP(n, k)$
is defined to have the vertex set
$V(GP(n,k)) = \{ u_i \colon i \in \ZZ_n \}\cup \{v_i \colon i \in \ZZ_n \}$,
and edge set
\begin{eqnarray}
\label{genpet}
  E(GP(n, k)) &=& \{u_iu_{i+1} \colon i \in \ZZ_n\}\cup
    \{ v_{i}v_{i+k}  \colon i \in \ZZ_n\}\cup
    \{  u_iv_i \colon i \in \ZZ_n\}.
\end{eqnarray}

\begin{proposition}
\label{pro:GPG}
{\rm \cite{BA83}}
The generalized Petersen graph $GP(n,k)$, $n\ge 2$ and $1\le k\le n-1$
contains a Hamilton cycle
if and only if it is neither $GP(n,2)\cong GP(n,n-2)\cong
GP(n,(n-1)/2)\cong GP(n,(n+1)/2)$ when $n\equiv {5\pmod 6}$
nor $GP(n,n/2)$ when $n\equiv {0\pmod 4}$ and $n\ge 8$.
\end{proposition}

A graph is {\em Hamilton-connected} if every pair of vertices is joined by a Hamilton path, and is  {\em Hamilton-laceable} if it is bipartite and
every pair of vertices on opposite sides of the bipartition is joined by a Hamilton path. The following result on
existence of Hamilton paths in generalized Petersen graphs
without Hamilton cycles \cite{R13}, will be needed later on.

\begin{proposition}
\label{pro:GPG2}
{\rm \cite[Theorem~2.2]{R13}}
Label the vertices of $GP(n,2)$ as in (\ref{genpet}).
Then the following hold.
\begin{enumerate}[(i)]
\itemsep=0pt
\item $GP (n, 2)$ is Hamilton-connected
if and only if $n \equiv {1,3\pmod 6}$.
\item If $n \equiv {0 \pmod 6}$ and $x$ and $y$ are distinct vertices of
$GP(n, 2)$ so that, for any $i$ and $t$, $\{x, y\}$ is
neither of the pairs $\{u_i, u_{i+2}\}$ and
$\{u_i , u_{i+6t} \}$, then there is a Hamilton path joining $x$ and $y$
in $GP (n, 2)$.
\item If $n \equiv {2\pmod 6}$ and $\{x, y\} \ne \{v_i, v_{i+4+6t} \}$,
for any $i$ and $t$, then there is a Hamilton path joining $x$ and $y$ in
$GP(n, 2)$.
\item If $n \equiv {4\pmod 6}$ and $\{x, y\}$ is none of the pairs
$\{u_i, u_{i+2}\}$, $\{u_i, v_{i\pm 1}\}$, $\{u_i, v_{i+2+6t} \}$,
$\{v_i, v_{i+4+6t} \}$, for any i and t, then there is
a Hamilton path joining x and y in $GP (n, 2)$.
\item If $n \equiv {5\pmod 6}$, and $x$ and $y$ are not
adjacent and $\{x, y\} \ne \{v_i, v_{i+3+6t} \}$ for any $i$ and
$t$, then there is a Hamilton path joining
$x$ and $y$ in $GP (n, 2)$.
\end{enumerate}
\end{proposition}


\section{Polynomials of degree $4$ over finite fields representing
quadratic residues}
\label{sec:polynomials}
\indent

In early eighties, motivated by a question posed by Alspach, Heinrich and Rosenfeld \cite{AHR80}
in the context of decompositions of complete symmetric digraphs,
Madden and V\'elez \cite{MV82} investigated polynomials that represent quadratic residues  at primitive roots. They proved that,
with finally many exceptions,
for any finite field $F$ of odd characteristic,
for every polynomial $f(x) \in F[x]$ of degree $r\ge 1$ not of the form
$\alpha g(x)^2$ or $\alpha x g(x)^2$, there exists a primitive root
$\beta$ such that $f(\beta)$ is a nonzero square in $F$.
It is the purpose of this section to refine their result for
polynomials of degree $4$.
This will then be used in Section~\ref{sec:PSL} in the constructions
of Hamilton cycles for some of the basic orbital graphs arising
from the action of $\PSL(2,p)$ on cosets of $D_{p-1}$.
This refinement, stated in the theorem below,
will be proved following a series of lemmas.

\begin{theorem}
\label{the:polynomial-degree-4}
Let $F$ be a finite field of prime order $p$, where
$p$ is an odd prime not given in Tables~\ref{tab:exceptions1}
and~\ref{tab:exceptions2}.
Then for every polynomial $f(x)\in F[x]$ of degree $4$
that has a nonzero constant term and
is not of the form $\alpha g(x)^2$
there exists a primitive root $\beta \in F$ such that $f(\beta)$ is
a square in $F$.
\end{theorem}

The following result, proved in \cite{MV82},
is a basis of our argument and will be used throughout this section.

\begin{proposition}
\label{pro:corollary1}
{\rm \cite[Corollary~1]{MV82}}
Let $F$ be a finite field with $p^n$ elements. If $s$ and $t$ are
integers such that
\begin{enumerate}[(i)]
\itemsep=0pt
\item $s$ and $t$ are coprime,
\item a prime $q$ divides $p^n-1$ if and only if $q$ divides $st$, and
\item $2\phi(t)/t>1+(rs-2)p^{n/2}/(p^n-1)+(rs+2)/(p^n-1)$,
\end{enumerate}
then, given any polynomial $f(x)\in F[x]$ of degree $r$, square-free and
with nonzero constant term, there exists a primitive root $\gamma \in F$
such that $f(\gamma)$ is a nonzero square in $F$.
\end{proposition}

Throughout this section let $p$ be an odd prime and
let $q_1=2, q_2,\ldots, q_m$ be the increasing sequence
of prime divisors of $p-1=q_1^{i_1}q_2^{i_2}\cdots q_m^{i_m}$.
As in \cite{MV82} we define the following functions with respect
to this sequence:
\begin{eqnarray}
\label{eq:1}
d(n,m)=2(1-\frac{1}{q_n})(1-\frac{1}{q_{n+1}}) \cdots (1-\frac{1}{q_m}),
\end{eqnarray}
\begin{eqnarray}
\label{eq:2}
c_r(n,m)=2r\sqrt{{(\frac{q_1q_2\cdots q_{n-1}}{q_nq_{n+1}\cdots q_m})}},
\end{eqnarray}
and $k(m)$ as the unique integer such that
$d(k(m)-1,m)\le 1<d(k(m),m)$. Hence $k(m)\ge 2$.
Analogously the functions $d$ and $c_r$ can be
defined for any positive integers $r\ge 1$,
$n<m$ and an arbitrary sequence $\{q_1,\ldots,q_m\}$ of primes.
The following lemma is a generalization of \cite[Lemma~3]{MV82}.

\begin{lemma}
\label{lem:m>}
Let $\{2=q_1,q_2,\ldots, q_m\}$ be a finite sequence of primes
satisfying $m\ge 2k(m)+2$, and let $r=4$. Then
\begin{eqnarray}
\label{eq:3}
d(k(m)+1,m)-c_r(k(m)+1,m)>1.
\end{eqnarray}
\end{lemma}

\begin{proof}
Since $2\le k(m)\le \frac m2-1$, we  have $m\ge 6.$
Let $\O $ be the increasing sequence of all prime numbers.
For a given  prime $q$, let
$
\mathcal{I}_q=\{ w_1=2, w_2, w_3, \ldots, w_{k(m)}=q, w_{k(m)+1},\ldots , w_m\}
$
be a subsequence of $\O$ not missing any prime in $\O$
from the interval $[w_2,w_m]$.
Also, let $d(n, m)'$ and $c_4(n, m)'$
be the corresponding values for $\mathcal{I}_q$
as defined by functions $d$ and $c_r$  in (\ref{eq:1}) and (\ref{eq:2}).
Further, let
$
\mathcal{J}_q=\{ q_1=2,q_2, q_3, \ldots, q_{k(m)}=q, q_{k(m)+1},\ldots , q_m\}
$
denote the subsequence of $\O$ from the statement of Lemma~\ref{lem:m>}.
Then one can easily see that
$d(k(m)+1, m)'\ge d(k(m)+1, m)$ and that $c_4(k(m)+1, m)'\le c_4(k(m)+1, m)$,
and so (\ref{eq:3}) holds for $\mathcal{J}_q$ if it holds for $\mathcal{I}_q$.
This shows that in what follows we can assume that
$\mathcal{J}_q = \mathcal{I}_q$.

Since
$$
d(k(m)+1,m)=(1+\frac 1{w_{k(m)}-1})d(k(m),m)>  1+\frac 1{w_{k(m)}-1},
$$
 (\ref{eq:3}) holds if
$$
1+\frac 1{w_{k(m)}-1}-2r{(\frac{w_1w_2\cdots w_{k(m)}}{w_{k(m)+1}w_{k(m)+2}\cdots w_m})}^{\frac 12}>1,
$$
which may be rewritten in the following form
\begin{eqnarray}
\label{eq:4}
w_2w_3\cdots w_{k(m)}(w_{k(m)}-1)^2< \frac 1{128} w_{k(m)+1}\cdots w_{m-1}w_{m},
\end{eqnarray}
in view of the fact that $r=4$ and $w_1=2$.

We divide the proof into two cases, depending on whether
$m\ge 7$ or $m=6$.

\begin{case}\end{case}
$m\ge 7$.

\medskip

\noindent
We shall in fact prove a more  general result:
\begin{eqnarray}
\label{eq:5}
w_2w_3\cdots w_{l}(w_{l}-1)^2< \frac 1{128} w_{l+1}\cdots w_{m-1}w_{m},
\end{eqnarray}
where $m\ge 7$ and $l\le \frac m2-1$ is any integer.
 If $w_m\ge 128$, then  (\ref{eq:5}) is clearly true.
So we only need to consider primes that are smaller than or equal to $127$.
If
\begin{eqnarray}
\label{eq:6}
(m-l)-(l-1+2)=m-2l-1\ge 2,
\end{eqnarray}
then  (\ref{eq:5}) holds provided  $w_{m-1}w_{m}>128$ holds.
Note that this is true if $w_m\ge 13$, which is the case since $m\ge 7$.
Next, note that for either $m$ being even and $l<\frac m2-2$  or $m$ being odd,
 (\ref{eq:6}) holds.
So we may assume that $m$ is even and that
$
l=m/2-1\ge 2.
$

Now we prove that (\ref{eq:5}) holds under this assumption
for any even integer $m\ge 8$ by  induction. Suppose first that $m=8$. Then $l=3$ and
(\ref{eq:5}) rewrites as
\begin{eqnarray}
\label{eq:7}
w_2w_3(w_3-1)^2< \frac 1{128} w_4w_{5}w_{6}w_7w_8.
\end{eqnarray}
A computer search  shows that  (\ref{eq:7}) holds
for all primes $w_8\le 127$.
Suppose now that  (\ref{eq:5}) is true for an even integer  $m\ge 8$.
Then   we have
$$\begin{array}{lll}
w_2w_3w_4\cdots w_{l}w_{l+1}(w_{l+1}-1)^2
=w_2(w_3\cdots w_{l}w_{l+1}(w_{l+1}-1)^2)
&< & w_2(w_{l+2}w_{l+3}\cdots w_mw_{m+1})\\
&< & (w_{l+2}w_{l+3}\cdots w_mw_{m+1})w_{m+2}.\end{array}$$
Therefore  (\ref{eq:5}) is true for all even integers  $m\ge 8$
and then for all integers $m\ge 7.$ Hence (\ref{eq:4}) holds,
and so does (\ref{eq:3}).

\begin{case}\end{case}
$m=6$.

\medskip

\noindent
Then  $k(m)=2$, and so (\ref{eq:5})  becomes
\begin{eqnarray}
\label{eq:8}
w_2(w_2-1)^2< \frac 1{128} w_3w_4w_{5}w_{6}.
\end{eqnarray}
A computer search shows that (\ref{eq:8}) does not
hold only  for
$$
w_{k(m)}=w_2\in \{11,13,17,19,23,29,31,37,41,43,53,59,61,67,71\}.
$$
An additional computer search shows that for $w_1=2$
(\ref{eq:3}) holds in each of these exceptional cases.
This completes the proof of Lemma~\ref{lem:m>}.
\end{proof}

\setcounter{case}{0}

The following result proved in \cite{MV82} will be needed
in the next lemma.

\begin{proposition}
\label{pro:lemma5}
{\rm \cite[Lemma~5]{MV82}}
Let $\{2=q_1,q_2,\ldots, q_m\}$ be a finite sequence of primes
satisfying $m\le 2k(m)+1$. Then $m\le 9$ and $q_{k(m)-1}\le 5$.
In fact the sequence must satisfy one of the following:
\begin{enumerate}[(i)]
\itemsep=0pt
\item $k(m)=4$, $q_{k(m)-1}=5$ and $m=9$,
\item $k(m)=3$, $q_{k(m)-1}=5$ and $m\le 7$,
\item $k(m)=3$, $q_{k(m)-1}=3$ and $m\le 7$, or
\item $k(m)=2$, $q_{k(m)-1}=2$ and $m\le 5$.
\end{enumerate}
\end{proposition}

\begin{lemma}
\label{lem:qm>128}
Let $\{2=q_1,q_2,\ldots, q_m\}$ be a finite sequence of primes
satisfying $m\le 2k(m)+1$, and let
$p-1=q_1^{i_1}q_2^{i_2}\cdots q_m^{i_m}$ with $q_m\ge 131$.
Then there exist $s$ and $t$ such that
\begin{enumerate}[(i)]
\itemsep=0pt
\item $s$ and $t$ are coprime,
\item a prime $q$ divides $p-1$ if and only if $q$ divides $st$, and
\item $2\phi(t)/t>1+(4s-2)\sqrt{p}/(p-1)+(4s+2)/(p-1)$.
\end{enumerate}
\end{lemma}

\begin{proof}
Since $m\le 2k(m)+1$ the four cases (i) - (iv) of Proposition~\ref{pro:lemma5}
need to be considered. In each case, as in \cite[Lemma 7]{MV82},
we will prescribe a choice for $s$ (which then determines $t$ uniquely)
and use the conditions in each
of these four cases to find
the lower bound $\alpha$ for the expression $(2\phi(t)t^{-1}-1)$,
that is, $(2\phi(t)t^{-1}-1)\ge \alpha$. We will then
be able to use the assumption $q_m\ge 131$ to show that
\begin{eqnarray}
\label{eq:alpha}
\alpha > \frac{(4s-2)\sqrt{p}+4s+2}{p-1}.
\end{eqnarray}
Suppose first that Proposition~\ref{pro:lemma5}(i) holds,
that is, $k(m)=4$, $q_{k(m)-1}=5$ and $m=9$.
Then $q_9\ge 131$.  Also, one can easily see that
such a sequence of primes must begin with
$q_1=2$, $q_2=3$ and $q_3=5$.
Let $s=2\cdot 3\cdot 5$ and $t=q_4q_5\cdots q_9$.
Then
$$
2\frac{\phi(t)}{t}-1\ge
2(1-\frac{1}{7})(1-\frac{1}{11})(1-\frac{1}{13})(1-\frac{1}{17})
(1-\frac{1}{19})(1-\frac{1}{131}) -1 \ge 0.27287.
$$
Thus $p$ satisfies  (\ref{eq:alpha}) with $\alpha=0.27287$
and $s=30$ if and only if $p> 187899$.
Suppose now that there is a prime $p\le  187899$
that satisfies the conditions of the case under analysis.
We know that $2\cdot 3\cdot 5\cdot q_9$ divides $p-1$ with
$q_9\ge 131$. However this requires
$q_4q_5q_6q_7q_8 < 187899/(2\cdot 3\cdot 5\cdot 131)\le 48$ which is clearly not possible, since $q_4q_5q_6q_7q_8\ge 7\cdot 11 \cdot 13 \cdot 17\cdot 19=323323$.

We now consider the other three cases of
Proposition~\ref{pro:lemma5}, that is, suppose that
Proposition~\ref{pro:lemma5}(ii), (iii) or (iv) holds.
In all three cases $k(m)\le 3$.
Since $p$ is an odd prime we know $q_1=2$, and we now consider the
various possibilities for $q_2$.
First, assume that $q_2=3$ (note that this is possible in the last two cases)
and therefore $m\le 7$. We set $s=2\cdot 3$ and $t=q_3q_4q_5q_6q_7$.
Thus
$$
2\frac{\phi(t)}{t}-1\ge
2(1-\frac{1}{5})(1-\frac{1}{7})(1-\frac{1}{11})(1-\frac{1}{13})(1-\frac{1}{131}) -1 \ge 0.14206.
$$
Now $p$ satisfies (\ref{eq:alpha}) with $\alpha=0.14206$
and $s=6$ if and only if $p\ge 24351$. If $p< 24351$
we see that $q_3q_4\cdots q_{m-1}<24351/(2\cdot 3\cdot 131)< 31$.
Since $q_i\ge 5$ for $i\in\{3,4,\ldots, m-1\}$  one can see that
either $m=3$ or $m=4$. In other words, either $t=q_3$ or $t=q_3q_4$, and thus
we can improve the value for $\alpha$ with
$$
2\frac{\phi(t)}{t}-1\ge
2(1-\frac{1}{5})(1-\frac{1}{131}) -1 \ge 0.58778.
$$
In this case $p$ satisfies (\ref{eq:alpha})
with $\alpha=0.58778$ if and only if $p> 1490$.
If $p\le 1490$ observe that the assumption that
$6q_m$ divides $p-1$ with $q_m\ge 131$ implies that $q_3<2$,
a contradiction.

We now use the same approach for the case $q_2=5$.
We choose $s=2\cdot 5$ and $t=q_3q_4\cdots q_m$.
Here we have
$$
2\frac{\phi(t)}{t}-1\ge
2(1-\frac{1}{7})(1-\frac{1}{11})(1-\frac{1}{13})(1-\frac{1}{17})(1-\frac{1}{131}) -1 \ge 0.34361.
$$
Hence $p$ satisfies (\ref{eq:alpha})
with $\alpha=0.34361$
if and only if
$p> 12475$. If, however, $p\le 12475$ then since $10q_m$ divides
$p-1$ we have that $q_3 < 10$, and so either $m=4$ and $q_3=7$
or $m=3$. In both cases we can improve the value for
$\alpha$ since $t=q_3$ or $t=q_3q_4$. In particular,
$$
2\frac{\phi(t)}{t}-1\ge
2(1-\frac{1}{7})(1-\frac{1}{131}) -1 \ge 0.70119956.
$$
\medskip
In this case $p$ satisfies (\ref{eq:alpha}) with $\alpha=0.70119956$
if and only if $p> 3057$.
If $p\le 3057$ observe that the assumption that
$10q_m$ divides $p-1$ with $q_m\ge 131$ implies that $q_3<2$,
a contradiction.

Finally we consider the case $q_2\ge 7$.
Then, by Proposition~\ref{pro:lemma5}, we have $k(m)=2$ and $m\le 5$.
Here we choose $s=2$ and use the same technique as above to complete
the proof. In particular, we have
$$
2\frac{\phi(t)}{t}-1\ge
2(1-\frac{1}{7})(1-\frac{1}{11})(1-\frac{1}{13})(1-\frac{1}{131}) -1 \ge 0.42758.
$$
In this case $p$ satisfies (\ref{eq:alpha}) with $\alpha=0.42758$
if and only if $p> 243$.
If $p\le 243$ observe that the assumption that
$2q_m$ divides $p-1$ with $q_m\ge 131$ implies that $q_3<2$,
a contradiction.

In summary we have seen that given any finite sequence of primes with
$q_m\ge 131$ we can choose $n$ in such a way that
when $s=q_1q_2\cdots q_n$ and $t=q_{n+1}q_{n+2}\cdots q_m$ we have
\begin{eqnarray}
\label{eq:alpha1}
\frac{2\phi(t)}{t} > 1 + \frac{(4s-2)\sqrt{st+1}}{st}+\frac{4s+2}{st},
\end{eqnarray}
completing the proof of Lemma~\ref{lem:qm>128}.
\end{proof}

In order to proceed with the proof of
Theorem~\ref{the:polynomial-degree-4} we now need to
identify all those sequences $\{2=q_1,q_2,\ldots,q_m\}$
with $q_m<131$ for which one cannot choose $s=q_1q_2\cdots q_n$
and $t=q_{n+1}q_{n+2}\cdots q_m$ so as to satisfy (\ref{eq:alpha1}).
Since Lemma~\ref{lem:m>} holds for each $q_m$
we can assume that for each of these sequences
Proposition~\ref{pro:lemma5} applies.
A computer search of these finitely many sequences yields the
exceptional sequences
which are listed in Tables~\ref{tab:exceptions1} and \ref{tab:exceptions2}.
For each of these exceptional sequences we fix $s=q_1q_2\cdots q_n$
and $t=q_{n+1}q_{n+2}\cdots q_m$, and we then search for a constant $k$
such that  $x>k$ implies the inequality
\begin{eqnarray}
\label{eq:k}
\frac{2\phi(t)}{t}>1+\frac{2(2s-1)\sqrt{x}}{x-1}+\frac{4s+2}{x-1}.
\end{eqnarray}
For each of these sequences Tables~\ref{tab:exceptions1} and \ref{tab:exceptions2}  give the smallest bound $k$ obtained in this way.
The third column of these tables indicates for which choice of $t$ the given
bound $k$ is obtained: Type 1 means that the bound $k$ was obtained
with $t=q_{m-1}q_m$,
Type 2 means that the bound was obtained with $t=q_m$, and
Type 3 means that the bound was obtained with $t=1$.
A computer search then identifies those primes that are
smaller than or equal to the bound $k$, as summarized
in the proposition below.

\begin{proposition}
\label{pro:qm<128}
Let $\{2=q_1,q_2,\ldots, q_m\}$ be a finite sequence of primes
satisfying $m\le 2k(m)+1$, and let
$p-1=q_1^{i_1}q_2^{i_2}\cdots q_m^{i_m}$ with $q_m< 131$.
If $p$ is not listed in Tables~\ref{tab:exceptions1} and \ref{tab:exceptions2}
then there exist $s$ and $t$ such that
\begin{enumerate}[(i)]
\itemsep=0pt
\item $s$ and $t$ are coprime,
\item a prime $q$ divides $p-1$ if and only if $q$ divides $st$, and
\item $2\phi(t)/t>1+(4s-2)\sqrt{p}/(p-1)+(4s+2)/(p-1)$.
\end{enumerate}
\end{proposition}

{\small
\begin{table}[h!]
$$
\begin{array}{|c|c|c|c|c|}
\hline
\textrm{Sequence $\cal{T}$} & k  & \textrm{ Type } & p \le k & p\equiv{1\pmod 4}\le k\\
& & &\textrm{ with $\cal{T}$  }& \textrm{ with } {\cal T} ,  (p+1)/2 \textrm{ prime} \\
\hline
\hline
2  & 55 & 3 & 3,5,17& 5\\ \hline
2, 3, 5, 11
& 2458 & 1 &  331, 661, 991, 1321
& 661, 1321
\\ \hline
2, 3, 5, 43
& 1622 & 1 &  1291
 & no
\\ \hline
2, 3, 7, 17
& 1372 & 1  & no
 & no
\\ \hline
2, 3, 5, 7, 13
& 7040 & t=455 &  2731
& no
\\ \hline
2, 3, 43
& 460 & 1  & no
 & no
\\ \hline
2, 3, 31
& 496 & 1 &  373
 & no
\\ \hline
2, 3, 61
& 435 & 1 &  367
 & no
\\ \hline
2, 3, 5, 7, 23
& 5145 & t=805 &  4831
& no
\\ \hline
2, 3, 23
& 547 & 1 &  139, 277
& 277
\\ \hline
2, 3, 67
& 430 & 1  & no
 & no
\\ \hline
2, 3, 7, 13
& 1517 & 1 &  547, 1093
& 1093
\\ \hline
2, 3, 17
& 632 & 1 &  103, 307, 409, 613
& 613
\\ \hline
2, 3, 5, 13
& 2238 & 1 &  1171, 1951
 & no
\\ \hline
2, 3, 11
& 788 & 2 &  67, 199, 397, 727
& 397
\\ \hline
2, 7
& 99 & 2 &  29
 & no
\\ \hline
2, 3, 13
& 739 & 2 &  79, 157, 313
& 157, 313
\\ \hline
2, 3, 7
& 1023 & 2 &  43, 127, 337, 379, 673, 757, 883, 1009
& 673, 757
\\ \hline
2, 23
& 65 & 2 &  47
 & no
\\ \hline
2, 3, 5, 37
& 1656 & 1  & no
 & no
\\ \hline
2, 5
& 133 & 2 &  11, 41, 101
 & no
\\ \hline
2, 3, 5, 41
& 1632 & 1 &  1231
 & no
\\ \hline
2, 3, 59
& 437 & 1  & no
 & no
\\ \hline
2, 3, 53
& 444 & 1  & no
 & no
\\ \hline
2, 3, 7, 19
& 1327 & 1 & no
 & no
\\ \hline
2, 3, 5, 29
& 1727 & 1  & no
 & no
\\ \hline
2, 17
& 69 & 2 & no
 & no
\\ \hline
2, 11
& 78 & 2 &  23
 & no
\\ \hline
2, 3, 5, 19
& 1921 & 1 &  571
 & no
\\ \hline
2, 3, 41
& 464 & 1  & no
 & no
\\ \hline
\end{array}
$$
\caption{\label{tab:exceptions1} The list of sequences not satisfying (\ref{eq:alpha1}) I.
}
\end{table}
}
{\small
\begin{table}[h!]
$$
\begin{array}{|c|c|c|c|c|}
\hline
\textrm{Sequence $\cal{T}$} & k  & \textrm{ Type } & p \le k & p\equiv{1\pmod 4}\le k\\
& & &\textrm{ with $\cal{T}$  }& \textrm{ with } {\cal T} ,  (p+1)/2 \textrm{ prime} \\
\hline
\hline
2, 3, 5, 7, 11
& 8160 & t=385 &  2311, 4621
& 4621
\\ \hline
2, 3, 5
& 1432 & 2 &  31, 61, 151, 181, 241, 271,
& 61, 541, 1201
\\
 & & & 541, 601, 751, 811, 1201 &
\\ \hline
2, 3, 5, 47
& 1604 & 1  & no
 & no
\\ \hline
2, 3, 5, 31
& 1705 & 1  & no
 & no
\\ \hline
2, 3, 7, 23
& 1265 & 1 &  967
 & no
\\ \hline
2, 5, 17
& 180 & 1  & no
 & no
\\ \hline
2, 3, 11, 13
& 1130 & 1 &  859
 & no
\\ \hline
2, 13
& 74 & 2 &  53
 & no
\\ \hline
2, 5, 11
& 218 & 1  & no
 & no
\\ \hline
2, 5, 13
& 200 & 1 &  131
 & no
\\ \hline
2, 3, 37
& 475 & 1 &  223
 & no
\\ \hline
2, 3, 5, 7
& 3649 & 1 &  211, 421, 631, 1051, 1471, 2521, 3361
& 421
\\ \hline
2, 3, 5, 7, 19
& 36145 & 1 &  11971, 35911
 & no
\\ \hline
2, 3
& 384 & 2 &  7, 13, 19, 37, 73, 97, 109, 163, 193
& 13, 37, 73, 193
\\ \hline
2, 5, 7
& 315 & 1 &  71, 281
 & no
\\ \hline
2, 3, 5, 23
& 1819 & 1 &  691, 1381
& 1381
\\ \hline
2, 3, 47
& 453 & 1 &  283
 & no
\\ \hline
2, 3, 5, 7, 17
& 37400 & 1 &  3571, 10711, 14281, 17851
 & no
\\ \hline
2, 3, 29
& 506 & 1 &  349
 & no
\\ \hline
2, 3, 7, 11
& 1646 & 1 &  463
 & no
\\ \hline
2, 3, 5, 17
& 1995 & 1 &  1021, 1531
 & no
\\ \hline
2, 29
& 63 & 2 &  59
 & no
\\ \hline
2, 3, 19
& 596 & 1 &  229, 457
& 457
\\ \hline
2, 19
& 68 & 2 & no
 & no
\\ \hline
\end{array}
$$
\caption{\label{tab:exceptions2} The list of sequences not satisfying (\ref{eq:alpha1}) II.
}
\end{table}
}

We are now ready to prove Theorem~\ref{the:polynomial-degree-4}.

\medskip

\begin{proofT}
It follows by Proposition~\ref{pro:corollary1} that a
polynomial $f(x)$ represents
a nonzero square at some primitive root in $F$
if there exist $s$ and $t$ satisfying the following three conditions:
\begin{enumerate}[(i)]
\itemsep=0pt
\item $s$ and $t$ are coprime,
\item a prime $q$ divides $p-1$ if and only if $q$ divides $st$, and
\item $2\phi(t)/t>1+(4s-2)\sqrt{p}/(p-1)+(4s+2)/(p-1)$.
\end{enumerate}
Our goal is therefore to
show that such $s$ and $t$ exist for  all odd primes $p$ that are not listed in
Tables~\ref{tab:exceptions1} and~\ref{tab:exceptions2}.

Let $\{q_1=2,q_2,\ldots, q_m\}$ be an increasing sequence of prime
divisors of $p-1$.
If $m\le 2k(m)+1$
then Lemma~\ref{lem:qm>128} applies for $q_m\ge 131$, and
 Proposition~\ref{pro:qm<128} applies for $q_m < 131$.

Suppose now that $m\ge 2k(m)+2$. Then, by Lemma~\ref{lem:m>}, we have
$$
d(k(m)+1,m) > 1 + c_4(k(m)+1,m).
$$
If we let $s=q_1q_2\cdots q_{k(m)}$ and $t=q_{k(m)+1}\cdots q_m$
we have
$2\phi(t)/t=d(k(m)+1,m)$, and
\begin{eqnarray*}
c_4(k(m)+1,m)&=&8\cdot \sqrt{\frac{q_1q_2\cdots q_{k(m)}}{q_{k(m)+1}q_{k(m)+2}\cdots q_m}}\\
&=&\frac{8s}{\sqrt{q_1q_2\cdots q_m}}\ge \frac{8s}{\sqrt{p-1}}
\end{eqnarray*}
Since $s$ is even and $4(p-1)\ge 4s\ge 3$ we may apply
\cite[Lemma~6]{MV82} to see that
$$
\frac{(4s-2)\sqrt{p}}{p-1}\le \frac{4s}{\sqrt{p-1}}.
$$
It follows that
\begin{eqnarray*}
\frac{2\phi(t)}{t}&=&d(k(m)+1,m)\ge 1+c_4(k(m)+1,m)\ge 1+\frac{8s}{\sqrt{p-1}}\\
&\ge &1+\frac{(4s-2)\sqrt{p}}{p-1}+\frac{4s+2}{p-1}.
\end{eqnarray*}
(Note that the last inequality hold since $p\ge 7$.)
\end{proofT}

In the search of Hamilton cycles in graphs arising from the action
of $\PSL(2,p)$ on the cosets of $D_{p-1}$
the following result about particular polynomials over finite fields
of prime order $p$, where $p$ is one of the primes listed in the
last column of Tables~\ref{tab:exceptions1} and \ref{tab:exceptions2},
obtained with a computer search,
will be needed.

\begin{proposition}
\label{pro:small}
Let $F$ be a finite field of odd prime order $p$, and let $k\in F$.
If
\begin{eqnarray*}
p&\in&\{5,13,37,61,73,157,193,277,313,397,421,457,541,\\
&& 613,661,673,757,1093,1201,1321,1381,4621\}
\end{eqnarray*}
then there exists a primitive root $\beta$ of $F$ such that
$f(\beta)=\beta^4+k\beta^2+1$ is a square in $F$ except
when
\begin{eqnarray*}
(p,k)&\in&\{(5,4), (13,1),(13,4),(13,5),(13,6),(13,7),(13,10),\\
&&(37,3),
(37,28), (37,29),(61,18),(61,37),(61,40)\}.
\end{eqnarray*}
Amongst these exceptions only  for
$(p,k) \in \{(13,1),(37,28),(61,18)\}$
there exists $\xi\in S^*\cap S^*+1$
such that $k=2(1-2\xi)$. In particular,
$\xi=10$ for $(p,k)=(13,1)$,
$\xi=12$ for $(p,k)=(37,28)$,
and $\xi=57$ for $(p,k)=(61,18)$.
Moreover,
amongst these exceptions only  for
$(p,k) \in \{(13,1),(37,28),(61,18)\}$
there exists $\bar{\xi}\in S^*\cap S^*+1$
such that $k=- 2(1-2\bar{\xi})$. In particular,
$\bar{\xi}=4$ for $(p,k)=(13,1)$,
$\bar{\xi}=26$ for $(p,k)=(37,28)$,
and $\bar{\xi}=5$ for $(p,k)=(61,18)$.

\end{proposition}



\section{Vertex-transitive graphs of order $pq$: explaining the strategy}
\label{sec:strategy}
\noindent

The goal of this paper is to prove that the Petersen
graph is the only connected vertex-transitive graph of order a product of two primes without a Hamilton cycle.
Recall that vertex-transitive graphs of prime-squared order are necessarily Cayley graphs of abelian groups.
The existence of Hamilton cycles in such graphs was proved
by the third author in 1983  \cite{DM83}.
If one of the two primes is equal to $2$
then the graphs are of order twice a prime, and the existence
of Hamilton cycles in such graphs
(with the exception of the Petersen graph)
was proved by Alspach back in 1979 \cite{A79}.

\begin{proposition}
\label{thm:p2,2p}
{\rm \cite{A79,DM83}}
Let $p$ be a prime.
With the exception of the Petersen graph,
every connected vertex-transitive graph
of order $qp$, where
$q\in\{2,p\}$, contains a
Hamilton cycle.
\end{proposition}

In pursuing our goal we
therefore only need to consider vertex-transitive graphs
whose order is a product of two different odd primes $p$
and $q$ ($p>q$).
As mentioned in the introduction there are three disjoint classes
of such graphs.
Recall that the first class consists of graphs admitting an imprimitive
subgroup of automorphisms with blocks of size $p$
- it coincides with the class of $(q,p)$-metacirculants defined in Subsection~\ref{ssec:semiregular}.
The second class consists of graphs admitting an imprimitive
subgroup of automorphisms with blocks of size $q$ and no
imprimitive subgroup of automorphisms with blocks of size $p$ - it coincides with the class of
Fermat graphs defined in Subsection~\ref{ssec:semiregular}.
Finally, the third class consists of vertex-transitive
graphs with no imprimitive subgroup of automorphisms.
Following \cite[Theorem~2.1]{MS5} the theorem below
gives a complete classification
of connected vertex-transitive graphs of order $pq$.
We would like to remark, however, that there is an additional
family of primitive graphs of order $91=7\cdot 13$ that was not
covered neither in \cite{MS5} nor in \cite{PX}.
This is due to a missing case in Liebeck - Saxl's table
\cite{LS85} of primitive group actions of degree $mp$, $m<p$.
This missing case consists of primitive groups
of degree $91=7\cdot 13$
with socle $\PSL(2,13)$ acting on cosets of $A_4$.
In the classification theorem below this missing case is included in Row 7 of Table~\ref{tab:groups}.

\begin{theorem}
\label{the:main1}
{\rm \cite[Theorem~2.1]{MS5}}
A connected vertex-transitive graph of order $pq$, where
$p$ and $q$ are odd primes and $p>q$,
must be one of the following:
\begin{enumerate}[(i)]
\itemsep=0pt
\item a metacirculant,
\item a Fermat graph,
\item a generalized orbital graph
associated with one of the groups in Table~\ref{tab:groups}.
\end{enumerate}
\end{theorem}

\begin{table}[h!]
\begin{center}
\begin{tabular}{|c|c||c|c|c|}
\hline
 &      &               &                   &                      \\
Row & $soc\ G$   &   $(p,q)$   &    $action$    &   $comment$     \\
 &      &               &                   &               \\ \hline\hline
1 & $P\Omega^\epsilon(2d,2)$  &  $(2^d-\epsilon,2^{d-1}+\epsilon)$
& singular &  $\epsilon=+1:$ $d$ Fermat prime \\
& & & $1$-spaces & $\epsilon=-1:$ $d-1$ Mersenne prime\\ \hline
2 & $M_{22}$  &  $(11,7)$  & see Atlas &        \\ \hline
3 & $A_7$   &  $(7,5)$  &  triples &      \\ \hline
4 & $\PSL(2,61)$ & $(61,31)$ & cosets of  &   \\
&            &      &        $A_5$              &   \\ \hline
5 & $\PSL(2,q^2)$ & $(\frac{q^2+1}{2},q)$ & cosets of & $q \geq 5$  \\
&            &      &      $\PGL(2,q)$ &   \\ \hline
6 & $\PSL(2,p)$  &  $(p,\frac{p+1}{2})$ & cosets of  & $p \equiv1\,(mod\,4)$    \\
&            &      &         $D_{p-1}$    & $p \geq 13$   \\ \hline
7 & $\PSL(2,13)$ & $(13,7)$ & cosets of & missing in \cite{LS85}\\
&  & $ $ &  $A_4$ & \\ \hline
\hline
\end{tabular}
\caption{\small \label{tab:groups} Primitive groups of degree $pq$ without imprimitive
subgroups and with non-isomorphic genera\-lized orbital graphs.}
\end{center}
\end{table}

The existence of Hamilton cycles in graphs given in
Theorem~\ref{the:main1}(i) and (ii) was proved,
respectively, in \cite{AP82a} and \cite{DM92}.
For the sake of completeness, let us briefly explain
the corresponding construction methods.

In \cite{AP82a} the existence of Hamilton cycles was
proved for all $(m,n)$-metacirculants with $m$ odd,
and not only for metacirculants of order
a product of two odd primes.
Let, however, $X$ be a metacirculant defined by
the array
$(q,p,\alpha,T_0,\ldots,T_\mu)$ as in Subsection~\ref{ssec:semiregular}.
If $\gcd(c,q)=1$, where $c=a/\gcd(a,q)$
and $a$ is the order of $\alpha\in \ZZ_p^*$,
then it follows by \cite{AP82} that $X$
is a Cayley graph of
the group $\la\rho,\sigma^c\ra=\la\rho\ra\rtimes\la\sigma^c\ra$.
Thus in this case the result about existence
of Hamilton cycles in Cayley graphs of semidirect products of a prime order cyclic group with an
abelian group, proved in \cite{D83,DM83}, can be applied.
When $\gcd(c,q)\ne 1$ one can use the fact that
the quotient $X_\P$ with respect to
the set of orbits $\P$ of $\rho$,  is a circulant
of order $q$   with
symbol $\{\pm i \colon 1\le \mu, \textrm{ and } T_i\ne \emptyset\}$.
A Hamilton cycle in $X$ is then
constructed as a lift of a particular
Hamilton cycle in $X_\P$
(see \cite{AP82a} for details).

Hamilton cycles in Fermat graphs were constructed in \cite{DM92}
in the following way. Let $X=F(p,q,S,T)$ be a Fermat graph
with $p=2^{2^s}+1$ being a Fermat prime and $q$
a prime dividing $p-2$ (see Subsection~\ref{ssec:semiregular}
for exact definition).
Since $p-2=2^{2^s}-1$ is divisible by $3$ this number
is a prime if and only if $(p,q)=(5,3)$, implying
that $X$ is isomorphic to the line graph of the Petersen
graph which clearly admits a Hamilton cycle.
It can therefore be assumed that $q< p-2$.
The quotient graph
$X_\mathcal{Q}$ with respect to
the set of orbits $\mathcal{Q}$
of a $(p,q)$-semiregular automorphism in $X$ is
isomorphic to the complete graph $K_p$.
Then in $X_\mathcal{Q}$
 a particular $(p-1)$-cycle which leaves out the
vertex $\{(\infty,r)\colon r\in F_q\}$ is
constructed in such a way that it lifts
to a cycle $C$ in the original graph $X$ leaving out
only vertices of the form $(\infty,r)$, $r\in F_q$.
One can then show that each of these
missing vertices is adjacent to two
neighboring vertices on $C$.
The cycle $C$ can therefore be extended
to a Hamilton cycle in $X$
(see \cite{DM92} for details).

Combining  the above results from \cite{AP82a,DM92}
with Proposition~\ref{thm:p2,2p}  we have a complete
solution of the hamiltonicity problem for vertex-transitive
graphs of order a product of two primes in the imprimitive case.

\begin{proposition}
\label{thm:ham-imprimitive}
With the exception of the Petersen graph,
every connected vertex-transitive graph
of order $qp$, where
$p$ and $q$ are primes, with an imprimitive subgroup of automorphisms
contains a  Hamilton cycle.
\end{proposition}

In view of Proposition~\ref{thm:ham-imprimitive} it follows
that Theorem~\ref{the:main} will be proved and our goal
will be achieved if we manage to show that every primitive graph
of order $pq$ contains a Hamilton cycle.
More precisely, we need to show that graphs
arising from primitive group actions given in Table~\ref{tab:groups}
have a Hamilton cycle.
The existence of Hamilton cycles needs to be proved
for all connected generalized orbital graphs
arising from these actions (see Subsection~\ref{ssec:orbital}).
Recall that a generalized orbital graph is a union of
basic orbital graphs. Since the above actions are primitive
and hence the corresponding basic orbital graphs are connected,
it suffices to prove the existence of Hamilton cycles solely in
basic orbital graphs of these actions.
This is done in the next three sections.

Graphs arising from the actions in the first four rows and the last row of
Table~\ref{tab:groups} are considered in Section~\ref{sec:small}.
Existence of Hamilton cycles in these graphs is proved
using Jackson's theorem
(see Proposition~\ref{pro:jack}) and in some cases combined also with
an ad hoc computer based search.

Graphs arising from the actions in Rows 5 and 6 are considered, respectively,
in Sections~\ref{sec:PGL} and \ref{sec:PSL}.
The method used in the proof of existence of Hamilton cycles
in these graphs is for the most part based on the
so-called lifting cycle technique \cite{A89,KM09,DM83}.
Lifts of Hamilton cycles from quotient graphs
which themselves have a Hamilton cycle
are always possible, for example,  when the quotienting is done relative to
a semiregular automorphism of prime order and when the corresponding quotient
multigraph has two adjacent orbits joined
by a  double edge contained in
a Hamilton cycle. This double edge gives us the possibility to conveniently
``change direction" so as to get a walk in the quotient that lifts to a full cycle above (see Example~\ref{ex:lifting}).
We remark that by \cite{M81} a vertex-transitive graph of order $pq$,
$q<p$ primes, contains a $(q,p)$-semiregular automorphism.
Consequently the lifting cycle technique can be
applied to graphs arising from Rows 5 and 6 of
Table~\ref{tab:groups} provided appropriate Hamilton cycles
can be found in the corresponding quotients.
Note, however, that in graphs arising from Row 5 of Table~\ref{tab:groups}
the quotienting is done with respect to a $(p,q)$-semiregular
automorphism (see Section~\ref{sec:PGL}).
As one would expect it is precisely
the existence of such Hamilton cycles in the quotients that represent
the hardest obstacle one needs to overcome in order to
assure the existence of  Hamilton cycles in the graphs in question.
In this respect many specific tools will be applied, most notably the
classical Chv\'atal's theorem \cite{C72}, and results on polynomials from Section~\ref{sec:polynomials}.

In the example below we illustrate this method on one of the basic
orbital graphs of valency $6$
arising from Row 6 of Table~\ref{tab:groups}.

\begin{example}
\label{ex:lifting}
{\rm \cite[Example~3.2]{KM09}
The orbital graph $X$ arising from the action
of $\textrm{PSL}(2,13)$ on cosets of $D_{12}$
with respect to a union ($\S_7^+\cup S_7^-$ from Example~\ref{ex:smallest}) of a suborbit
of size $3$ and its paired suborbit
contains a $(7,13)$-semiregular automorphism $\rho$,
and it can be nicely represented
in Frucht's notation as shown in Figure~\ref{fig:ch}.
Since the quotient graph $X_\rho$ has a Hamilton cycle
containing a double edge and since $13$ is a prime,
this cycle lifts to a Hamilton
cycle in the original graph $X$ (see Figure~\ref{fig:ch}).
}
\end{example}

\begin{figure}[h!]
\begin{footnotesize}
\begin{center}
 \includegraphics[width=0.50\hsize]{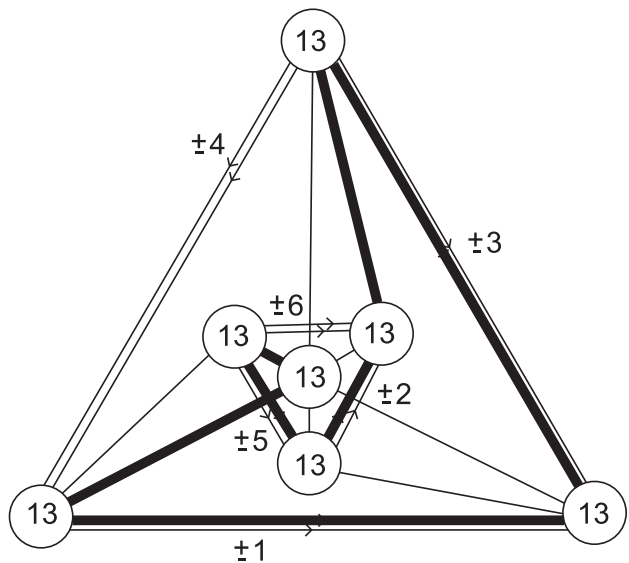}
\caption{\label{fig:ch}\footnotesize
A vertex-transitive graph arising from the
action of $\textrm{PSL}(2,13)$ on cosets of $D_{12}$
given in Frucht's notation with respect to
the $(7,13)$-semiregular automorphism $\rho$ where undirected lines carry label $0$.
Edges in bold show a Hamilton cycle.}
\end{center}
\end{footnotesize}
\end{figure}

\begin{example}
\label{ex:O4}
{\rm
The so-called odd graph $O_4$
is a basic orbital graph of valency $4$ arising from the action
of the alternating group $A_7$ acting on $3$-subsets
of $\{1,\ldots, 7\}$ where two $3$-subsets are adjacent
if they are disjoint. The odd graph $O_4$ has a $(5,7)$-semiregular
automorphism as well as $(7,5)$-semiregular automorphism.
The later is used to construct a Hamilton cycle using
the lifting cycle technique. The corresponding quotient multigraph
contains a $7$-cycle with a double edge (see Figure~\ref{fig:O4}),
and thus $O_4$ is hamiltonian.
In the below matrix we also give the symbol of $O_4$ with respect to
the orbits $S_i=\{v_i^j\colon j\in\mathbb{Z}_{5}\}$, $i\in\mathbb{Z}_7$,
of a $(7,5)$-semiregular automorphism:
$$
\left[
\begin{array}{ccccccc}
\emptyset &  \{0\} & \emptyset & \emptyset & \{0\} & \emptyset & \{0,4\}\\
\{0\} & \emptyset  & \{0,4\} & \emptyset  & \emptyset  & \emptyset & \{2\}\\
\emptyset & \{0,1\} & \emptyset  & \{0,3\} & \emptyset  & \emptyset  & \emptyset \\
\emptyset & \emptyset & \{0,2\} & \emptyset & \{4\} & \{0\} &\emptyset\\
\{0\} & \emptyset & \emptyset & \{1\} & \emptyset & \{0,2\} & \emptyset\\
\emptyset & \emptyset & \emptyset & \{0\} & \{0,3\} & \emptyset & \{1\} \\
\{0,1\} & \{3\} & \emptyset & \emptyset & \emptyset & \{4\} & \emptyset\\
\end{array}\right].
$$
}
\end{example}

\begin{figure}[h!]
\begin{footnotesize}
\begin{center}
 \includegraphics[width=0.45\hsize]{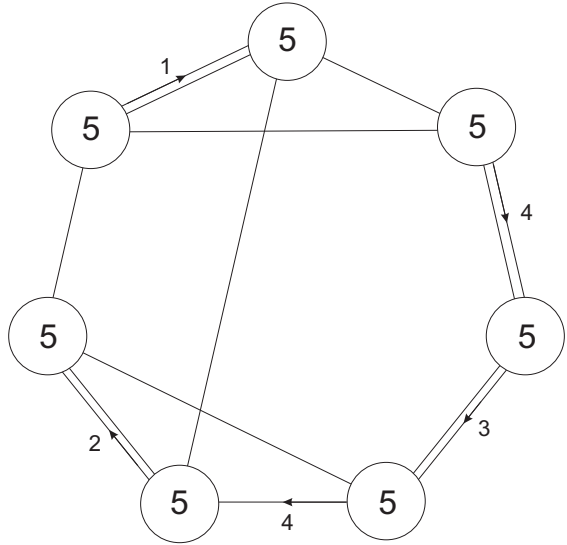}
\caption{\label{fig:O4}\footnotesize
The odd graph $O_4$ given in Frucht's notation with respect to
a $(7,5)$-semiregular automorphism $\rho$ where
undirected lines carry label $0$.}
\end{center}
\end{footnotesize}
\end{figure}

It will be useful to introduce the following terminology.
Let $X$ be a graph that admits an $(m,n)$-semiregular automorphism $\rho$. Let $\P = \{S_1, S_2, \ldots , S_m\}$
be the set of orbits of $\rho$, and let $\pi : X \to X_\P$
be the corresponding projection of $X$ to its quotient $X_\P$.
For a (possibly closed) path
$W = S_{i_1}S_{i_2}\ldots S_{i_k}$ in $X_\P$  we let
the {\em lift} of $W$ be
the set of all paths in $X$ that project to $W$.
The proof of following lemma is straightforward and is
just a reformulation of \cite[Lemma~5]{MP82}.


\begin{lemma}
\label{lem:cyclelift}
Let $X$ be a graph admitting
an $(m,p)$-semiregular automorphism $\rho$, where $p$ is a prime.
Let $C$ be a cycle of length $k$ in the quotient graph $X_\P$,
where $\P$ is the set of orbits of $\rho$.
Then, the lift of $C$ either contains a cycle of length
$kp$ or it consists of $p$ disjoint $k$-cycles.
In the latter case we have $d(S,S') = 1$ for every edge $SS'$ of $C$.
\end{lemma}


\section{Graphs arising from certain small rank/degree group actions}
\label{sec:small}
\indent

We deal here with the first four rows and the last row
of Table~\ref{tab:groups} of which the first three rows
correspond to
groups of rank $3$ and $4$.


In the first proposition we show the existence of Hamilton cycles in
the graphs  arising from the first three rows of Table~\ref{tab:groups}.
With the exception of one of the graphs arising from the
action of $M_{22}$ (Row 2 of Table~\ref{tab:groups})
and of the odd graph $O_4$ (Row 3 of
Table~\ref{tab:groups}, see Example~\ref{ex:O4})
for which the Hamilton cycle is constructed
via the lifting Hamilton cycle technique,
in all other cases the hamiltonicity is proved using
Proposition~\ref{pro:jack}.

\begin{proposition}
\label{the:small}
Vertex-transitive graphs arising from
primitive groups in Rows 1-3 of Table~\ref{tab:groups}
are hamiltonian.
\end{proposition}

\begin{proof}
Consider first the rank 3 action of the orthogonal group
$P\Omega^{\epsilon}(2d,2)$ on singular 1-spaces, for
$\epsilon\in\{+1,-1\}$,
in Row~1 of Table~\ref{tab:groups}. The primes $p$ and $q$ are as follows:
  $p=2^d-1$  and $q=2^{d-1} +1$ for $P\Omega^+(2d,2)$
and  $p=2^d+1$  and $q=2^{d-1} -1$ for $P\Omega^-(2d,2)$.
In the first case, the valencies of the two graphs in question are
$(2^{d-1} -1)(2^{d-1} +2) = q^2-q-2$ and $2^{2d-2} = q^2-2q+1$.
Similarly, in the second case the valencies of the two graphs are
$(2^{d-1} +1)(2^{d-2} -2) = q^2+q-2$ and $2^{2d-2} = q^2+2q+1$.
It is straightforward to see that for each of these graphs the valency exceeds
one third of the corresponding order, and so, by Proposition~\ref{pro:jack}, the result follows.

Consider now the action of $M_{22}$ of rank $3$ and degree $77$ given in Row~2 of Table~\ref{tab:groups}.
The corresponding nontrivial suborbits are of cardinalities  $16$ and $60$. Again Proposition~\ref{pro:jack} applies to the graph of valency $60$.
As for the graph of valency $16$ we observe that the quotiening
relative to a
$(7,11)$-semiregular automorphism gives
rise to the quotient graph isomorphic to
the complete graph $K_7$ and contains multiple edges.
An appropriate Hamilton cycle is then chosen and
Lemma~\ref{lem:cyclelift} applied  to obtain a Hamilton cycle in
the original graph.

Finally, consider $A_7$ acting on triples in
$\{1,2,\ldots,7\}$ from Row~3 of Table~\ref{tab:groups}.
This action gives rise to six different graphs associated with
three nontrivial suborbits of cardinalities 4, 12 and 18.
Clearly, in view of Proposition~\ref{pro:jack},
there is a Hamilton cycle in each of these graphs
with the exception of the odd graph $O_4$ of valency $4$.
As for the Hamilton cycle in $O_4$ it is constructed in
Example~\ref{ex:O4}. We would like to remark, however, that the
hamiltonicity of $O_4$ was first
proved by Balaban in \cite{B72}.
\end{proof}

The generalized orbital graphs arising from the action of $\PSL(2,61)$ on
the cosets of its maximal subgroup isomorphic to $A_5$
(Row 4 of  Table~\ref{tab:groups})
are of order $1891=31\cdot 61$. This action has $40$
nontrivial suborbits, $32$ of which are self-paired
and $8$ are not self-paired.
More precisely, one suborbit is of length $6$,
one of length $10$, two of length $12$, four of length $20$,
five of length $30$ and $27$ of length $60$. Of the latter
$8$ are non-self-paired.
In Figure~\ref{fig:61}
the graph of valency $6$ is shown together with a Hamilton
cycle in the quotient graph with respect to a semiregular
automorphism of order $61$ that, by Lemma~\ref{lem:cyclelift},
lifts to a full Hamilton cycle in the graph itself.
In a similar manner Hamilton cycles are constructed
in the remaining orbital graphs. Also, hamiltonicity of all of
these $36$ basic orbital graphs has been checked using
Magma~\cite{Mag}. Hence the following proposition holds.

\begin{proposition}
\label{pro:conder}
Vertex-transitive graphs of order $1891=31\cdot 61$
arising from the action of $\PSL(2,61)$ on $A_5$ given in Row 4 of
Table~\ref{tab:groups} are hamiltonian.
\end{proposition}

\begin{figure}[h!]
\begin{footnotesize}
\begin{center}
 \includegraphics[width=1.0\hsize]{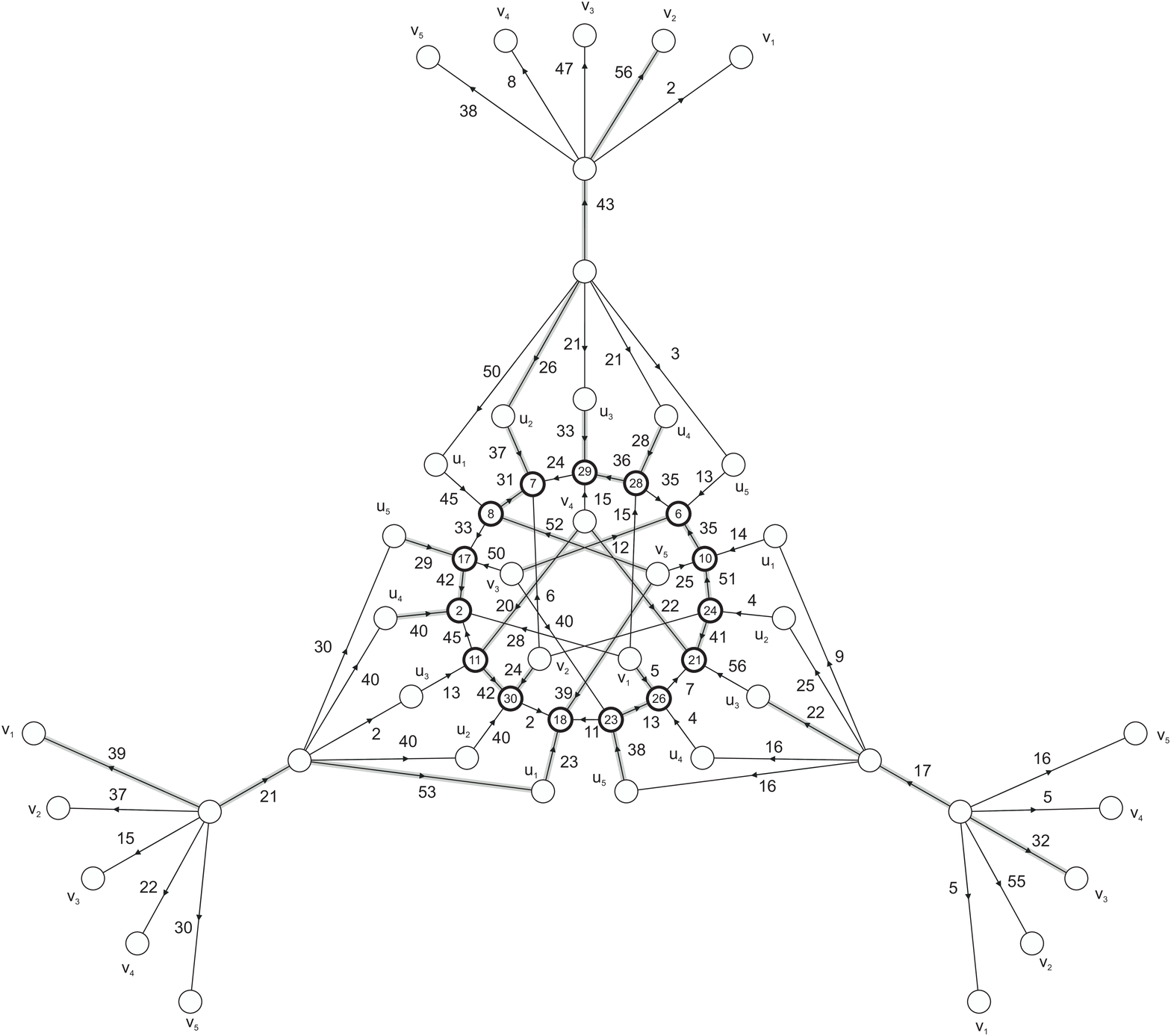}
\caption{\label{fig:61}\footnotesize
The  basic orbital graph $X$ arising from the action of
$\PSL(2,61)$ on the cosets of $A_5$ of valency $6$
given in Frucht's notation where labels inside circle are omitted
if there are no edges inside orbits and where we only give
the label $a$ inside the circle if the corresponding orbit
induces a cycle with jumps $a$ (circles in bold).
A Hamilton cycle in $X_\rho$ that lifts to a Hamilton cycle in $X$
is presented with bold grey lines.
}
\end{center}
\end{footnotesize}
\end{figure}

We end this section with generalized orbital graphs
of order $91=7\cdot 13$
arising from Row 7 in Table~\ref{tab:groups}.
It was proved in \cite{MS94} that the subdegrees of this action are
$1$, $4$, $4$, $4$, $6$,
$12$, $12$, $12$, $12$,  $12$, $12$.
With the exception of two suborbits of length $12$ all other suborbits
are self-paired. The corresponding basic orbital graphs are therefore
of the following valencies:
three of valency $4$, one of valency $6$, four of valency
$12$ and one of valency $24$.
Using Magma~\cite{Mag} it can be checked
that up to isomorphism there are in fact only
two graphs of valency $4$ and only three graphs of valency $12$.
For each of these seven graphs we find a semiregular automorphism
whose quotient contains a Hamilton
cycle that lifts to a Hamilton cycle in the original graph.
In all cases, with the exception of one of the graphs of valency $4$
where this automorphism is $(13,7)$-semiregular, the semiregular
automorphism giving a Hamilton cycle in the quotient graph
is of order $13$. In Figure~\ref{fig:missing}
the quotient graphs with a Hamilton cycle that lifts are
presented for both graphs of valency $4$.

\begin{proposition}
\label{pro:psl213}
Vertex-transitive graphs of order $91=7\cdot 13$
arising from the action of $\PSL(2,13)$ on $A_4$
given in Row 7 of Table~\ref{tab:groups} are hamiltonian.
\end{proposition}

\begin{figure}[h!]
\begin{footnotesize}
\begin{center}
 \includegraphics[width=0.45\hsize]{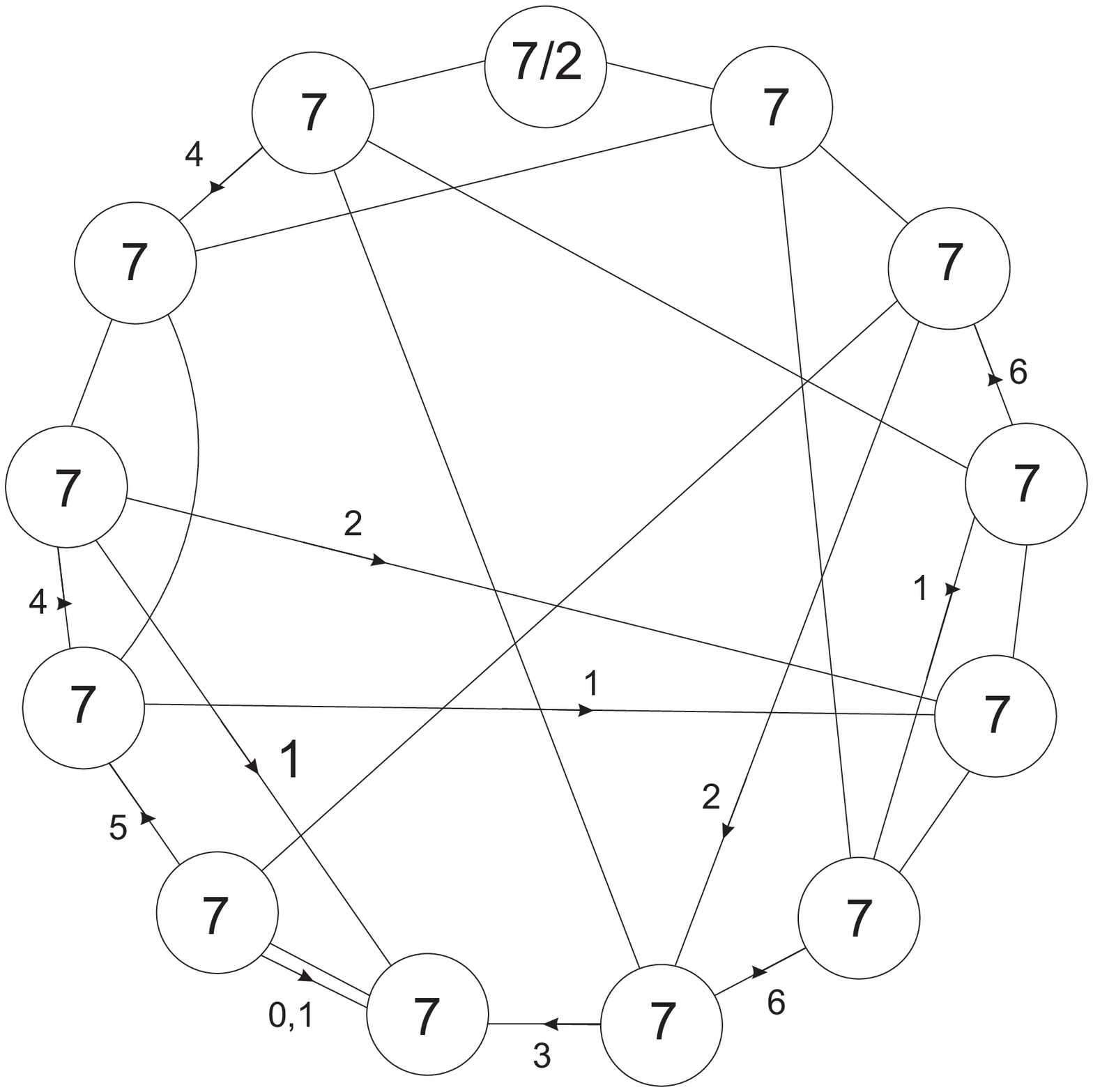}
 \includegraphics[width=0.45\hsize]{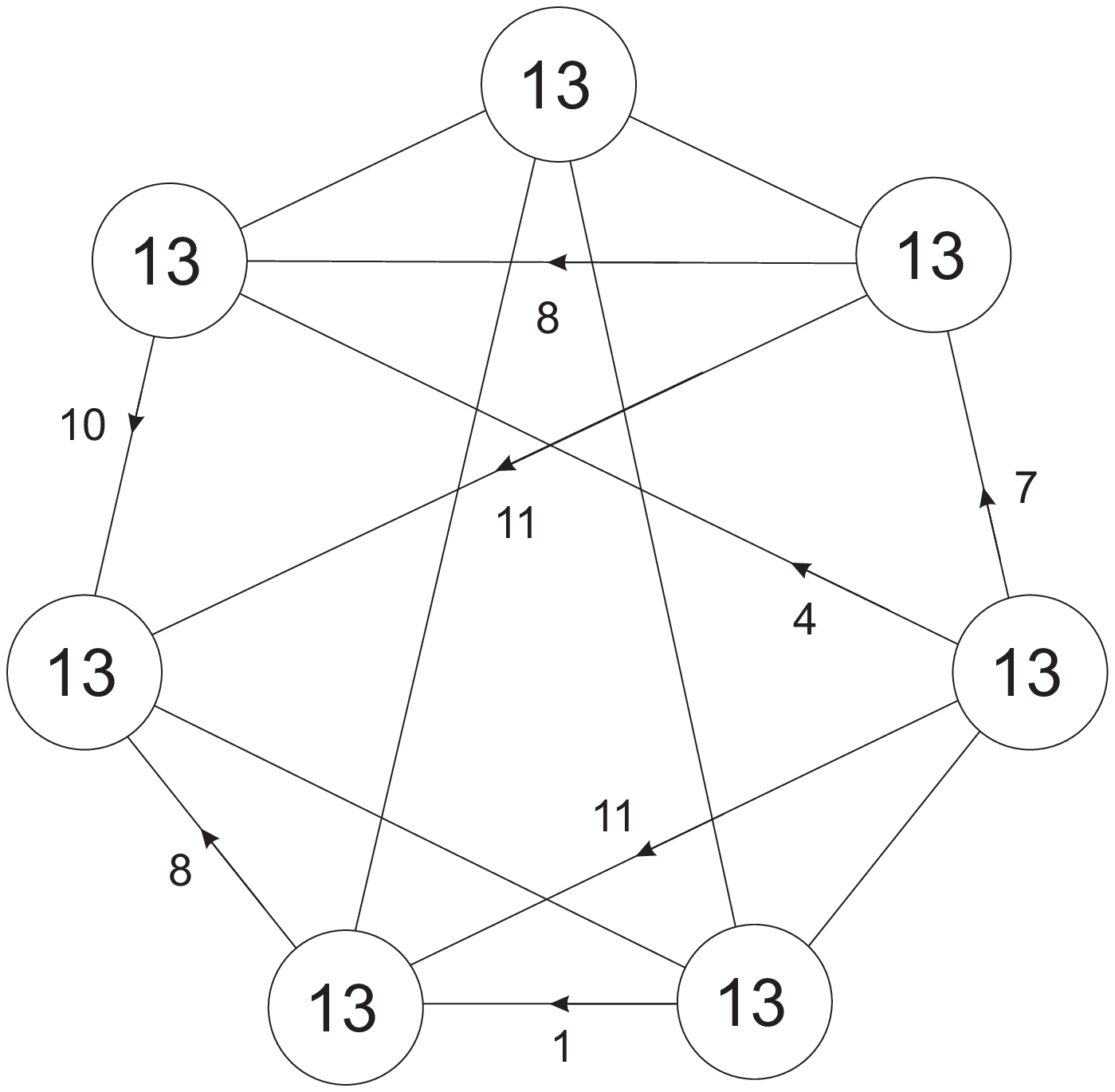}
\caption{\label{fig:missing}\footnotesize
The two vertex-transitive graphs arising from the
action of $\textrm{PSL}(2,13)$ on cosets of $A_4$ of valency $4$
given in Frucht's notation with respect to, respectively,
the $(13,7)$-semiregular automorphism, and
the $(7,13)$-semiregular automorphism.
Undirected lines carry label $0$.}
\end{center}
\end{footnotesize}
\end{figure}


\section{Actions of $\PSL(2,q^2)$}
\label{sec:PGL}
\indent

In this section the existence of Hamilton cycles
in basic orbital graphs arising
from the group action $\PSL(2,q^2)$
on the cosets of $\PGL(2,q)$ given in Row~5 of Table~\ref{tab:groups}
is considered. The following group-theoretic result due to
Manning will
be needed in this respect.

\begin{proposition}
\label{fix}
{\rm \cite[Theorem~3.6']{W66}}
Let $G$ be a transitive group on $\O$ and let $H=G_\a$
for some $\a\in \O$.
Suppose that $K\le G$ and at least one $G$-conjugate of
$K$ is contained in $H$. Suppose further that the set of
$G$-conjugates of $K$ which are contained in $H$ form $t$
conjugacy classes of $H$ with representatives $K_1$, $K_2$, $\cdots,$ $K_t$.
Then $K$ fixes $\sum_{i=1}^{t}|N_G(K_i):N_H(K_i)|$ points of $\O$.
\end{proposition}

Let $G=\PSL(2,q^2)$, where $q\ge 5$ is an odd prime.
Then $G$ has two conjugacy classes of subgroups
isomorphic to  $\PGL(2,q)$, with the corresponding representatives
$H$ and $H'$. Since each element
in $\PGL(2,q^2)$ interchanges
these two classes, it suffices to consider the action of
$G$ on the set $\mathcal{H}$ of right cosets of $H$ in $G$.
The degree of this action is $pq$, where $p=(q^2+1)/2$.
Further, let $P$ be a subgroup of $H'$ of order $q$,
that is, a subgroup of $G$ of order $q$
which has trivial intersection with $H$.
We have the following result.

\begin{lemma}
\label{orbit}
The action of $P$ on $\mathcal{H}$
is semiregular. Furthermore, the action of its normalizer
$N_G(P)$ on $\mathcal{H}$ has $\frac{q+1}2$ orbits of length $q$ and
one orbit of length $\frac{q^2(q-1)}2$.
\end{lemma}

\begin{proof}
We first prove that the action of $P$ on $\mathcal{H}$ is semiregular.
Suppose on the contrary that there exists $g\in G$ such that
$HgP=Hg$. Then $HgPg^{-1}=H$, and so $gPg^{-1}\le H$.
But this contradicts the choice of $P$. Hence $P$ is semiregular
on $\mathcal{H}$.

We now compute the orbits  of the normalizer
$N=N_G(P)\cong \ZZ_q^2\rtimes \ZZ_{q-1}$ of $P$ in $G$ in
its action on $\mathcal{H}$,
by analyzing subgroups of $H$  conjugate in $G$
to subgroups of $N$.
(Note that there is only one conjugacy class of subgroups in $G$
isomorphic to $N$.)
A subgroup of $N$ is isomorphic to one of the following groups
$\ZZ_q^2\rtimes \ZZ_{q-1}$, $\ZZ_q$, $\ZZ_q\rtimes \ZZ_{q-1}$,
$\ZZ_q\rtimes \ZZ_l$, where $2\le l<q-1$, and
$\ZZ_l$, where $l$ divides $q-1$.

Obviously, $N$ cannot fix a coset in $\mathcal{H}$ for
otherwise there would exists $g\in G$ such that $HgN=Hg$,
and so $gNg^{-1}\le H$ which is impossible since $|N|=q^2(q-1)$
and $|H|=(q-1)q(q+1)$.  Also, as $P$ is semiregular on $\mathcal{H}$
no subgroup
of $N$ isomorphic to $\ZZ_q$ fixes a coset in $\mathcal{H}$.

The group $N$ contains $q+1$ maximal subgroups isomorphic to
$\ZZ_q\rtimes \ZZ_{q-1}$, which form $q+1$ different conjugacy classes
in $N$, but are divided into two equal size classes in $G$, one containing subgroups of $H$
 and the other containing subgroups of $H'$.
Each of these two conjugacy classes
contains  $\frac{q+1}2$ subgroups
of $N$.
Let $K$ be such a subgroup of $H$ isomorphic to $\ZZ_q\rtimes \ZZ_{q-1}$.
Since $N_G(K)=N_H(K)=K$, it follows, by
Proposition~\ref{fix}, that $K$ fixes only the coset $H$.
Hence any subgroup of $N$ conjugate to $K$ in $G$ fixes
one coset of $\mathcal{H}$, and the corresponding
orbit of $N$ on $\mathcal{H}$ is of length $|N|/|K|=q$.
Since $N$ admits $\frac{q+1}2$ subgroups conjugate to $K$ in $G$, which
form $\frac{q+1}2$ different conjugacy classes inside $N$,
we can conclude that $N$ has $\frac{q+1}2$ orbits of length $q$.

A subgroup $K_1 \le K$ of $H$ isomorphic to
$\ZZ_q\rtimes \ZZ_l$, where $2\le l<q-1$,
has the same fixed cosets as $K$ (and so it is
a subgroup of a coset stabilizer). Consequently
$N$  does not have orbits of
length $q\cdot \frac {q-1}l$ for $2\le l<q-1$.
Further, for any subgroup $K_2\le K$ of $H$ isomorphic to
$\ZZ_l$, where $l$ divides $q-1$ and $l\ne 2$,
the fact that $|N_G(K_2):N_H(K_2)|=|
D_{q^2-1}:D_{2(q-1)}|=\frac {q+1}2$,
implies that $K_2$ fixes $\frac{q+1}2$ cosets. These cosets are clearly
contained in the above $\frac{q+1}2$ orbits of $N$ of length $q$,
and consequently $N$  does not have orbits of length $\frac {q-1}l.$

We have therefore show that the only other possible stabilizers are $\ZZ_2$
and $\ZZ_1$.
Since $|\mathcal{H}|=q(q^2+1)/2$ and since
the length of an orbit of $N$
on $\mathcal{H}$ with coset stabilizer isomorphic to $\ZZ_2$ or to $\ZZ_1$
equals, respectively, $\frac{q^2(q-1)}{2}$ and $q^2(q-1)$, we have
\begin{eqnarray}
\label{eq:qq}
\frac{q(q^2+1)}{2} = q\frac{q+1}{2} + a\frac{q^2(q-1)}{2} + bq^2(q-1),
\end{eqnarray}
where $a$ is the number of orbits of $N$ on $\mathcal{H}$
with coset stabilizer isomorphic to $\ZZ_2$ and
$b$ is the number of orbits of $N$ on $\mathcal{H}$ on
which $N$ acts regularly.  The equation (\ref{eq:qq})
simplifies to
\begin{eqnarray*}
q^2= q + aq(q-1) + 2bq(q-1),
\end{eqnarray*}
which clearly has $a=1$ and $b=0$
as the only possible solution.
This completes the proof of Lemma~\ref{orbit}.
\end{proof}

Lemma~\ref{orbit} will play an essential part in our
construction of Hamilton cycles in
basic orbital graphs arising from the action of
$\PSL(2,q^2)$ on cosets of
$\PGL(2,q)$ given in Row~5 of Table~\ref{tab:groups}.
The strategy goes as follows.
Let $X$ be such an orbital graph.
By Lemma~\ref{orbit},
the action of the normalizer $N=N_G(P)$
on the quotient graph $X_\P$ with respect to
the orbits $\P$ of a semiregular subgroup
$P$ consists of one large orbit of length $q(q-1)/2$ and
$(q+1)/2$ isolated vertices.
We will show the existence of a Hamilton cycle in $X$
by first showing that the subgraph of $X_\P$ induced
on the large orbit has at most two connected
components and that each component
contains a Hamilton cycle with double edges in the
corresponding quotient multigraph.
If the component is just one then
its Hamilton cycle is modified to a Hamilton cycle in
$X_\P$ by choosing in an arbitrary manner
$(q+1)/2$ edges and replacing them by $2$-paths having as central vertices the $(q+1)/2$ isolated vertices of $N$ in $X_\P$. By Lemma~\ref{lem:cyclelift},
this cycle lifts to a Hamilton cycle in $X$.
Such $2$-paths indeed exist because
every isolated vertex has to be
adjacent to every vertex in the large orbit
(see Lemma~\ref{mathching}).
If the subgraph of $X_\P$ induced on the large orbit
has two components with corresponding
Hamilton cycles $C_0$ and $C_1$,
then a Hamilton cycle in $X$
is constructed by first constructing a Hamilton
cycle in $X_\P$ in the following way.
We use two isolated vertices
to modify these two cycles $C_0$ and $C_1$ into
a cycle of length $q^2(q-1)/2+2$ by replacing
an edge in $C_0$ and an edge in $C_1$
by two $2$-paths each having one endvertex
in $C_0$ and the other in $C_1$, whereas the central
vertices are the above two isolated
vertices.  In order to produce the desired Hamilton
cycle in $X_\P$ the remaining isolated vertices are attached to this cycle in the same manner as in
the case of one component. By Lemma~\ref{lem:cyclelift},
this cycle lifts to a Hamilton cycle in $X$. Formal proofs are given in Propositions~\ref{pro:0} and~\ref{pro:ne0}.

It follows from the previous paragraph  that we only need to prove
that the subgraph of $X_\P$
induced on the large orbit of $N$ contains a Hamilton cycle with
at least one double edge in the corresponding multigraph
or two components each of which contains a Hamilton
cycle with double edges in the corresponding multigraph.
For this purpose we now proceed with the analysis of the structure
of basic orbital graphs (and corresponding suborbits)
arising from the action of
$\PSL(2,q^2)$ on cosets of
$\PGL(2,q)$ given in Row~5 of Table~\ref{tab:groups}.
We apply the approach taken in \cite{PX} where the computation
of suborbits is done using the fact that
$\PSL(2,q^2)\cong P\O ^{-}(4, q)$ and that the
action of  $\PSL(2,q^2)$ on the cosets of $\PGL(2,q)$
is equivalent to the induced action of $P\O ^{-}(4, q)$
on nonsingular  $1$-dimensional vector subspaces.
For the sake of completeness, we give a more detailed description
of this action together with a short explanation of
the isomorphism $\PSL(2,q^2)\cong P\O ^{-}(4, q)$
(see \cite[p.45]{KL} for details).

Let $F_{q^2}=F_q(\a)$, where $\a^2=\th$ for $F_q^*=\lg \th \rg $, and let
$\phi\in \Aut(F_{q^2})$ be the Frobenius automorphism of $F_{q^2}$
defined by the rule $\phi(a)=a^q$, $a\in F_{q^2}$.
(Note that $\phi$ is an involution.)
Let $W=\lg \mathbf{w}_1, \mathbf{w}_2\rg=F_{q^2}^2$ be a
natural $\SL(2,q^2)$-module. Then  $\SL(2,q^2)$ acts on $W$ in
a natural way. In particular, the action of
$g=\mm{a}{b}{c}{d}\in \SL(2,q^2)$ on $W$ is given by
\begin{eqnarray*}
\mathbf{w}_1g &=& a\mathbf{w}_1+c\mathbf{w}_2,\\
\mathbf{w}_2g &=& c\mathbf{w}_1+d\mathbf{w}_2.
\end{eqnarray*}
Let $\ow$ be an  $\SL(2,q^2)$-module with the underlying space $W$ and
the action of $\SL(2,q^2)$ defined by the rule $\mathbf{w}*g=\mathbf{w}g^\phi$,
where $g=(a_{ij})\in \SL(2,g^2)$ and $g^\phi=(\phi(a_{ij})_{ij})=(a_{ij}^q)$.
One can now see that
$\cdot \colon W\otimes \ow \times \SL(2,q^2) \to W\otimes \ow$
defined by the rule
$$
(\mathbf{w}\otimes \mathbf{w}')\cdot g=
\mathbf{w}g\otimes \mathbf{w}' * g=
\mathbf{w}g\otimes \mathbf{w}'g^\phi
$$
is an action of $\SL(2,q^2)$ on the $4$-dimensional space
$W\otimes \ow$ (that is, on a
tensor product of $W$ and $\ow$).
The kernal of this action  equals
$Z(\SL(2,q^2))$, and thus
this is in fact a $4$-dimensional representation  of
$G=\PSL(2,q^2)$ (an embedding of $G$ into $\GL(4,q^2)$).
Further, the set
$\B=\{\mathbf{v}_1,\mathbf{v}_2,\mathbf{v}_3,\mathbf{v}_4\}$, where
\begin{eqnarray*}
\mathbf{v}_1&=&\mathbf{w}_1\otimes \mathbf{w}_1,\\
\mathbf{v}_2&=&\mathbf{w}_2\otimes \mathbf{w}_2, \\
\mathbf{v}_3&=&\mathbf{w}_1\otimes \mathbf{w}_2+\mathbf{w}_2\otimes
\mathbf{w}_1,\\
\mathbf{v}_4&=&\a (\mathbf{w}_1\otimes \mathbf{w}_2-\mathbf{w}_2\otimes \mathbf{w}_1),
\end{eqnarray*}
is a basis for $W\otimes \ow$ over $F_{q^2}$.

Since $G$ fixes the $4$-dimensional space
$V=\mathrm{span}_{F_q}(\B)$ over $F_q$ it can
be viewed as a subgroup of $\GL(4,q)$.
A non-degenerate symplectic
form $f$ of $W$ and $\ow $  defined by
$f(\mathbf{w}_1,\mathbf{w}_2)=
-f(\mathbf{w}_2, \mathbf{w}_1)=1$ and
$f(\mathbf{w}_1, \mathbf{w}_1)=f(\mathbf{w}_2, \mathbf{w}_2)=0$
is fixed by $\SL(2, q^2)$.
It follows that
$G$ fixes  a non-degenerate symmetric bilinear
form  of $W\otimes \ow $ defined by the rule
$$
(\mathbf{w}_1'\otimes \mathbf{w}_2',
\mathbf{w}_1''\otimes \mathbf{w}_2'')=
f(\mathbf{w}_1', \mathbf{w}_1'')f(\mathbf{w}_2', \mathbf{w}_2'').
$$
Then  we have
$$
((\mathbf{v}_i,\mathbf{v}_j))_{4\times 4}=
\left(
\begin{array}
{cccc}
0 &1 &0 &0\\
1&0&0&0\\
0&0& -2&0\\
0&0&0&2\th
\end{array}\right),
$$
and so for $\mathbf{x}=\sum_{i=1}^4x_iv_i\in V$ and $\mathbf{y}=\sum_{i=1}^4y_iv_i\in V$  the symmetric form $(\mathbf{x},\mathbf{y})$ and the associated   quadratic form
$\mathbf{Q}$ are given by the rules
$$
(\mathbf{x},\mathbf{y})=x_2y_1+x_1y_2-2x_3y_3
+2\th x_4y_4 \textrm{ and  }
\mathbf{Q}(\mathbf{x})=\frac 12(\mathbf{x},\mathbf{x})=x_1x_2-x_3^2+\th  x_4^2.
$$
By computation it follows that
$\mathbf{Q}$ has $q^2+1$ singular $1$-dimensional subspaces of $V$. As for the remaining $q(q^2+1)$
nonsingular 1-dimensional
subspaces,
$G$ has two orbits $\{\lg \mathbf{v}\rg \di \mathbf{Q}(\mathbf{v})=1, \mathbf{v}\in V\}$ and  $\{ \lg \mathbf{v}\rg \di \mathbf{Q}(\mathbf{v})\in F_q^*\setminus S^*, \mathbf{v}\in V \}$ which are interchanged
by a diagonal automorphism of $G$.
Let $\O$ be the first of these two orbits.  Then the action
of $G$ on $\mathcal{H}$ is equivalent to the action of $G$ on $\O$.
By comparing their orders,  we get $\PSL(2,q^2)\cong P\O ^{-}(4, q)$. The following result characterizing suborbits of
the action of $G$ on the cosets of $\PGL(2,q)$
in the context  of the action of
$P\O ^{-}(4, q)$ on $\O$ was proved in \cite{PX}.

\begin{proposition}
\label{pro:PX}
{\rm\cite[Lemma~4.1]{PX}}
For  any $\lg \mb{v}\rg \in \O$, the nontrivial suborbits of the action of $G$ on $\O$
(that is, the orbits of $G_{\lg \mb{v}\rg}$) are the sets
$\D_{\pm \ld}=\{\lg \mb{x}\rg \in \O\di (\mb{x},\mb{v})=\pm 2\ld \}, $     where $\ld \in F_q$, and
\begin{enumerate}[(i)]
\itemsep=0pt

\item $|\D_0|=\frac{q(q\mp 1)}2$  for $q\equiv \pm {1\pmod 4};$

 \item $|\D_{\pm 1}|=q^2-1;$

\item $|\D_{\pm \ld}|=q(q+1)$ with  $\ld ^2-1\in N^*;$

 \item  $|\D_{\pm \ld}|=q(q-1)$ with  $\ld ^2-1\in
 S^*$.
 \end{enumerate}
Moreover, all the suborbits are self-paired.
\end{proposition}

Let $X=X(G,H,\D_\ld)$  be  the basic orbital graph associated
with $\D_\ld $, and
 take
$$
\rho=\left[
 \begin{array}{cc} 1&1\\0&1\end{array}\right] \in G.
$$
(For simplicity reasons we refer
to the elements of $G$ as matrices; this should cause no confusion.) Clearly, $\rho$ is of order $q$.
For $k\in F_q$ we have
\begin{eqnarray*}
\mb{v}_1\rho^k&=&\mb{v}_1,\\
\mb{v}_2\rho^k&=&k^2\mb{v}_1+\mb{v}_2+k\mb{v}_3,\\
\mb{v}_3\rho^k&=&2k\mb{v}_1+\mb{v}_3,\\
\mb{v}_4\rho^k&=&\mb{v}_4,
\end{eqnarray*}
and so $\rho^k$ maps the vector $x=\sum_{i=1}^4x_i\mb{v}_i\in V$ to
$$
\mb{x}\rho^k=(x_1+k^2x_2+2kx_3)\mb{v}_1+x_2\mb{v}_2+(kx_2+x_3)\mb{v}_3+x_4\mb{v}_4.
$$
Identifying $\mb{x}$ with $(x_1, x_2, x_3, x_4)$ we have
$$
\mb{x}\rho^k=(x_1+k^2x_2+2kx_3, x_2, kx_2+x_3, x_4).
$$
One can check that for $k\ne 0$ we have $\lg \mb{x}\rho^k\rg \ne \lg \mb{x}\rg $, and thus $\rho$ is
$(p,q)$-semiregular. Let $P=\la\rho\ra$, and let
$\P$ be the set of orbits of $P$. These orbits
will be referred to as blocks.
The set $\O$  decomposes into two
subsets
each of which is a union of blocks from $\P$:
\begin{enumerate}[]
\itemsep=0pt
\item $\mathcal{I}=\lg (0,0,x_3,x_4)\rg P =\{ \lg (2kx_3, 0, x_3, x_4)\rg \di k\in F_q\},$ where  $-x_3^2+\th x_4^2=1$.

\item $\mathcal{L}=\lg (x_1, x_2, 0, x_4)\rg P =\{ \lg (x_1+k^2x_2, x_2, kx_2, x_4)\rg \di k\in F_q\},$ where   $x_2\ne 0$ and $x_1x_2+\th x_4^2=1$.
\end{enumerate}
Note that the subset $\mathcal{I}$ contains
$\frac{q(q+1)}2$ vertices which form  $\frac{q+1}2$ blocks,  and  the subset
$\mathcal{L}$ contains $\frac{q^2(q-1)}2$ vertices which form $\frac{q(q-1)}2$ blocks.
By $\mathcal{I}_P$ and $\mathcal{L}_P$, we denote, respectively, the set of blocks in $\mathcal{I}$ and  $\mathcal{L}$; that is,
$\P=\mathcal{I}_P\cup\mathcal{L}_P$.

\begin{remark}
{\rm Note that
$$
N=N_G(P)=\lg \left[
 \begin{array}{cc}
a&b\\0&a^{-1}
\end{array}\right]\di  a\in \lg \a\rg , b\in F_{q^2}\rg,
$$
where $\la\alpha\ra$ denotes the multiplicative
group generated by $\alpha$.
One may check directly  that $\mathcal{I}_P$ consists precisely of
 the orbits of $N$ of length $q$  and  that $\mathcal{L}$ is
the orbit of $N$ of length  $\frac{q^2(q-1)}2$.}
 \end{remark}

In the next lemma we observe that
$X\la\mathcal{L}\ra$ and
$X\la\mathcal{L}\ra_\P$
are  vertex-transitive
and show that the bipartite
subgraph of $X_\P$ induced by
$\mathcal{I}_P$ and $\mathcal{L}_P$
is a complete bipartite graph.

\begin{lemma}
\label{mathching} With the above notation, the following hold:
\begin{enumerate}
\itemsep=0pt
\item[(i)]  The induced subgraph $X\la\mathcal{L}\ra$ and the quotient graph
$X\la\mathcal{L}\ra_\P$ are both vertex-transitive.
\item[(ii)]   For
$\la \mb{x}\ra P\in\mathcal{I}_P$ and $\la \mb{y}\ra P\in\mathcal{L}_P$
we have
$$
d(\la \mb{x}\ra P, \la \mb{y}\ra P)=\left\{
\begin{array}{ll}
1,&\ld =0\\
2,&\ld\ne 0
\end{array}\right..
$$
\end{enumerate}
\end{lemma}

\begin{proof}
By Lemma~\ref{orbit},  $N$ is transitive on $\mathcal{L}$, and so the induced subgraph $X\la\mathcal{L}\ra$ and the
quotient graph $X\la\mathcal{L}\ra_\P$  are
  both vertex transitive, and thus (i) holds.

To prove (ii), 
take two arbitrary vertices  $\lg \mb{x}\rg =\lg (0, 0,x_3, x_4)\rg \in \mathcal{I}$ and   $\lg \mb{y}\rg =\lg (y_1, y_2, 0, y_4)\rg \in \mathcal{L}$. Then  $y_2\ne 0$ and $x_3\ne 0$, and $\lg \mb{x}\rg
  \thicksim \lg \mb{y}\rho^k\rg $  if and only if
$$(\mb{x},\mb{y}\rho^k)=((0,0,x_3,x_4),
(y_1+k^2y_2,y_2,ky_2, y_4))=\pm 2\ld ,$$ that is, if and only if
\begin{eqnarray}
\label{eq:II}
-2x_3ky_2+2\th x_4y_4=\pm 2\ld.
\end{eqnarray}
From (\ref{eq:II})
we get that  $k=\frac{\th x_4y_4\mp \ld }{x_3y_2}$ and so
for given $\lg \mb{x}\rg $  and $\lg \mb{y}\rg$ we have
a unique solution for $k$  if $\ld =0$
and   two solutions if $\ld \ne 0$.
It follows that for
$\la \mb{x}\ra P\in\mathcal{I}_P$ and $\la \mb{y}\ra P\in\mathcal{L}_P$
we have $d(\la \mb{x}\ra P, \la \mb{y}\ra P)=1$ or 2, depending on whether $\ld =0$ or $\ld \ne 0$,
 completing part (ii) of
Lemma~\ref{mathching}.
\end{proof}

In what follows, we divide the proof into two cases
depending on whether $\ld =0$
or $\ld \ne 0$.


\subsection{Case  $\D_0$ }
\noindent

Let
$$
\e=\left\{\begin{array}{ll}
2, & q\equiv {1,3\pmod 8}\\
0, & q\equiv {5,7\pmod 8}
\end{array}\right..
$$

The following lemma gives us
the number of edges inside a block
and between two blocks from
$\mathcal{L}_P$ for the orbital graph $X(G,H,\S_0)$.

\begin{lemma}
\label{0}
Let $X=X(G,H,\S_0)$. Then for $\la \mb{x}\ra \in\mathcal{L}$ the following hold:
\begin{enumerate}[(i)]
\itemsep=0pt
\item $d(\la \mb{x}\ra P)=\e$,
\item $d(\la \mb{x} \ra P, \la \mb{y} \ra P)=1$ for $\frac{q+1}2$ blocks $\la \mb{y}\ra P\in\mathcal{L}_P$,
\item $d(\la \mb{x} \ra P, \la \mb{y} \ra P)=2$ for $\frac 14(q^2-3q-2(\e+1))$ blocks $\la \mb{y}\ra P\in\mathcal{L}_P$ if $q\equiv  {1\pmod 4}$, and
for $\frac 14(q^2-q-2(\e+1))$ blocks $\la \mb{y}\ra P\in\mathcal{L}_P$ if $q\equiv  {3\pmod 4}$.
\end{enumerate}
\end{lemma}

\begin{proof}
Fix a vertex $\lg \mb{x}\rg =\lg (1, 1, 0, 0)\rg \in \mathcal{L}$.
For any  $\lg \mb{y}\rg =\lg (y_1, y_2, 0, y_4)\rg \in\mathcal{L}$, where $y_2\ne 0$, we have $\lg
 \mb{x}\rg \thicksim \lg \mb{y}\rg \rho^k$  if and only if
$y_1+(k^2+1)y_2=0$, and therefore, since $y_1y_2+\th y_4^2=1$,
if and only if
\begin{eqnarray}
\label{eq:IV}
k^2=-y_2^{-2}+\th (y_2^{-1}y_4)^2-1.
\end{eqnarray}
It follows from (\ref{eq:IV})
that $\lg \mb{x}\rg $ is adjacent to one vertex
in the block $\lg \mb{y}\rg P\in\mathcal{L}_P$ if  $k=0$ and to two vertices in this block if $k\ne 0$.
Clearly, $k=0$ if and only if
\begin{eqnarray}
\label{eq:V} \th y_4^2=1+y_2^2.\end{eqnarray}
Proposition~\ref{num} implies that
(\ref{eq:V})  has $q+1$ solutions for $(y_2,y_4)$, and therefore
since $\la \mb{y}\ra=\la -\mb{y}\ra$
we have a total of $\frac{q+1}2$ choices for $\lg \mb{y}\rg $.
This implies that
 $d(\lg \mb{x}\rg P,\la \mb{y}\ra P)=1$ for $\frac{q+1}2$ blocks
$\la \mb{y}\ra P\in\mathcal{L}_P$,
proving part (ii).

To prove part (i), take $\mb{y}=\pm \mb{x}=\pm (1,1,0,0)$. Then, by (\ref{eq:IV}), there are edges inside the block $\lg \mb{x}\rg P$  if and only if $k^2=-2$.  This equation has solutions
 if and only if $q\equiv {1,3\pmod 8}$ (see Propositions~\ref{pro:-1} and~\ref{pro:2}),
and thus the induced subgraph
$X\la \lg \mb{x}\rg P\ra$ is a $q$-cycle for  $q\equiv {1,3\pmod 8}$ and a totally disconnected graph
$qK_1$ if $q\equiv {5,7\pmod 8}.$

Finally, to prove part (iii) let $m$ be the number of blocks
$\la \mb{y}\ra P\in\mathcal{L}_P$ for which $d(\lg \mb{x}\rg P, \la \mb{y}\ra P)=2$.
Suppose first that $q\equiv {1\pmod 4}$.
Then, combining together the facts that
$X$ is of  valency $\frac 12q(q- 1)$, that $d(\la \mb{x}\ra P)=\e$ and
that  $\lg \mb{x}\rg $ is adjacent to $\frac 12(q+1)$ vertices
in the set $\mathcal{I}$
and to exactly one vertex from $\frac{q+1}2$ blocks in
$\mathcal{L}_P$, we have
  $$ m=\frac 12(\frac 12q(q- 1)-\frac{q+1}2-\frac{q+1}2-\e)=\frac 14(q^2-3q-2(1+\e)).$$

Suppose now that $q\equiv {3\pmod 4}$.
Then, replacing the valency of $X$
in the above computation
with $\frac 12q(q+ 1)$ we obtain, as desired, that
  $$ m=\frac 14(q^2-q-2(1+\e)).$$
\end{proof}

We are now ready to prove
existence of a Hamilton cycle in $X(G,H,\S_0)$.

\begin{proposition}
\label{pro:0}
The graph $X=X(G,H,\S_0)$ is hamiltonian.
\end{proposition}

\begin{proof}
Let $X\la\mathcal{L}\ra'$ be the graph obtained from $X\la\mathcal{L}\ra$
by deleting the edges  between  any two blocks
 $B_1,B_2\in\mathcal{L}_P$ for which $d(B_1,B_2)=1$ (see Lemma~\ref{0}(ii)). By Lemma~\ref{mathching}, $X\la \mathcal{L} \ra_\P$ is vertex-transitive,
and consequently
one can see that also $X\la\mathcal{L}\ra'_\P$ is vertex-transitive.

If $q\equiv  {1\pmod 4}$ then
Proposition~\ref{pro:PX} and Lemma~\ref{0}(iii) combined together
imply that $X\la\mathcal{L}\ra'_\P$ is of valency
$m=\frac 14(q^2-3q-2(1+\e))$.
If, however, $q\equiv  {3\pmod 4}$ then
Proposition~\ref{pro:PX} and Lemma~\ref{0}(iii) combined together
imply that $X\la\mathcal{L}\ra'_\P$ is of valency
$m=\frac 14(q^2-q-2(1+\e))$.
If $q=5$ then $\e=0$ and $m=\frac 14(q^2-3q-2(1+\e))=2$.
If $q\ge 7$ then using the facts
that $q^2-7q-6(1+\e)\ge 0$ for $q\equiv  {1\pmod 4}$ and that
$q^2-q-6(1+\e)\ge 0$ for $q\equiv   {3\pmod 4}$
one can see that
$$
m=\frac 14(q^2-(2\pm 1)q-2(1+\e))\ge \frac13\frac{q(q-1)}2=\frac{1}{3}|\mathcal{L}_P|.
$$

Suppose first that  $X\la \mathcal{L} \ra_\P'$ is connected.  If $q=5$, then
$X\la\mathcal{L}\ra'_\P$ is just a  cycle $C$. For $q\ge 7$, by
 Proposition~\ref{pro:jack},  $X\la\mathcal{L}\ra'_\P$ admits a Hamilton cycle, say  $C$ again.
Clearly $C$ is also a Hamilton cycle of $X\la\mathcal{L}\ra_\P$.
Form  $C$   a Hamilton cycle in $X_\P$
can  be constructed by choosing arbitrarily
$(q+1)/2$ edges and replacing them by $2$-paths having as
central vertices the $(q+1)/2$ isolated vertices of $N$ in $X_\P$.
 By Lemma~\ref{lem:cyclelift},  this lifts to a Hamilton cycle
in $X$.

Next, suppose that  $X\la \mathcal{L} \ra_\P'$ is disconnected.
For $q=5$, since $X\la \mathcal{L} \ra_\P'$ is a vertex transitive  graph of order 10, it must be a union of two 5-cycles.
For $q\ge 7$, since  $m\ge \frac{1}{3}|\mathcal{L}_P|$,
it follows that $X\la \mathcal{L} \ra_\P'$  has just two components.
By Proposition~\ref{pro:jack},
each component  admits a Hamilton cycle.
Take a respective Hamilton  path for each component, say $\mathcal{U}=U_1U_2\cdots  U_l$, and
$\mathcal{U}'=U'_1 U'_2\cdots, U'_l$, where $l=\frac{q(q-1)}4.$
Choose any two  isolated vertices $W_1$ and $W_2$
and construct the cycle
$\mathcal{D}=W_1\mathcal{U}W_2\mathcal{U}'W_1$. Choose  arbitrarily $(q+1)/2-2$ edges in $\mathcal{U} \cup \mathcal{U}'$
 and replace them by $2$-paths having as
central vertices  the remaining $(q+1)/2-2$ isolated vertices. Then we get a Hamilton cycle in $X_\P$, which,
by Lemma~\ref{lem:cyclelift},
lifts to a Hamilton cycle in $X$.
\end{proof}


\subsection{Case   $\D_\ld $ with $\ld \ne 0$}
 \noindent

%
%
%


\begin{proposition}
 \label{pro:ne0}
The graph $X=X(G,H,\S_{\pm \ld})$, where $\ld\ne 0$, is hamiltonian.
\end{proposition}

\begin{proof}
As in the proof of Lemma~\ref{0}, fix a vertex $\lg \mb{x}\rg =\lg (1, 1, 0, 0)\rg \in \mathcal{L}$.  For any  $\lg \mb{y}\rg =\lg (y_1, y_2, 0, y_4)\rg \in\mathcal{L}$, $y_2\ne 0$,  we have $\mb{y}\rho^k=(y_1+k^2y_2,y_2,ky_2, y_4)$, and so  $\lg \mb{x}\rg \thicksim \lg \mb{y}\rho^k\rg $  if and only if
$y_1+(k^2+1)y_2=\pm 2\ld$,
which implies, since $y_1y_2+\th y_4^2=1$, that
$$
k^2=\pm 2\ld y_2^{-1}-y_2^{-2}+\th (y_2^{-1}y_4)^2-1.
$$
It follows that there are at most four solutions for $k$. Hence
each vertex in $\mathcal{L}$ is adjacent to at most four vertices in
the same block from $\mathcal{L}_P$
(including the block containing this vertex).

Let $m$ be the  valency of $X\la \mathcal{L}\ra_\P$.
Since, by Proposition~\ref{pro:PX}, the valency of $X$ is, respectively,
$q^2-1$,  $q^2-q$ and  $q^2+q$,
we get that  $m\ge \frac13|\mathcal{L}_P|=\frac13\frac{q(q-1)}2$ provided
$$ m\ge \frac 14((q^2-j)-(q+1)-4)\ge \frac 14(q^2-q-j-5)\ge \frac13\frac{q(q-1)}2,
$$
where $j\in\{1, q, -q\}$ for $q\ge 7$  and   $j\in \{1,-q\}$ for $q=5$.  
One can check that this inequality
holds for all $q\ge 5$.
We can therefore conclude that $X\la\mathcal{L}\ra_\P$,
which is vertex-transitive by Lemma~\ref{mathching},
has at most two connected components.
The rest of the argument follows word by word from the argument given in the proof of Proposition~\ref{pro:0}, since, by Lemma~\ref{mathching},  $d(\la \mb{x}\ra P, \la \mb{y}\ra P)=2$,  for any $\la \mb{x}\ra P\in\mathcal{I}_P$ and $\la \mb{y}\ra P\in\mathcal{L}_P$.
\end{proof}

\section{Actions of $\PSL(2,p)$}
\label{sec:PSL}
\noindent

In this section the existence of Hamilton cycles
in basic orbital graphs arising
from the group action in Row~6 of Table~\ref{tab:groups}
is considered.
Let us remark that
subdegrees of all primitive permutation representations of $\PSL(2,k)$
were calculated in \cite{T}. This thesis is quite unavailable, but
some extractions appeared in \cite{FI}.

Observe first that in order for $q = (p+1)/2$
to be a prime we must have $p\equiv {1\pmod 4}$,
and that by Row 6 of Table~\ref{tab:groups} we have $p\ge 13$.
Throughout this section let $F = F_p$ be a finite field of order $p$,
and let $F^*$, $S^*$ and $N^*$ be defined as in
Subsection~\ref{ssec:numbers}, that is:
$F^* = F \setminus \{0\}$, $S^*=\{s^2\colon s\in F^*\}$
and $N^*=F^*\setminus S^*$.

In the description of the graphs arising from the action of
$G = \PSL(2,p)$ on the set
${\cal H}$ of right cosets of $H = D_{p-1}$ we follow \cite{MS3}.
(For further details as well as all the proofs see \cite{MS3}.)
For simplicity reasons we refer
to the elements of $G$ as matrices; this should cause no confusion.
We may choose  $H$ to consist of all the matrices of the form
$$
\mm{x}{0}{0}{x^{-1}} \ (x\in F^*) \ \textrm{ and } \ \mm{0}{-x}{x^{-1}}{0} \ (x\in F^*).
$$
Note that, since $p\ge 13$, $H$ is a dihedral subgroup $D_{p-1}$.
Further, let
$$g=\mm{a}{b}{c}{d}\in G
$$
be fixed. Then each element of the right
coset $Hg$ is either of the form
$$
\mm{ax}{bx}{cx^{-1}}{dx^{-1}} \ \textrm{ or of the form } \
\mm{cx}{dx}{-ax^{-1}}{-bx^{-1}} \ (x\in F^*).
$$
Moreover, a typical element of the left coset $gH$ is either of the form
$$
\mm{ax}{bx^{-1}}{cx}{dx^{-1}} \ \textrm{ or of the form } \
\mm{bx^{-1}}{-ax}{dx^{-1}}{-cx} \ (x\in F^*).
$$
The computation and description of the suborbits of $G$ acting by
right multiplication on the set $\H$ of the right cosets of $H$ in $G$
depends heavily on the concise description of the elements of $\H$.
If $g$ satisfies $ab\ne 0$,
define $\xi(g) = ad$ and $\eta(g) = a^{-1}b$, and call
$\chi(g) = (\xi(g),\eta(g))$ the {\it character} of $g$.
The following proposition, proved in \cite[Lemmas~2.1 and 2.2]{MS3},
recalls basic properties of characters.

\begin{proposition}
{\rm \cite[Lemmas~2.1 and 2.2]{MS3}}
\label{pro:characters}
Let $\chi(g)=(\xi,\eta)$.
\begin{enumerate}[(i)]
\itemsep=0pt
\item If $abcd\ne 0$ and $g'\in Hg$ then either $\chi(g')=(\xi,\eta)$
or $\chi(g')=(1-\xi,\xi\eta/(\xi-1))$.
\item If $ab\ne 0$ and $g'\in gH$ then either
$\chi(g')=(\xi,y\eta)$ or $\chi(g')=(1-\xi,-y\eta^{-1})$ for some $y\in S^*$.
\end{enumerate}
\end{proposition}

Let $\approx$ be the equivalence relation on $F \times F^*$ defined by
\begin{eqnarray}
\label{eq:relation}
(\xi,\eta) \approx (1-\xi,\frac{\xi \eta}{\xi- 1}) \textrm{ for } \xi \neq 0,1.
\end{eqnarray}
There is then a natural identification of the sets ${\cal H}$
and $(F\times F^*)/_{\approx} \cup \{\infty\}$
where $\infty$ corresponds to $H$ and $(\xi,\eta)$ corresponds
to the coset $Hg$
satisfying $\chi(g) = (\xi,\eta)$. This identification will be used throughout the rest of this section.
Note that the identification is direct with
$(F^*\setminus \{1\}\times F^*)/_{\approx} \cup \{\infty\} \cup \{0,1\}\times F^*$.
The symbol $\infty$ corresponds to the subgroup $H$ and
$\{0,1\}\times F^*$ represents all those cosets which contain at least one
matrix with exactly one of the entries equal  to
zero. In summary, all matrices which have two entries
equal to zero
are in the subgroup $H$, and
all matrices with exactly one of the entries equal  to
zero belong to
$2(p-1)$ cosets of $H$ with typical representatives
$$
\mm{1}{y}{0}{1} \ \textrm{ and } \ \mm{1}{y}{-y^{-1}}{0}  \ (y\in F^*)
$$
with respective characters $(1,y)$ and $(0,y)$.
Finally, all remaining cosets in $\H$ contain matrices with no entry equal to zero,
where we have to bring in the equivalence relation $\approx$ on characters
defined by (\ref{eq:relation}).

For each  $\xi \in F$ define the
following subsets of $\H$ (in fact subsets of
$(F\times F^*)/_{\approx} \cup \{\infty\}$ if $\xi\ne 0,1$; these subsets
will be used to represent subsets of $\H$
throughout this section, which should cause no confusion:
\begin{eqnarray*}
{\cal S}_{\xi}^+ &=& \{(\xi,\eta)\colon\eta \in S^* \},\\
{\cal S}_{\xi}^- &=& \{(\xi,\eta)\colon\eta \in N^* \},\\
{\cal S}_{\xi} &=& {\cal S}_{\xi}^+ \cup {\cal S}_{\xi}^-.
\end{eqnarray*}
Observe that (because of the equivalence relation $\approx$)
the sets $\{{\cal S}_{\xi}^+$, ${\cal S}_{\xi}^- \}$ and
$\{{\cal S}_{1-\xi}^+$, ${\cal S}_{1-\xi}^- \}$ coincide for $\xi \neq 0,1$.
Moreover, since $(\frac{1}{2},\eta) \approx (\frac{1}{2},-\eta)$, it
follows that the cardinality of ${\cal S}_{\xi}$ is $p-1$ except
for $\xi = \frac{1}{2}$ when the cardinality is $\frac{p-1}{2}$.
Similarly, the cardinalities of ${\cal S}_{\xi}^ + $
and  ${\cal S}_{\xi}^-$ are $\frac{p-1}{2}$ except for
$\xi = \frac{1}{2}$ when the cardinalities are $\frac{p-1}{4}$.
The following result proved in \cite{MS3}
determines all the suborbits of the action of $G$ on
${\cal H}$. The suborbits given in the theorem
are summarized in Table~\ref{tab:cases}.

\begin{theorem}
\label{th:dihed}
{\rm \cite[Theorem]{MS3}}
The action of $G$ on ${\cal H}$ has
\begin{enumerate}[(i)]
\itemsep=0pt
\item $\frac{p+7}{4}$ suborbits of length $p-1$, all of them self-paired.
These are ${\cal S}_0 ^+ \cup {\cal S}_1 ^+$,
${\cal S}_0 ^- \cup {\cal S} _1 ^-$ and ${\cal S}_{\xi}$ for all those $\xi\ne\frac{1}{2}$ which satisfy $\xi^{-1}-1 \in N^*$.

\item  $\frac{p-5}{2}$ suborbits of length $\frac{p-1}{2}$, namely
${\cal S}_\xi^+$ and ${\cal S}_\xi^-$, where
$\xi\ne\frac{1}{2}$ and $\xi^{-1}-1 \in S^*$.
Among them the self-paired suborbits correspond to all those $\xi$
for which both $\xi$ and $\xi-1$ belong to $S^*$ and so their
number is $\frac{p-9}{4}$ if $p\equiv {1\pmod 8}$ and
$\frac{p-5}{4}$ if $p\equiv {5 \pmod 8}$.

\item  2 suborbits of length $\frac{p-1}{4}$, namely
${\cal S}_{\frac{1}{2}}^+$ and  ${\cal S}_{\frac{1}{2}}^-$ which
are self-paired if and only if $p\equiv {1\pmod 8}$.
\end{enumerate}
\end{theorem}

\begin{example}
\label{ex:smallest}
{\rm The smallest admissible pair of primes  $p = 13$ and $q = 7$
gives rise to the action of $G=\PSL(2,13)$ on cosets of $H=D_{12}$
with the following suborbits:
\begin{enumerate}[(i)]
\itemsep=0pt
\item ${\cal S}_0 ^+ \cup {\cal S}_1 ^+$,
${\cal S}_0 ^- \cup {\cal S} _1 ^-$, ${\cal S}_2$, ${\cal S}_3$,
${\cal S}_5$ of size $12$, all of them self-paired,
\item ${\cal S}_4^+$, ${\cal S}_4^-$
of size $6$, all of them self-paired,
\item ${\cal S}_6^+$,
${\cal S}_6^-$ of size $6$, which are not self-paired,  and
\item ${\cal S}_7^+$,  ${\cal S}_7^-$
of size $3$ which are not self-paired.
\end{enumerate}
Therefore each of the corresponding generalized orbital graphs
is a union of the graphs
$X(G,H,{\cal W})$ with
${\cal W} \in \{{\cal S}_0 ^+ \cup {\cal S}_1 ^+,
{\cal S}_0 ^- \cup {\cal S} _1 ^-, {\cal S}_2, {\cal S}_3,
{\cal S}_4^+, {\cal S}_4^-,
{\cal S}_5, {\cal S}_6, {\cal S}_7 \}$.
}
\end{example}
\medskip

{\small
\begin{table}[h!]
$$
\begin{array}{|c|c|c|c|c|c|c|}
\hline
\textrm{Row} & \W &\textrm{Conditions on } \xi&  val(X) &
\textrm{Conditions on } p & d(V_\infty) & d(V_\infty, V_x)\\
& & &  &  &   &  x\in F^*\\
\hline
\hline
1&\S_\xi &   \xi\ne \frac{1}{2},1  & p-1 & & 0 & 2\\
 &  & \xi \in S^*\cap N^*+1 \textrm{ or }& &   &&\\
& &  \xi \in N^*\cap S^*+1 &  & & & \\[1ex]
\hline
2&\S_\xi &  \xi\ne \frac{1}{2} &  p-1 & & 0 & 2\\
&  &  \xi \in N^* \cap N^*+1 &   && &\\[1ex]
\hline
3&\S_0^+\cup \S_1^+ & & p-1& & (p-1)/2 & 2 \textrm{ if } x\in S^*\\
& & && & & 0 \textrm{ if } x\in N^*\\[1ex]
\hline
4&\S_0^-\cup \S_1^- & & p-1& & (p-1)/2 & 0 \textrm{ if } x\in S^*\\
& & && & & 2 \textrm{ if } x\in N^*\\[1ex]
\hline
5&\S_{\frac{1}{2}}&  \xi=\frac{1}{2} & (p-1)/2 & p\equiv {5\pmod 8} & 0 & 1\\[1.2ex]
\hline
6&\S_{\frac{1}{2}}^+&  \xi=\frac{1}{2} & (p-1)/4 & p\equiv {1\pmod 8}
& 0 & 1 \textrm{ if } x\in S^*\\
& & && & & 0 \textrm{ if } x\in N^*\\[1ex]
\hline
7&\S_{\frac{1}{2}}^-&  \xi=\frac{1}{2} & (p-1)/4 & p\equiv {1\pmod 8}
& 0 & 0 \textrm{ if } x\in S^*\\
& & && & & 1 \textrm{ if } x\in N^*\\[1ex]
\hline
8&\S_\xi^+ &  \xi\ne \frac{1}{2},1  & (p-1)/2 && 0 & 2 \textrm{ if } x\in S^*\\
&  & \xi \in S^* \cap S^*+1&   & & & 0 \textrm{ if } x\in N^*\\[1ex]
\hline
9&\S_\xi^- &  \xi\ne \frac{1}{2},1  & (p-1)/2 && 0 & 0 \textrm{ if } x\in S^*\\
&  & \xi \in S^* \cap S^*+1 &   & & & 2 \textrm{ if } x\in N^*\\[1ex]
\hline
\end{array}
$$
\caption{\label{tab:cases}
\footnotesize The list of all basic orbital graphs
$X=X(G,H,{\cal W})$, where ${\cal W}$ is a self-paired union
of suborbits described in Theorem~\ref{th:dihed}.
In the last two columns valencies
of $d(V_\infty)$ and $d(V_\infty, V_x)$, $x\in F^*$, are
listed.  By Proposition~\ref{pro:isomorphic} graphs arising
from suborbits in Rows 6 and  7 are pairwise isomorphic. Also,
graphs arising from suborbits in Rows 8 and  9 are pairwise isomorphic.
}
\end{table}
}

The following proposition follows from \cite{PX}.
Consequently, only seven different types of basic orbital graphs
arising from the action of $\PSL(2,p)$ on cosets of $D_{p-1}$
need to be considered.

\begin{proposition}
\label{pro:isomorphic}
{\rm \cite[Table VI]{PX}}
The basic orbital graphs arising from Rows 6 and 7 of Table~\ref{tab:cases} are isomorphic, that is, $X(G,H,\S_{\frac{1}{2}}^+) \cong X(G,H,\S_{\frac{1}{2}}^-)$.
Also, the basic orbital graphs arising from Rows 8 and 9 of Table~\ref{tab:cases}
are isomorphic, that is, $X(G,H,\S_\xi^+) \cong X(G,H,\S_\xi^-)$,
where $\xi \in S^*\cap S^*+1$ and $\xi\ne\frac{1}{2},1$.
\end{proposition}

With the explicit description of the suborbits
of $G$ on ${\cal H}$ the construction of the
corresponding generalized orbital graphs
$X = X(G,H,{\cal W})$, where ${\cal W}$ is a self-paired union of suborbits
of $G$, is relatively simple.
Namely, the edge set of $X$ is precisely the set
$\{\{Hg,Hwg\}\colon g\in G, w\in {\cal W}\}$.

The description of these graphs
$X = X(G,H,{\cal W})$, where ${\cal W}$ is a self-paired union of suborbits
of $G$ is best done via a `factorization modulo'
the Sylow $p$-subgroup
$$
P=\la \mm{1}{1}{0}{1}\ra =\{\mm{1}{a}{0}{1}\colon a\in F\}.
$$
Observe that $P$ has $(p+1)/2$ orbits on $\H$. These are
\begin{eqnarray*}
V_\infty &=&\{H\mm{1}{a}{0}{1} \colon a\in F\} \ \textrm{ and }\\
V_x  &=&\{H\mm{1}{x}{-x^{-1}}{0}\mm{1}{a}{0}{1} \colon a\in F\}
= V_{-x}, \ x\in F^*.
\end{eqnarray*}
In the proofs below we will often use the fact that
$V_x$, $x\in F^*$, contains both
$$
H\mm{1}{x}{-x^{-1}}{0} \ \textrm{ and } \
H\mm{1}{-x}{x^{-1}}{0}.
$$
Note also that, using the above mentioned identification, we have
\begin{eqnarray*}
V_\infty&=&\{\infty, (1,1),(1,2),\ldots,(1,p-1)\} \ \textrm{ and }\\
V_x &=&\{(0,x),(0,-x)\}\cup\{(\xi,x), (\xi,-x) \colon \xi\in F^*\setminus\{1\}\}, x\in F^*.
\end{eqnarray*}

The generator
$$
\rho=\mm{1}{1}{0}{1}\in P
$$
is a $((p+1)/2,p)$-semiregular automorphism of $X$.
Let $\P=\{V_\infty\}\cup\{V_x\colon x\in F^*\}$ be the set of
orbits of $\rho$, and
consider the corresponding quotient graph $X_\P$
and quotient multigraph $X_\rho$ (see Subsection~\ref{ssec:semiregular}).
The following proposition gives the values of
$d(V_\infty)$ and $d(V_\infty,V_x)$, $x\in F^*$, for the basic orbital
graphs $X(G,H,\W)$ (see Theorem~\ref{th:dihed}).

\begin{proposition}
\label{lem:infty}
The valencies $d(V_\infty)$ and $d(V_\infty, V_x)$, $x\in F^*$,
for a basic orbital graph $X=X(G,H,\W)$
are as given in Table~\ref{tab:cases}.
\end{proposition}

\begin{proof}
Note that  $H \in V_\infty$ is adjacent to all the vertices
of the form $Hw$, where $w\in \W$,
which belong to $V_\infty$ if and only if its
character is of the form $(1,\eta)$, $\eta\in F^*$.
Since, by Theorem~\ref{th:dihed},
$\S_0 ^+ \cup \S_1 ^+$ and $\S_0 ^- \cup \S_1 ^-$
are the only two suborbits with nontrivial intersection with $\S_1$
we obtain the value of $d(V_\infty)$ as given in Table~\ref{tab:cases}.

To determine the values for $d(V_\infty, V_x)$ let us consider
the neighbors of $H$. Note that a representative of a coset
$Hg$ outside $V_\infty$
adjacent to $H$ is of the form
$$
\mm{1}{z}{\frac{y-1}{z}}{y} \ \textrm{ for a suitable $z\in F^*$}.
$$
Now, if this neighbor is inside the orbit $V_x$ then there exists $j\in F$ such that
$$
\mm{1}{z}{\frac{y-1}{z}}{y}=\mm{1}{x}{-x^{-1}}{0}\mm{1}{j}{0}{1}=
\mm{1}{j+x}{-x^{-1}}{-\frac{j}{x}}.
$$
Hence, recalling the equivalence
relation (\ref{eq:relation}) used
for identification of cosets and characters,
 either $(y,z)$ or $(1-y,\frac{yz}{y-1})$ is equal to $(-\frac{j}{x},j+x)$.
Now, the two equations $(y,z)=(-\frac{j}{x},j+x)$ and
$(1-y,\frac{yz}{y-1})=(-\frac{j}{x},j+x)$ give
solutions
$$
j=\frac{yz}{y-1}, x=\frac{z}{1-y} \ \textrm{ and } \
j=z, x=\frac{z}{y-1}=-\frac{z}{1-y}.
$$
Since $V_x=V_{-x}$ and $\frac{yz}{y-1}\not=y$ if $y\ne\frac{1}{2}$
we can conclude that for
$\W\in \{\S_0 ^+ \cup \S_1 ^+,\S_0 ^- \cup \S_1 ^-, \S_\xi,\S_\xi^+,
\S_\xi^-\}$ either
$V_\infty$ and $V_x$ are adjacent in $X_\rho$ with a double edge
or there is no edge between them, whereas for
$\W\in\{\S_{\frac{1}{2}},\S_{\frac{1}{2}}^+,\S_{\frac{1}{2}}^-\}$
all the edges in $X_\rho$ containing the vertex $V_\infty$ are simple edges.
Applying the conditions for suborbits given in Theorem~\ref{th:dihed}
one can now obtain the values of $d(V_\infty,V_x)$ for all possible
basic suborbits $\W$.
\end{proof}

A quasi-semiregular action is a natural
generalization of a semiregular action
(see Subsection~\ref{ssec:semiregular}).
Following~\cite{KMMM}, we say that
a group $G$ acts {\em quasi-semiregularly} on a set $V$ if there exists an element $v$ in $V$ such that $v$ is fixed by any element of $G$, and $G$ acts semiregularly on $V \setminus \{v\}$. If $G$ is nontrivial,   then $v$ is uniquely determined, and is referred to as the {\em fixed point} of $G$. A nontrivial automorphism $g$ of a graph $X$ is called {\em quasi-semiregular} if the group $\la g\ra$ acts quasi-semiregularly on $V(X)$. Equivalently, $g$ fixes a vertex and the only power $g^i$ fixing another vertex is the identity mapping.
If a group $G$ is quasi-semiregular on the vertex set of the graph with $m+1$ orbits, then the graph is called a {\em quasi $m$-Cayley graph} on $G$.
If $G$ is cyclic and quasi-semiregular with two nontrivial orbits
then the graph is said to be a {\em quasi-bicirculant}.

In the next proposition we prove that the quotient
graph $X_\P$ of a generalized orbital graph $X$ is a quasi-bicirculant.
The corresponding group automorphism is given by
a diagonal matrix whose diagonal consists of an appropriate pair
of generators of $S^*$. More precisely, let
\begin{eqnarray}
\label{eq:sigma}
\sigma=\mm{z}{0}{0}{z^{-1}}, \textrm{ where } \la z\ra=S^*.
\end{eqnarray}

\begin{proposition}
\label{pro:quasibicirculant}
Let $X=X(G,H,\W)$. Then $X_\P$ is a quasi-bicirculant.
The corresponding
quasi-semiregular action on $X_\P$ is given by the subgroup
generated by $\sigma$.
Moreover, the fixed point of this  action
is $V_\infty$ and the two nontrivial orbits are
$\OO(S^*)=\{V_x\colon x \in S^*\}$ and
$\OO(N^*)=\{V_x\colon x \in N^*\}$.
\end{proposition}

\begin{proof}
Note first that $\sigma\in H$.
Observe that for an element $\mm{1}{a}{0}{1} \in P$, $a\in F$, we have
$$
\sigma^{-1}\mm{1}{a}{0}{1}\sigma=
\mm{1}{az^{-2}}{0}{1} \in P,
$$
and consequently $\la \sigma\ra \le N_G(P)$.
It follows that for every $a\in F$ we have
$$
H\mm{1}{a}{0}{1} \sigma= H\sigma \mm{1}{az^{-2}}{0}{1} =
H\mm{1}{az^{-2}}{0}{1}\in V_\infty.
$$
We can conclude that $V_\infty$ is fixed by $\la \sigma\ra$.
Now let us look at the action of $\sigma$ on
elements from an orbit $V_x$, $x\in F^*$.
Recall that elements of $V_x$ are of the form
$$
H\mm{1}{x}{-x^{-1}}{0}\mm{1}{a}{0}{1}, \textrm{ where $a\in F$}.
$$
Applying the action of $\sigma$ on any of these elements gives
\begin{eqnarray*}
H\mm{1}{x}{-x^{-1}}{0}\mm{1}{a}{0}{1}\sigma&=&
H\mm{1}{x}{-x^{-1}}{0}\sigma\mm{1}{az^{-2}}{0}{1} \\
&=&H\mm{z}{xz^{-1}}{-x^{-1}z}{0}\mm{1}{az^{-2}}{0}{1} \in V_{xz}.
\end{eqnarray*}
It follows that $\la \sigma\ra$ cyclically permutes the orbits of $P$,
and we can represent $\sigma$ as
a permutation of the vertex set of $X_\P$ in the following way
\begin{eqnarray*}
\sigma(V_\infty)=V_\infty \ \textrm{ and }
\sigma(V_x)=V_{xz}.
\end{eqnarray*}
Since $z \in S^*$ it is now clear
that $\la\sigma\ra$ is a quasi-semiregular group of automorphisms
of $X_\P$ (as well as of $X_\rho$) with one of the nontrivial orbits consisting of $V_x$, $x \in S^*$, and the other consisting
$V_x$, $x \in N^*$.
\end{proof}

In subsequent lemmas and propositions the following observations on characters of adjacent vertices in $X_\P$ will be frequently used.
Let $X=X(G,H,\W)$,
where $\W$ is one of the basic self-paired union of suborbits given
in Rows 1, 2, 5, 6, 8 and 9 of Table~\ref{tab:cases},
 and let $V_y \in V(X_\rho)$, $y\in F^*$.
Then a representative of a coset adjacent to a coset in $V_y$
is of the form
$$
\mm{1}{\eta}{\frac{\xi-1}{\eta}}{\xi}
\mm{1}{y}{-y^{-1}}{0}=
\mm{1-\eta y^{-1}}{y}{\frac{\xi-1}{\eta}-\xi y^{-1}}{\frac{y(\xi-1)}{\eta}}, \ \textrm{for some $\eta \in F^*$},
$$
where $\xi$ determines the suborbit $\W$ (see Table~\ref{tab:cases}).
If this neighbor is inside an orbit $V_x$, $x\in F^*$, then there exists $j\in F$
such that
$$
\mm{1-\eta y^{-1}}{y}{\frac{\xi-1}{\eta}-\xi y^{-1}}{\frac{y(\xi-1)}{\eta}}\equiv
\mm{1}{x}{-x^{-1}}{0}\mm{1}{j}{0}{1}=\mm{1}{j+x}{-x^{-1}}{-jx^{-1}}.
$$
Hence, in view of the equivalence relation (\ref{eq:relation}), one can see that either
\begin{eqnarray}
\label{eq:equation}
(\frac{(\xi-1)(y-\eta)}{\eta}, \frac{y^2}{y-\eta})=(-\frac{j}{x},j+x)
 \textrm{ or }
 (\frac{(\xi-1)(y-\eta)}{\eta}, \frac{y^2}{y-\eta})= (1+\frac{j}{x}, j).
\end{eqnarray}
This gives us that
\begin{eqnarray}
\label{eq:equation-j}
j_{1,2}=\frac{1}{2}(y-x\pm\sqrt{(x+y)^2- 4xy\xi}), \,
j_{3,4}=\frac{1}{2}(y-x\pm\sqrt{(x-y)^2+4xy\xi}),
\end{eqnarray}
and that
\begin{eqnarray}
\label{eq:equation-eta}
\eta_{1,2}=y - \frac{2y^2}{y+x\pm\sqrt{(x+y)^2-4xy\xi}}, \,
\eta_{3,4}=y - \frac{2y^2}{y+x\pm\sqrt{(x-y)^2+4xy\xi}},
\end{eqnarray}

Whenever we need to compare the values of $\eta_i$
for different pairs of orbits $V_x,V_y$ and $V_z,V_w$ we
will write $\eta_i(x,y)$ and $\eta_i(z,w)$.

\begin{proposition}
\label{pro:inside}
Let $X=X(G,H,\W)$,
where $\W$ is one of the basic self-paired union of suborbits given
in  Rows 1, 2, 5, 6, 7, 8 and 9 of Table~\ref{tab:cases}.
Then  for $y \in F^*$ we have
$$
d(V_y)=\left\{
\begin{array}{cl}
0, & \xi\in N^*\cap N^*+1 \\
2, & \xi\in S^*\cap N^*+1 \textrm{ or } \xi\in N^*\cap S^*+1\\
4, & \xi\in S^*\cap S^*+1 \\
0, & \xi=1/2 \textrm{ and } p\equiv {5\pmod 8}\\
2, & \xi=1/2 \textrm{ and } p\equiv {1\pmod 8}\\
\end{array}\right.,
$$
where  $\xi$ indicates the subscript
$\xi$ at $\S_\xi^\epsilon$,
$\epsilon=\pm 1$, in self-paired union  $\W$, as given
in Table~\ref{tab:cases}.
\end{proposition}

\begin{proof}
Observe that $d(V_y)$ is given by the number of solutions of the equations
given in (\ref{eq:equation-j}) and (\ref{eq:equation-eta}) for $x=y$.
\end{proof}

In the next five subsections we prove the existence
of Hamilton cycles for the basic orbital graphs
arising from group actions given in Table~\ref{tab:cases}.


\subsection{Case  $\S_\xi$ with $\xi\ne\frac{1}{2},1$ (Rows 1 and 2
of Table~\ref{tab:cases})}
\label{subsec:12}
\noindent

\begin{proposition}
\label{pro:row-1-2}
Let $X=X(G,H,\W)$, where $\W$ is one of the basic self-paired unions of suborbits given in Rows 1 and 2 of Table~\ref{tab:cases}. Then $X$ is hamiltonian.
\end{proposition}

\begin{proof}
Observe first that Proposition~\ref{lem:infty} implies that
$d(V_\infty)=(p-1)/2$ and $d(V_\infty,V_x)=2$ for every $x\in F^*$.
Note also that $d(V_x,V_y)\le 4$ for every pair $x,y\in F^*$,
and that, by Proposition~\ref{pro:inside},  $d(V_x)\le 4$, $x\in F^*$.
It follows that each $V_x$ is joined to at least $(p-5)/4$ vertices
in $X_\P-\{V_\infty\}$. Namely, subtracting from the valency
$p-1$ of $X$ the valency $d(V_x,V_y)$ and the maximal possible
valency $d(V_x)\le 4$ of $V_x$
we are left with at least $p-7$ additional edges in $X_\rho$
incident with $V_x$. But  $p\equiv {1\pmod 4}$, and so
every vertex in $X_\P-\{V_\infty\}$ is of valency at least $(p-5)/4$.

If $X_\P-\{V_\infty\}$ is regular then,
since $p\ge 13$, we apply Proposition~\ref{pro:jack}
to get a Hamilton cycle in   $X_\P-\{V_\infty\}$.
Since  $d(V_\infty,V_x)=2$ for every $x\in F^*$, we
can extend this Hamilton cycle in $X_\P-\{V_\infty\}$
to a Hamilton cycle in $X_\P$, and then apply Lemma~\ref{lem:cyclelift}
to conclude that $X$ is hamiltonian.

We may therefore assume that $X_P-\{V_\infty\}$ is not regular.
Without loss of generality we may assume that
$val_{X_\P}(V_x) > val_{X_\P}(V_y)$ for $x\in S^*$ and $y\in N^*$.
(Recall that, by Proposition~\ref{pro:quasibicirculant},  there exists
an automorphism of $X_\rho$ which cyclically permutes
vertices in the set $\OO(S^*)=\{V_x\colon x\in S^*\}$  and
vertices in the set $\OO(N^*)=\{V_y\colon y\in N^*\}$. Consequently,
vertices inside each of these two sets
are of the same valency.)

Since we are in the case $\S_\xi=\S_\xi^+\cup\S_\xi^-$
the solutions of (\ref{eq:equation-j}) and (\ref{eq:equation-eta})
for $x=y$ depends solely on $\xi$. Consequently,
$d(V_x)=d(V_y)$ for all $x,y\in F^*$.
Suppose first that  $d(V_x)=0$, $x\in F^*$.
Note that this happens in Row 2 of Table~\ref{tab:cases}.
Combining together the additional
facts that $d(V_\infty,V_x)=2$ for every $x\in F^*$,
 that $p\equiv {1\pmod 4}$ and that the valency is an integer,
we may assume that the valency of a vertex $V_x$, $x\in F^*$,
in $X_\P$ satisfies
$$
val_{X_\P}(V_x)\ge\left\{
\begin{array}{cl}
(p+3)/4,& x\in N^*  \\
(p+7)/4, & x\in S^*
\end{array}\right.,
$$
and so the existence of a Hamilton cycle
in $X_\P$ follows by Proposition~\ref{pro:chvatal}.
Namely, the conditions of Proposition~\ref{pro:chvatal} are vacuously
satisfied since $S_i=\emptyset$ for every $i < |V(X_\P)|/2$.
Clearly this Hamilton cycle  contains at least one double edge
in $X_\rho$ (for example, all edges incident with $V_\infty$ are such
double edges), Lemma~\ref{lem:cyclelift} implies that $X$ is hamiltonian.

In the remaining case of Row~1 of Table~\ref{tab:cases}
we have, in view of (\ref{eq:equation-j}), that $d(V_x)=2$, $x\in F^*$. Then
using the same arguments as in the previous
case, we may assume that the valency of a vertex $V_x$, $x\in F^*$,
in $X_\P$ satisfies
$$
val_{X_\P}(V_x)\ge \left\{
\begin{array}{cl}
(p-1)/4,& x\in N^*  \\
(p+3)/4, & x\in S^*
\end{array}\right.,
$$
and so the existence of a Hamilton cycle
in $X_\P$ now follows by Proposition~\ref{pro:chvatal}.
Namely,  $i=(p-1)/4$ is the only $i < |V(X_\P)|/2=(p+1)/4$
for which we have to check whether the conditions of
Proposition~\ref{pro:chvatal} are satisfied.
Clearly $|S_{(p-1)/4}|\not\le (p-1)/4-1$.
But $|S_{(p+1)/2-(p-1)/4-1}|=|S_{(p-1)/4}|\le (p-1)/4$,
and so Proposition~\ref{pro:chvatal} indeed applies.
As in the previous paragraph
this Hamilton cycle clearly contains at least one double edge
in $X_\rho$  and so
Lemma~\ref{lem:cyclelift} implies that $X$ is hamiltonian.
\end{proof}


\subsection{Cases  $\S_0^+\cup S_1^+$ and $\S_0^-\cup S_1^-$ (Rows 3 and 4
of Table~\ref{tab:cases})}
\label{subsec:34}
\noindent

\begin{proposition}
\label{pro:row-3-4}
Let $X=X(G,H,\W)$, where $\W\in\{\S_0^+\cup\S_1^+, \S_0^-\cup\S_1^-)$,
be  one of the graphs arising from Rows 3 or 4 of Table~\ref{tab:cases}.
Then $X$ is hamiltonian.
\end{proposition}

\begin{proof}
Suppose  that $\W=\S_0^+\cup\S_1^+$.
By Proposition~\ref{lem:infty} we have
$d(V_\infty)=(p-1)/2$ and
$$
d(V_\infty,V_x)=\left\{
\begin{array}{cl}
2, & x\in S^*\\
0, & x\in N^*
\end{array}\right..
$$

We now need to compute the valency $val_{X_\P-\{V_\infty\}}(V_y)$
of $V_y$ in the graph $X_\P-\{V_\infty\}$.
The character of a coset in $\S_0^+\cup\S_1^+$
is either of the form $(0,\eta)$ or $(1,\eta)$, $\eta\in S^*$,
with respective representatives
$$
E_1=\mm{1}{\eta}{-\eta^{-1}}{0} \textrm{ and } E_2=\mm{1}{\eta}{0}{1}.
$$
Then a representative of a coset adjacent to a coset in $V_y$
is either of the form
$$
E_1\cdot
\mm{1}{y}{-y^{-1}}{0} \textrm{ or of the form }
E_2\cdot
\mm{1}{y}{-y^{-1}}{0}.
$$
Further, if this neighbor is inside an orbit $V_x$, $x\in F^*$,
then there exists $j\in F$
such that either
$$
E_1\cdot
\mm{1}{y}{-y^{-1}}{0}\equiv
\mm{1}{x}{-x^{-1}}{0}\mm{1}{j}{0}{1}
\textrm{ or }
E_2\cdot
\mm{1}{y}{-y^{-1}}{0}\equiv
\mm{1}{x}{-x^{-1}}{0}\mm{1}{j}{0}{1}.
$$
With the equivalence relation (\ref{eq:relation}) in mind, one can see
that the four cases given in Table~\ref{tab:01+} arise.
Let $\eta_i$ be the possible solution for $\eta$ given in the $i$-th row  of
Table~\ref{tab:01+}.
Then $\eta_1\eta_4=\eta_2\eta_3=y^2$. This implies that
either $(\eta_1,\eta_4)\in S^*\times S^*$ or
$(\eta_1,\eta_4)\in N^*\times N^*$, and that
either $(\eta_2,\eta_3)\in S^*\times S^*$ or
$(\eta_2,\eta_3)\in N^*\times N^*$.

\begin{table}[h!]
$$
\begin{array}{|c|c|c|}
\hline
\textrm{Row} & j & \eta \\
\hline
\hline
1&y-x & xy/(x-y)\\[1ex]
\hline
2&y & xy/(x+y)\\[1ex]
\hline
3&-x & y(x+y)/x\\[1ex]
\hline
4&0 & y(x-y)/x\\[1ex]
\hline
\end{array}
$$
\caption{\label{tab:01+}  Conditions on $j$ and $\eta$ for existence
of an edge between $V_x$ and $V_y$, $x,y\in F^*$, in
$X=X(G,H,\W)$, where $\W=\S_0^+\cup\S_1^+$.}
\end{table}

For $x=y$ Rows 1 and 4 of Table~\ref{tab:01+}
give no solution. Hence
$\eta_2=x/2$ and $\eta_3=2x$ are the only solutions,
and thus for $x\in S^*$ we have
$$
d(V_x)= 2 \Leftrightarrow 2\in S^* \Leftrightarrow p\equiv {1\pmod 8},
$$
and for $x\in N^*$ we have
$$
d(V_x)= 2 \Leftrightarrow 2\in N^* \Leftrightarrow p\equiv {5\pmod 8}.
$$
We now compute the valencies of the vertices in the quotient graph $X_\P$.
Since $X_\P$ is a quasi-bicirculant, in order to
compute the valency of an arbitrary vertex
$V_x\in\OO(S^*)$ it suffices to compute the valency $val_{X_\P}(V_1)$
of the vertex $V_1$
instead of computing the valency of an arbitrary vertex
$V_x\in\OO(S^*)$
 (in short, we may assume that $x=1$).
Since $\W=S_0^+\cup S_1^+$ we must have $\eta \in S^*$.
It follows from Table~\ref{tab:01+} that if $V_y\in\OO(S^*)$
is adjacent to $V_1$ in $X_\P$ then $d(V_1,V_y)\in\{2,4\}$, and so
either $y\in S^* \cap S^*+1$ or $y\in S^* \cap S^*-1$.
Note that for $p\equiv {1\pmod 8}$ we have $\pm 1\in
(S^* \cap S^*+1) \cup (S^* \cap S^*-1)$, and for
$p\equiv {5\pmod 8}$ this union does not contain $\pm 1$.
Therefore, combining together
Propositions~\ref{pro:preseki} and~\ref{pro:AB}
we conclude that
there exist at least
$(|S^* \cap S^*+1|+2-2)/2 = (p-5)/8$ vertices in $\OO(S^*)$
adjacent to $V_1$ if $p\equiv {1\pmod 8}$, and that there are at least
$(|S^* \cap S^*+1|+2)/2 = (p+3)/8$ vertices in $\OO(S^*)$
adjacent to $V_1$ if $p\equiv {5\pmod 8}$.
(Namely, the number of possible
number-theoretic solutions
in Propositions~\ref{pro:preseki} and~\ref{pro:AB}
needs to be divided by $2$ because of the fact that the
equivalence relation (\ref{eq:relation}) implies $V_y=V_{-y}$.)
Moreover, $(p-5)/8$ is not an integer if $p\equiv {1\pmod 8}$
and so  there are at least $(p-1)/8$ vertices in $\OO(S^*)$
adjacent to $V_1$ when $p\equiv {1\pmod 8}$. In short,
$$
val(V_1)_{X_\P\la \OO(S^*)\ra}=
\left\{
\begin{array}{cr}
(p-1)/8, &  p\equiv {1\pmod 8}\\
(p+3)/8, & p\equiv {5\pmod 8}
\end{array}\right..
$$
Similarly, using Propositions~\ref{pro:preseki} and~\ref{pro:AB} for
neighbors of $V_1$ in $\OO(N^*)$ we can see that their number is
$((p-1)/4+2)/2=(p+7)/8$ which is, after integer correction, equal to
$$
val(V_1)_{X_\P-\{V_\infty\}} - val(V_1)_{X_\P\la \OO(S^*)\ra} =
\left\{
\begin{array}{cr}
(p+7)/8, &  p\equiv {1\pmod 8}\\
(p+11)/8, & p\equiv {5\pmod 8}
\end{array}\right..
$$
It remains to calculate the valency of the subgraph
$X_\P\la \OO(N^*)\ra$ of $X_\P$ induced
on $\OO(N^*)$. Let $V_x\in\OO(N^*)$ be a fixed
vertex. If $V_y\in\OO(N^*)$ is adjacent to this vertex $V_x$ then
we must have $x\pm y \in S^*$, and thus the number of vertices
in $\OO(N^*)$ adjacent to $V_x$ depends on
the cardinality of the set
$$
(S^*\cap N^*+x) \cup (S^*\cap N^*-x),
$$
which is equal to the cardinality of the set
$
(N^*\cap S^*+1) \cup (N^*\cap S^*-1)$, and  also of the set
$(N^*-1\cap S^*) \cup (N^*+1\cap S^*)$.
Thus, since $\pm 1\in N^*-1\cap S^*$
for $p\equiv {5\pmod 8}$,
Propositions~\ref{pro:preseki} and~\ref{pro:ABn}
combined together imply that
$$
val(V_x)_{X_\P\la \OO(N^*)\ra}=
\left\{
\begin{array}{cr}
(p+7)/8, &  p\equiv {1\pmod 8}\\
(p-5)/8, & p\equiv {5\pmod 8}
\end{array}\right..
$$
If follows that, with the exception of $V_\infty$
which is of valency $(p-1)/4$, all other vertices in $X_\P$
are of valency at least $(p+3)/4$
and so more than half of the order of $X_\P$.
Now Proposition~\ref{pro:chvatal}
implies the existence of a Hamilton cycle in $X_\P$.
Namely, with the corresponding notation  for the sets $S_i$ we have
 $|S_{(p-1)/4}|=1\le (p-1)/4-1$.
This Hamilton cycle in $X_\rho$
clearly has double edges,  and so,
by Lemma~\ref{lem:cyclelift}, it lifts
to a Hamilton cycle in $X$.

The hamiltonicity of the graph $X$ for $\W=\S_0^-\cup\S_1^-$
is determined in  an analogous way. We omit the details.
\end{proof}


\subsection{Case $\S_\frac{1}{2}$ (Row 5 of Table~\ref{tab:cases})}
\label{subsec:5}
\noindent

\begin{proposition}
\label{pro:row-5}
Let $p\equiv{5\pmod 8}$
and let $X=X(G,H,\S_{\frac{1}{2}})$
be the graph  arising from
Row 5 of Table~\ref{tab:cases}.
Then $X$ is hamiltonian.
\end{proposition}

\begin{proof}
Note that $X$ is of valency $(p-1)/2$.
By Proposition~\ref{lem:infty} we have
$d(V_\infty)=0$ and
$d(V_\infty,V_y)=1$ for every $y\in F^*$.

We now need to compute the valency $val_{X_\P-\{V_\infty\}}(V_y)$,
$y\in F^*$. Let $x\in F^*$.
The number of edges $d(V_y,V_x)$ between $V_y$ and $V_x$ in $X_\rho$
is obtained from  (\ref{eq:equation-j}) and (\ref{eq:equation-eta}) by
letting  $\xi=1/2$. We obtain
$$
j_{1,2}=j_{3,4}=\frac{1}{2}(y-x\pm\sqrt{x^2+y^2})
$$
and
$$
\eta_{1,2}=\eta_{3,4}=y-\frac{2y^2}{y+x\pm\sqrt{x^2+y^2}}
=\frac{y}{x}(\pm \sqrt{x^2+y^2}-y).
$$
By Proposition~\ref{pro:2} it follows that
$2\in N^*$, and so we have $d(V_x)=0$ for every $x\in F^*$.
Further, if $x^2+y^2=0$ then $y^2=-x^2$, and so
$y=\pm\sqrt{-1}x$. Since $\sqrt{-1} \in N^*$ when $p\equiv {5\pmod 8}$
it follows that for every $x\in S^*$ there exists a unique $y\in N^*$ such that
$d(V_x,V_y)=1$, whereas
all other edges in $X_\rho-\{V_\infty\}$ containing $V_x$
are double edges. It follows that for every $x\in F^*$ we have
$$
val(V_x)=(\frac{p-1}{2}-2)/2+2=\frac{p-5}{4}+2=\frac{p+3}{4}.
$$
Since $val_{X_\P}(V_\infty)=(p-1)/2$ it follows that
all the vertices in $X_\P$ are of valency more than half of the order
of $X_\P$, and so Proposition~\ref{pro:chvatal} implies the existence of a
Hamilton cycle in $X_\P$. Since  this cycle clearly
contains a double edge,
Lemma~\ref{lem:cyclelift} implies that $X$ is hamiltonian.
\end{proof}


\subsection{Cases $\S_\frac{1}{2}^+$ and $\S_\frac{1}{2}^-$ (Rows 6 and 7
of Table~\ref{tab:cases})}
\label{subsec:67}
\noindent

In this and the next subsection Hamilton cycles are constructed
using the results  from Section~\ref{sec:polynomials} about polynomials
of degree $4$ that represent quadratic residues  at primitive roots.

\begin{proposition}
\label{pro:row-6-7}
Let $p\equiv{1\pmod 8}$
and let $X=X(G,H,\W)$, where $\W\in\{S_\frac{1}{2}^+,S_\frac{1}{2}^-\}$,
be  one of the graphs arising from  Row 6 or 7
of Table~\ref{tab:cases}.
Then $X$ is hamiltonian.
\end{proposition}

\begin{proof}
By Proposition~\ref{pro:isomorphic} the two graphs are isomorphic, and so
we may assume that $\W=S_{\frac{1}{2}}^+$.
Note that $X$ is of valency $(p-1)/4$.
Since $(p+1)/2$ is not a prime for $p=17$
we may also assume that $p>17$.

By Proposition~\ref{lem:infty} we have
$d(V_\infty)=0$ and
$d(V_\infty,V_x)=1$ for every $x\in S^*$.
We now need to compute the valency $val_{X_\P-\{V_\infty\}}(V_y)$,
$y\in F^*$. Let $x\in F^*$.
The number of edges $d(V_y,V_x)$ between $V_y$ and $V_x$ in $X_\rho$
is obtained from  (\ref{eq:equation-j}) and (\ref{eq:equation-eta}) by
letting  $\xi=1/2$. We obtain
\begin{eqnarray}
\label{eq:new}
j_{1,2}=j_{3,4}=\frac{1}{2}(y-x\pm\sqrt{x^2+y^2})
\textrm{ and }
\eta_{1,2}=\eta_{3,4}=\frac{y}{x}(\pm \sqrt{x^2+y^2}-y).
\end{eqnarray}
By Proposition~\ref{pro:2},
$2\in S^*$, and so $d(V_x)=2$ for every $x\in F^*$ for which
$\eta_{1,2}=\eta_{3,4}= \pm x(\sqrt{2}-1)\in S^*$.
Further, if $x^2+y^2=0$ then $y^2=-x^2$, and so
$y=\pm\sqrt{-1}x$ and $\eta_{1,2}=\mp\sqrt{-1}x$.
Since $\sqrt{-1} \in S^*$ for $p\equiv {1\pmod 8}$,
it follows that for every $x\in S^*$ there exits a unique
$V_y\in X_\rho\la \OO(S^*)\ra$ such that
$d(V_x,V_y)=1$.
All other edges in $X_\rho-\{V_\infty\}$ incident with $V_x$
are double edges.
Furthermore, also all  of the edges
in $X_\rho-\{V_\infty\}$ incident with $V_x$,
$x\in N^*$, are double edges.

Suppose that $x=1$. Then we
get from (\ref{eq:new}) that
$$
j_{1,2}=j_{3,4}=\frac{1}{2}(y-1\pm\sqrt{1+y^2})
\textrm{ and }
\eta_{1,2}=\eta_{3,4}
=y(\pm \sqrt{1+y^2}-y).
$$
Let us now consider elements of the form $1+g^4\in F^*$, where
$g\in F^*$ is a generator of $F^*$, that is, $F^*=\la g\ra$.
Since $p\equiv {1\pmod 8}$, combining together
Theorem~\ref{the:polynomial-degree-4} and Proposition~\ref{pro:small}
we have
that there always exists
$g\in F^*$ such that $F^*=\la g\ra$ and $1+g^4\in S^*$.

The element $s=g^2$ generates $S^*$.
We claim that either  $V_1$ is adjacent to $V_s$
or $V_g$ is adjacent to $V_{sg}$. This will in turn imply that
there is a full cycle either in the induced graph on
$\OO(S^*)$
or in the induced graph on $\OO(N^*)$.
The corresponding values $\eta_1$ and $\eta_2$ for the pairs
$V_1,V_s$ and $V_g, V_{gs}$ are, respectively,
$$
\eta_{1,2}(1,s)=s(\pm \sqrt{1+s^2} - s) \textrm{ and }
\eta_{1,2}(g,gs)=gs(\pm \sqrt{1+s^2} - s).
$$
Therefore
$$
\frac{\eta_{i}(g,gs)}{\eta_{i}(1,s)}=g\in N^*, i\in\{1,2\}.
$$
And consequently, the bicirculant $X_\rho-\{V_\infty\}$ indeed has
a full induced cycle $C$ either on the orbit $\OO(S^*)$ or
on the orbit $\OO(N^*)$.
More precisely, this cycle is induced by an edge inside
one of the two orbits of $\sigma$ and the action of $\sigma$
on this edge.
(Recall that $\OO(S^*)$ and $\OO(N^*)$
are  the two orbits of the quasi-semiregular automorphism $\sigma$
from Proposition~\ref{pro:quasibicirculant}).

We claim that $X_\rho-\{V_\infty\}$ contains a subgraph
isomorphic to a generalized Petersen graph
$GP((p-1)/4,k)$ for some $k\in\ZZ_{(p-1)/4}$.
In order to prove this we need to show first that the orbit that does not
contain $C$ contains a cycle or a union of cycles induced
by the action of $\sigma$ on an edge in this orbit and second
that the bipartite graph between these two orbits contains a matching
preserved by $\sigma$.
The latter holds because, by Proposition~\ref{lem:infty},
vertex $V_\infty$ is adjacent to all vertices in
$\OO(S^*)$ and no vertex in $\OO(N^*)$, and consequently
the connectedness of $X$ implies the existence of
at least one edge with one endvertex in $\OO(S^*)$
and one endvertex in $\OO(N^*)$. The action of $\sigma$
on this edge gives us the desired matching.
For the former, there are four possibilities
depending on whether  the full cycle is in
$\OO(S^*)$  or in  $\OO(N^*)$ and on whether
 $d(V_x)=2$ for each $x\in S^*$
or for each $x\in N^*$. A quick analysis based on
the valency conditions in the graph $X$ shows that
such a collection of cycles always exists with one exception only.
This exception occurs when
the full cycle is in  $\OO(N^*)$ and $d(V_x)=2$ for $x\in S^*$.
In this case, however, we can apply Proposition~\ref{pro:edges} to see
that $V_1$ is adjacent to $V_x$ for some $\sqrt{-1}\ne x\in S^*$,
and so the corresponding union of cycles induced by the action of $\sigma$
is the collection of cycles we were aiming for. Consequently,
$X_\rho-\{V_\infty\}$ contains the desired generalized
Petersen graph as a subgraph also in this case.
Since $\frac{p-1}{4}$ is even, Proposition~\ref{pro:GPG} implies that
$X_\rho-\{V_\infty\}$ contains a Hamilton cycle.
Of course, this cycle contains an edge of the form
$V_xV_y$, $x,y\in S^*$.
Replacing this edge with a path $V_xV_\infty V_y$ gives
a Hamilton cycle in $X_\rho$.
Obviously this Hamilton cycle contains double edges in $X_\rho$,
and so, by Lemma~\ref{lem:cyclelift},
lifts to a Hamilton cycle in $X$.
\end{proof}


\subsection{Cases  $\S_\xi^+$ and $\S_\xi^-$
with $\xi\ne \frac{1}{2},1$ (Rows 8 and 9 of Table~\ref{tab:cases})}
\label{subsec:89}
\noindent

In Proposition~\ref{pro:row-8-9} graphs arising from
Rows~8 and 9 of Table~\ref{tab:cases} are considered.
Before stating this proposition we cover three exceptional cases
for which the results about polynomials from Section~\ref{sec:polynomials}
cannot be fully applied. The three  exceptional pairs
$(p,\xi)$ are $(13,10)$, $(37,12)$, and $(61,57)$,
see Proposition~\ref{pro:small}.

\begin{example}
\label{ex:13}
{\rm
Let $X$ be a basic orbital graph
arising   from the action of $G=\PSL(2,13)$ on the cosets
of $D_{p-1}=D_{12}$ from Row~8  or 9 of Table~\ref{tab:cases}.
Because of the isomorphism given in Proposition~\ref{pro:isomorphic}
we may assume that $X$ arises from Row~8  of Table~\ref{tab:cases},
that is, it is associated with a suborbit
$\S_\xi^+$,
where $\xi \in S^*\cap S^*+1$, $\xi\ne \frac{1}{2}$ and $\xi\ne 1$.

For $p=13$ we have
\begin{eqnarray*}
(S^*\cup\{0\})\cap (S^*\cup\{0\})+1 &=& \{0, 1, 4, 10 \} =\{0, 1 \} \cup \{a, a^{-1}\colon a=4\}.
\end{eqnarray*}
There are two self-paired suborbits of length $6$,
giving, up to isomorphism, one vertex-transitive graph
of order $13\cdot 7=91$ and of valency $6$.
(For example, this can be checked  using Magma \cite{Mag}.)
The below matrix gives the symbol of this graph with respect to
the orbits $S_i=\{v_i^j\colon j\in\mathbb{Z}_{13}\}$, $i\in\mathbb{Z}_7$,
of a $(7,13)$-semiregular automorphism:
$$
\left[
\begin{array}{ccccccc}
\emptyset & \{0,2\} & \{6,12\} & \{1,9\} & \emptyset & \emptyset & \emptyset\\
\{0,11\} & \{\pm 2\} & \emptyset & \emptyset & \{0,3\} & \emptyset & \emptyset \\
\{1,7\} & \emptyset & \{\pm 6\} & \emptyset & \emptyset & \{0,4\} & \emptyset\\
\{4,12\} & \emptyset & \emptyset  & \{\pm 5\} & \emptyset & \emptyset  & \{7,8\}\\
\emptyset  & \{0,10\} & \emptyset & \emptyset  & \{\pm 3\} & \{5\} & \{10\}\\
\emptyset  & \emptyset & \{0,9\} & \emptyset & \{8\} & \{\pm 4\} & \{8\} \\
\emptyset & \emptyset & \emptyset  & \{5,6\} & \{3\} & \{5\} &\{\pm 1\}
\end{array}\right]
$$
Note that there is no Hamilton cycle in $X_\rho$ (see also Figure~\ref{fig:13}).
There however exists a cycle $S_0S_2S_5S_6S_3S_0$ in $X_\rho$
that lifts to a
$65$-cycle in $X$ containing the edge $v_5^5v_6^0$, and
there exists a $26$-cycle containing all
the vertices in the orbits $S_1$ and $S_4$ and containing the edge
$v_4^0v_4^3$. Replacing the edges $v_5^5v_6^0$ and $v_4^0v_4^3$
in these two cycles
with the edges $v_4^0v_5^5$ and $v_4^3v_6^0$ gives a Hamilton cycle in $X$.
}
\end{example}

\begin{figure}[h!]
\begin{footnotesize}
\begin{center}
 \includegraphics[width=0.4\hsize]{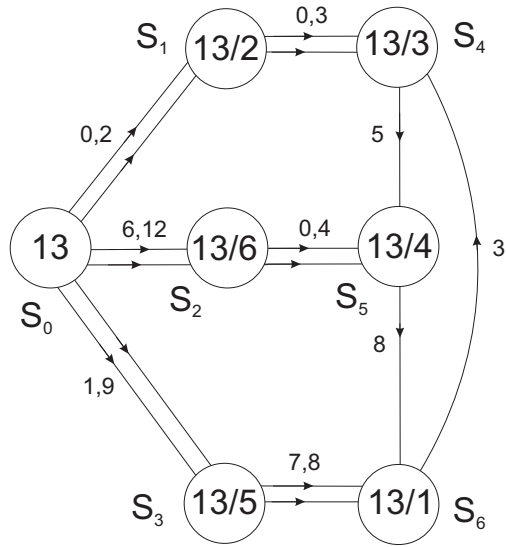}
\caption{\label{fig:13}\footnotesize
The orbital graph arising from the action of $G=\PSL(2,13)$ on the cosets
of $D_{p-1}=D_{12}$ with respect to a suborbit $\S_\xi^+$
where $\xi \in S^*\cap S^*+1$, $\xi\ne \frac{1}{2}, 1$, given in Frucht's notation with respect to
a $(7,13)$-semiregular automorphism $\rho$.}
\end{center}
\end{footnotesize}
\end{figure}

\begin{example}
\label{ex:37}
{\rm
Let $X$ be a basic orbital graph
arising   from the action of $G=\PSL(2,37)$ on the cosets
of $D_{p-1}=D_{36}$
from Row~8  or 9 of Table~\ref{tab:cases}.
Because of the isomorphism given in Proposition~\ref{pro:isomorphic}
we may assume that $X$ arises from Row~8  of Table~\ref{tab:cases},
that is, it is associated with a suborbit
$\S_\xi^+$,
where $\xi \in S^*\cap S^*+1$, $\xi\ne \frac{1}{2}$ and $\xi\ne 1$.

There are $8$ self-paired suborbits of length $(p-1)/2=18$,
giving $4$ non-isomorphic vertex-transitive graphs
of order $37\cdot 19=703$ and of valency $18$.
(For example, this can be checked  using Magma \cite{Mag}.)
We may assume that $X$ is one of these graphs.

For $p=37$ we have
\begin{eqnarray*}
(S^*\cup\{0\})\cap (S^*\cup\{0\})+1
&=& \{0, 1, 4, 10, 11, 12, 26,27,28,34 \} \\
&=&\{0, 1 \} \cup \{a, a^{-1}\colon a\in\{4,10,11,12\}\}.
\end{eqnarray*}
The set of primitive roots in $F_{37}$  equals
$\mathcal{R}=\{2,5,13,15,17,18,19,20,22,24,32, 35\}$.
It follows by (\ref{eq:equation-j}) and (\ref{eq:equation-eta})
that vertices $V_1$ and $V_x$, $x\in F_{37}^*$, are adjacent
in $X_\P$ if and only if
$$
x^2+2(1-2\xi)x+1 \in S^*\cup\{0\} \textrm{ and }
\eta_{1,2}(1,x)=1 - \frac{2}{1+x\pm\sqrt{x^2+2(1-2\xi)x+1}} \in S^*
$$
or
$$
x^2-2(1-2\xi)x+1 \in S^*\cup\{0\} \textrm{ and }
\eta_{3,4}(1,x)=1 - \frac{2}{1+x\pm\sqrt{x^2-2(1-2\xi)x+1}} \in S^*.
$$
Given an arbitrary $\tau\in \mathcal{R}$, it follows that
$V_1$ is adjacent to $V_{\tau^2}$ if and only if
either
$$
\tau^4+2(1-2\xi)\tau^2+1 \in S^*\cup\{0\}
\textrm{ and }  \eta_{1,2}(1,\tau^2)\in S^*
$$
or
$$
\tau^4-2(1-2\xi)\tau^2+1 \in S^*\cup\{0\} \textrm{ and }
\eta_{3,4}(1,\tau^2)\in S^*.
$$
We apply Proposition~\ref{pro:small}
to conclude that for
polynomials $f_{1,2}(x)=x^4\pm 2(1-2\xi)x^2+1$
there exists $\tau\in \mathcal{R}$ such that
$f_j(\tau)\in S^*\cup\{0\}$ either for $j=1$ or for $j=2$.

Let $\mathcal{T}^+$ be the subset
of $\mathcal{R}$ consisting of all those primitive roots $\tau$
for which $f_1(\tau)\in S^*\cup\{0\}$, and let
$\mathcal{T}^-$ be the subset
of $\mathcal{R}$ consisting of all those primitive roots $\tau$
for which $f_2(\tau)\in S^*\cup\{0\}$.
Table~\ref{tab:37} gives the list of elements in
$\mathcal{T}^+$ and $\mathcal{T}^-$
for each   $\xi$. Further, for each
element in $\mathcal{T}^+ \cup \mathcal{T}^-$
this table also gives information on whether $\eta_i(1,\tau^2)$ belongs
to $S^*$ or not.
Checking Table~\ref{tab:37} one can see that for every $\xi\notin\{10,28\}$
there is at least one $\tau\in \mathcal{R}$ such that $V_1$ is adjacent to
$V_{\tau^2}$. We conclude that there is a full
cycle in $X\la\OO(S^*)\ra$ preserved by the automorphism $\sigma$.
Note that $X_\P$ is a quasi-bicirculant with  $V_\infty$
as the fixed vertex,
and that by Table~\ref{tab:cases}, vertex $V_\infty$ is
adjacent only to vertices in $\OO(S^*)$.
Therefore, connectedness of $X$ implies that
the bipartite graph induced by the edges with one endvertex
in
$\OO(S^*)$ and the other in $\OO(N^*)$
contains a matching preserved by the $\sigma$.
Since, by  (\ref{eq:equation-eta}), we have
$$
\frac{\eta_{i}(\tau,\tau^3)}{\eta_{i}(1,\tau^2)}=\tau\in N^*, i\in\{1,2,3,4\},
$$
Table~\ref{tab:37} also implies that
vertices in $X\la\OO(N^*)\ra$ are of
valency at least $2$. Namely, for every $\xi$
there exist at least one $\tau\in \mathcal{T}^+\cup\mathcal{T}^-$
such that $\eta_i(1,\tau)\notin S^*$ which implies that
$\eta_i(\tau,\tau^3)\in S^*$ and thus $V_\tau$ is adjacent to $V_{\tau^3}$.
Therefore, for $\xi\notin\{10,28\}$
the bicirculant $X_\rho-\{V_\infty\}$
contains a generalized Petersen graph as a subgraph.
Hence, combining together Propositions~\ref{pro:GPG} and \ref{pro:GPG2}
we have that
$X_\rho-\{V_\infty\}$ contains a Hamilton cycle or at least a
Hamilton path with both endvertices in $\OO(S^*)$.
Since $V_\infty$ is adjacent to all  vertices in
$\OO(S^*)$ we can clearly extend this cycle/path to
a Hamilton cycle in $X_\P$ containing at least one double edge
in $X_\rho$. By Lemma~\ref{lem:cyclelift}
this cycle lifts to a Hamilton cycle in $X$.

We are left with the last two cases $\xi\in\{10,28\}$.
In both cases Table~\ref{tab:37} implies that there is a full cycle in
$\OO(N^*)$.  Namely,  $\eta_{1,2}(1,\tau^2)\not\in S^*$
implies that  $\eta_{1,2}(\tau,\tau^3)\in S^*$, and so $V_\tau$
is adjacent to $V_{\tau^3}.$
Further, for $\xi=10$  we have
$10^2-2(1-2\xi)\cdot 10+1=0$ and $\eta_{3,4}(1,10)=21\in S^*$, and so
 $V_1$ is adjacent to $V_{10}$ with a single edge.
Similarly, for  $\xi=28$ we have
$11^2-2(1-2\xi)\cdot 11+1=0$ and $\eta_{3,4}(1,11)=7\in S^*$, and so
 $V_1$ is adjacent to $V_{11}$
with a single edge. Since $10,11\in S^*$ and their squares
are not $\pm 1$, we can conclude that
 vertices in
$X\la\OO(S^*)\ra$ are of valency
at least $2$ also in these  two cases.
Therefore, combining together Propositions~\ref{pro:GPG} and \ref{pro:GPG2},
there exists a Hamilton cycle/path in $X_\rho-\{V_\infty\}$
also for $\xi\in\{10,28\}$.
As before this cycle/path can be extended to a Hamilton
cycle in $X_\rho$ which, by Lemma~\ref{lem:cyclelift},
lifts to a Hamilton cycle in $X$.
}
\end{example}

\begin{table}[h!]
{\footnotesize
$$
\begin{array}{|c|c|c|c|c|}
\hline
\xi & \tau\in\mathcal{T}^+ &
\tau\in\mathcal{T}^+ \textrm{ s.t. } &
\tau\in\mathcal{T}^- & \tau\in\mathcal{T}^- \textrm{ s.t. } \\
 & & \eta_{1,2}(1,\tau^2) \in S^*& & \eta_{3,4}(1,\tau^2)\in S^* \\
\hline\hline
4 & \pm 13,\pm 17 & \pm 13,\pm 17 & \pm 2,\pm13,\pm17,\pm18& \pm 2,\pm 18\\
10 & \pm 5,\pm 15 & - & \pm 2,\pm 18& -\\
11& \pm 2, \pm 13, \pm 17\pm 18&  \pm 13,\pm 17 &\pm 2,\pm 5, \pm 15, \pm 18& \pm 5,\pm 15\\
12 & - & - & \pm 2, \pm 5, \pm 15,\pm 18& \pm 2,\pm 18\\
26 & \pm 2,\pm 5, \pm 15,\pm 18& \pm 2, \pm 18 &-& - \\
27 & \pm 2,\pm 5, \pm 15,\pm 18& \pm 5, \pm 15 &
\pm 2,\pm 13,\pm 17,\pm 18& \pm 13,\pm 17\\
28 & \pm 2,\pm 18& - & \pm 5,\pm 15& -\\
34 &\pm 2,\pm 13, \pm 17,\pm 18& \pm 2, \pm 18 & \pm 13,\pm 17& \pm 13, \pm 17\\
\hline
\end{array}
$$
}
\caption{\label{tab:37} \footnotesize Information about existence of edges in
$X(\PSL(2,37),D_{36},\S_\xi^+)$ where $\xi \in S^*\cap S^*+1$, $\xi\ne \frac{1}{2}$ and $\xi\ne 1$.}
\end{table}

\begin{example}
\label{ex:61}
{\rm
Let $X$ be a basic orbital graph
arising   from the action of $G=\PSL(2,61)$ on the cosets
of $D_{p-1}=D_{60}$ from Row~8  or 9 of Table~\ref{tab:cases}.
Because of the isomorphism given in Proposition~\ref{pro:isomorphic}
we may assume that $X$ arises from Row~8  of Table~\ref{tab:cases},
that is, it is associated with a suborbit
$\S_\xi^+$,
where $\xi \in S^*\cap S^*+1$, $\xi\ne \frac{1}{2}$ and $\xi\ne 1$.
There are $14$ self-paired suborbits of length $30$,
giving $7$ non-isomorphic vertex-transitive graphs
of order $61\cdot 31=1891$ and of valency $30$.
(For example, this can be checked  using Magma \cite{Mag}.)
We may assume that $X$ is one of these $7$ graphs.

For $p=61$ we have
\begin{eqnarray*}
(S^*\cup\{0\})\cap (S^*\cup\{0\})+1 &=& \{0, 1, 4, 5, 13, 14, 15, 16, 20,  42, 46, 47, 48,  49,57, 58 \}\\
&=&\{0, 1 \} \cup \{a, a^{-1}\colon a\in\{4,5,13,14,15,16,20\}\}.
\end{eqnarray*}
The set of primitive roots in $F_{61}$ equals
$\mathcal{R}=\{\pm 2, \pm 6, \pm 7, \pm 10, \pm 17, \pm 18, \pm 26, \pm 30\}$.

By Table~\ref{tab:cases},
$V_\infty$ is adjacent to exactly one of the two
nontrivial orbits of a quasi $(2,15)$-semiregular automorphism $\sigma$.
Existence of a Hamilton cycle in each of these
graphs can be proved using the same arguments as in
Example~\ref{ex:37}. In particular, with the terminology of
Example~\ref{ex:37} one can see from Table~\ref{tab:61}
that for every $\xi$
there is at least one $\tau\in \mathcal{R}$ such that $V_1$ is adjacent to
$V_{\tau^2}$. We  conclude that there is a full
cycle in $X\la\OO(S^*)\ra$ preserved by the automorphism $\sigma$.
Since, by Table~\ref{tab:cases}, vertex $V_\infty$ is
adjacent only to vertices in $\OO(S^*)$
and since $X_\P$ is a quasi-bicirculant with  $V_\infty$
as the fixed vertex,
we deduce that the bipartite graph induced by the edges
with one endvertex in
$\OO(S^*)$ and the other in $\OO(N^*)$
contains a matching preserved by the automorphism $\sigma$.
Next, by  (\ref{eq:equation-eta}), we have
$$
\frac{\eta_{i}(\tau,\tau^3)}{\eta_{i}(1,\tau^2)}=\tau\in N^*, i\in\{1,2,3,4\},
$$
and so Table~\ref{tab:61} implies that
vertices in $X\la\OO(N^*)\ra$ are of
valency at least $2$. Namely, for every $\xi$
there exists at least one $\tau\in \mathcal{T}^+\cup\mathcal{T}^-$
such that $\eta_i(1,\tau)\notin S^*$ which implies that
$\eta_i(\tau,\tau^3)\in S^*$ and thus $V_\tau$ is adjacent to $V_{\tau^3}$.
We have thus proved that
the bicirculant $X_\rho-\{V_\infty\}$
contains a generalized Petersen graph as a subgraph
for every $\xi$.
Hence, combining together
Propositions~\ref{pro:GPG} and \ref{pro:GPG2}, we have that
$X_\rho-\{V_\infty\}$ contains a Hamilton cycle or (at least) a
Hamilton path with both endvertices in $\OO(S^*)$.
Since $V_\infty$ is adjacent to all of  vertices in
$\OO(S^*)$  this cycle/path can clearly be extend to
a Hamilton cycle in $X_\P$ which contains at least one double edge
in $X_\rho$. By Lemma~\ref{lem:cyclelift},
$X$ is hamiltonian.
}
\end{example}

\begin{table}[h!]
{\footnotesize
$$
\begin{array}{|c|c|c|c|c|}
\hline
\xi & \tau\in\mathcal{T}^+ &
\tau\in\mathcal{T}^+ \textrm{ s.t. } &
\tau\in\mathcal{T}^- & \tau\in\mathcal{T}^- \textrm{ s.t. } \\
 & & \eta_{1,2}(1,\tau^2) \in S^*& & \eta_{3,4}(1,\tau^2)\in S^* \\
\hline\hline
4 & \pm 2,\pm 6,\pm 10,\pm 30 & \pm 6,\pm 10 & \pm 2,\pm 30 & -\\
5 & \pm 6,\pm 7,\pm 10,\pm 17,\pm 18, \pm 26& \pm 6,\pm 7, \pm 10, \pm 26 & - & -\\
13 & \pm 2, \pm 6,\pm 7,\pm 10, & \pm 2,\pm 7, \pm 17, \pm 18, \pm 26
& \pm 2,\pm 6, \pm 10, \pm 30  & -\\
 & \pm 17,\pm 18, \pm 26, \pm 30 & && \\
14 &\pm 2,\pm 7, \pm 26,\pm 30& \pm 2,\pm 31
& \pm 7,\pm 17, \pm 18,\pm 26 & -\\
15 &\pm 7,\pm 26& - & \pm 2,\pm 17,\pm 18,\pm 30 & \pm 2,\pm 30\\
16 &\pm 2, \pm 6, \pm 10, \pm 30& \pm 6,\pm 10 &
\pm 6,\pm 10,\pm 17,\pm 18 & -\\
20 &\pm 6,\pm 10, \pm 17, \pm 18& \pm 6,\pm 10,\pm 17,\pm 18 & \pm 17,\pm 18 & \pm 17,\pm 18\\
42 &\pm 17,\pm 18 & \pm 17,\pm 18 & \pm 6,\pm 10,\pm 17,\pm 18 &
\pm 6, \pm 10, \pm 17,\pm 18\\
46 &\pm 6, \pm 10, \pm 17, \pm 18& - &
\pm 2, \pm 6, \pm 10, \pm 30 & \pm 6,\pm 10\\\
47 &\pm 2, \pm 17,\pm 18, \pm 30& \pm 2, \pm 30 & \pm 7,\pm 26 & -\\
48 &\pm 7, \pm 17,\pm 18,\pm 26& - & \pm 2,\pm 7, \pm 26, \pm 30 &
\pm 2, \pm 30\\
49 &\pm 2,\pm 6, \pm 10, \pm 30& - &
\pm 2, \pm 6,\pm 7, \pm 10 & \pm 2,\pm 7, \pm 17,\pm 18, \\
 & &   &
 \pm 17,\pm 18,\pm 26, \pm 30 & \pm 26, \pm 30\\
57 &-& - & \pm 6,\pm 7, \pm 10, \pm 17,\pm 18, \pm 26 &
\pm 6, \pm 7, \pm 10, \pm 26\\
58 &\pm 2, \pm 30& - & \pm 2, \pm 6, \pm 10, \pm 30 & \pm 6, \pm 10\\
\hline
\end{array}
$$
}
\caption{\label{tab:61} \footnotesize Information about existence of edges in
$X(\PSL(2,61),D_{60},\S_\xi^+)$ where $\xi \in S^*\cap S^*+1$, $\xi\ne \frac{1}{2}, 1$.}
\end{table}

With the approach used in Examples~\ref{ex:37} and~\ref{ex:61}
we will now prove  existence of Hamilton cycles in
any basic orbital graph arising from Rows 8 and 9
of Table~\ref{tab:cases}.

\begin{proposition}
\label{pro:row-8-9}
Let  $X=X(G,H,\W)$, where $\W\in\{S_{\xi}^+,S_{\xi}^-\}$,
$\xi\in S^*\cap S^*+1$ and $\xi\ne \frac{1}{2},1$,
be  one of the graphs arising from
Row 8 or 9 of Table~\ref{tab:cases}.
Then $X$ is hamiltonian.
\end{proposition}

\begin{proof}
Because of the isomorphism given in Proposition~\ref{pro:isomorphic}
we can assume that $\W=S_{\xi}^+$.
Note that $X$ is of valency $(p-1)/2$.

By Proposition~\ref{lem:infty} we have
$d(V_\infty)=0$,
$d(V_\infty,V_x)=2$ for every $x\in S^*$,
and $d(V_\infty,V_x)=0$ for every $x\in N^*$.
The number of edges $d(V_y,V_x)$, $x,y\in F^*$,
between $V_y$ and $V_x$ in $X_\rho$
is obtained from  (\ref{eq:equation-j}) and (\ref{eq:equation-eta}):
\begin{eqnarray}
\label{eq:jeta1}
j_{1,2}&=&\frac{1}{2}(y-x\pm\sqrt{x^2+2(1-2\xi)xy+y^2}), \\
\label{eq:jeta2}
j_{3,4}&=&\frac{1}{2}(y-x\pm\sqrt{x^2-2(1-2\xi)xy+y^2}), \\
\label{eq:jeta3}
\eta_{1,2}&=&\frac{y}{2\xi x}((2\xi-1)x-y \pm \sqrt{x^2+2(1-2\xi)xy+y^2}),\\
\label{eq:jeta4}
\eta_{3,4}&=&\frac{y}{2(1-\xi)x}((1-2\xi)x-y \pm \sqrt{x^2-2(1-2\xi)xy+y^2}).\end{eqnarray}
Note that the values of $j_{1,2,3,4}$ and $\eta_{1,2,3,4}$ for edges inside
$V_x$ are
\begin{eqnarray}
\label{eq:inside1}
j_{1,2}(x,x)=\pm x\sqrt{1-\xi}, \\
\label{eq:inside2}
j_{3,4}(x,x)=\pm x\sqrt{\xi},\\
\label{eq:inside3}
\eta_{1,2}(x,x)= \frac{x}{\xi}(\xi-1\pm\sqrt{-\xi}),\\
\label{eq:inside4}
\eta_{3,4}(x,x)=\frac{x}{1-\xi}(-\xi\pm\sqrt{\xi}).
\end{eqnarray}
Therefore it depends solely on $\xi$ whether
there are edges inside  the orbit $V_x$ or not.
In particular, we have that
\begin{eqnarray}
\label{eq:dVx}
d(V_x)\in\{0,2,4\} \textrm{ for } x\in F^*.
\end{eqnarray}
Suppose that $x=1$. Then we get from (\ref{eq:jeta1}) - (\ref{eq:jeta4})
the following values for the above quantities $j$ and $\eta$:
\begin{eqnarray*}
j_{1,2}&=&\frac{1}{2}(y-1\pm\sqrt{1+2(1-2\xi)y+y^2}), \\
j_{3,4}&=&\frac{1}{2}(y-1\pm\sqrt{1-2(1-2\xi)y+y^2}), \\
\eta_{1,2}&=&\frac{y}{2\xi}((2\xi-1)-y \pm \sqrt{1+2(1-2\xi)y+y^2}),\\
\eta_{3,4}&=&\frac{y}{2(1-\xi)}((1-2\xi)-y \pm \sqrt{1-2(1-2\xi)y+y^2}).
\end{eqnarray*}

Let us now consider elements of the form
$1+2(1-2\xi)g^2+g^4$ and $1-2(1-2\xi)g^2+g^4$, where
$g$ is a generator of $F^*$.
Since existence of Hamilton cycles in $X$ for $p\in\{13,37,61\}$
is proved in Examples~\ref{ex:13},~\ref{ex:37} and \ref{ex:61},
we may assume that $p\not\in\{13,37,61\}$.
Consequently, Theorem~\ref{the:polynomial-degree-4} and Proposition~\ref{pro:small} combined together imply
that for the two polynomials
$$
f(z)=1+2(1-2\xi)z^2+z^4 \textrm{ and } h(z)=1-2(1-2\xi)z^2+z^4
$$
there exist $g,g'\in F^*$ such that
$F^*=\la g\ra=\la g'\ra$ and $f(g),h(g')\in S^*\cup \{0\}$.
It follows that $V_1$ is adjacent to $V_{g^2}$ and to
$V_{(g')^2}$ in $X_\P$
depending on whether the corresponding values
$\eta_{1,2}$ and $\eta_{3,4}$
are squares or not.
Let $s=g^2$ and $s'=(g')^2$.
Then $s$ and $s'$ both generate $S^*$.
We claim that either  $V_1$ is adjacent to $V_s$
or $V_g$ is adjacent to $V_{sg}$, and moreover
either  $V_1$ is adjacent to $V_{s'}$
or $V_{g'}$ is adjacent to $V_{s'g'}$.
For this purpose we need to calculated the corresponding values
of
$\eta_1$, $\eta_2$, $\eta_3$, and $\eta_4$ for the following pairs of vertices
$(V_1,V_s)$, $(V_g, V_{gs})$,  $(V_1,V_{s'})$, $(V_{g'}, V_{g's'})$:
\begin{eqnarray*}
\eta_{1,2}(1,\bar{s})&=&\frac{\bar{s}}{2\xi}((2\xi-1)-\bar{s} \pm \sqrt{1+2(1-2\xi)\bar{s}+\bar{s}^2}),\\
\eta_{3,4}(1,\bar{s})&=&
\frac{\bar{s}}{2(1-\xi)}((1-2\xi)-\bar{s} \pm \sqrt{1-2(1-2\xi)\bar{s}+\bar{s}^2}),\\
\eta_{1,2}(\bar{g},\bar{g}\bar{s})&=&\bar{g}\eta_{1,2}(1,\bar{s}),\\
\eta_{3,4}(\bar{g},\bar{g}\bar{s})&=&\bar{g}\eta_{3,4}(1,\bar{s}),
\end{eqnarray*}
where $\bar{s}\in\{s,s'\}$ and $\bar{g}\in\{g,g'\}$.
Hence,
$$
\frac{\eta_{i}(\bar{g},\bar{g}\bar{s})}{\eta_{i}(1,\bar{s})}=\bar{g}\in N^*, i\in\{1,2,3,4\}.
$$
Consequently, for each $i$ exactly one of
$\eta_{i}(\bar{g},\bar{g}\bar{s})$ and
$\eta_{i}(1,\bar{s})$ belongs to $S^*$, implying
that in the bicirculant $X_\rho-\{V_\infty\}$
we have a full induced cycle preserved by the automorphism $\sigma$
either in the orbit
$\OO(S^*)$  or in the orbit $\OO(N^*)$.
We claim that $X_\rho-\{V_\infty\}$ contains a generalized Petersen graph.
In order to prove this claim
we only need to show that if there is an orbit that does not
contain a full cycle preserved by $\sigma$
then it contains a union of cycles (preserved by $\sigma$),
and that the bipartite graph between the two orbits
contains a matching preserved by $\sigma$.
The latter holds since $X$ is connected and since
$V_\infty$ is adjacent only to vertices in $\OO(S^*)$
 (see Table~\ref{tab:cases}).
Suppose therefore that one of the two orbits
$\OO(S^*)$  and $\OO(N^*)$ does not contain a full cycle.

Suppose first that  $\OO(N^*)$ does not contain the above
mentioned  full cycle. Then
$V_1$ is adjacent to both $V_s$ and $V_{s'}$,
implying that
\begin{eqnarray*}
val(X)=\frac{p-1}{2}&\ge& d(V_\infty,V_1)+ 4 + d(V_1) + |\sum d(V_1,V_x) \colon x\in N^*|.\\
&=&6+d(V_1)+|\sum d(V_1,V_x) \colon x\in N^*|.
\end{eqnarray*}
On the other hand, since, by assumption
the valency of the graph induced on $\OO(N^*)$ is either $0$ or $1$,
we have by calculating valency $val(X)$ at $V_x$, $x\in N^*$:
\begin{eqnarray*}
val(X)=\frac{p-1}{2}&=& \epsilon + d(V_x) +
 |\sum d(V_x,V_y) \colon y\in S^*|.
\end{eqnarray*}
where $\epsilon\in\{0,1\}$.
Since
$|\sum d(V_1,V_x) \colon x\in N^*|=|\sum d(V_x,V_y) \colon y\in S^*|$
it follows that $d(V_x)>4$, contradicting (\ref{eq:dVx}).

Suppose now that $\OO(S^*)$ does not contain the above
mentioned  full cycle. Similarly as above
it follows that each $V_x$, $x\in N^*$, has at least
$4$ neighbors in  $X\la \OO(N^*)\ra$.
This implies that
\begin{eqnarray*}
val(X)=\frac{p-1}{2}&\ge& 4 + d(V_x) + |\sum d(V_x,V_y) \colon y\in S^*|.
\end{eqnarray*}
On the other hand, since, by assumption
the valency of the graph induced on $\OO(S^*)$ is either $0$ or $1$,
we have by calculating valency $val(X)$ at $V_x$, $x\in S^*$:
\begin{eqnarray*}
val(X)=\frac{p-1}{2}&=&d(V_\infty,V_1) + \epsilon + d(V_1) + |\sum d(V_1,V_y) \colon y\in N^*|\\
&=&2+\epsilon + d(V_1) + |\sum d(V_1,V_y) \colon y\in N^*|,
\end{eqnarray*}
where $\epsilon\in\{0,1\}$.
It follows that $2+\epsilon+d(V_1)\ge 4+d(V_x)$,
and so $d(V_1)\ge 2-\epsilon+d(V_x)$.
Since, by (\ref{eq:dVx}),
$d(V_1)$ and $d(V_x)$ are both even numbers smaller than
or equal to $4$,
it follows that either  $d(V_1)=4$ and $d(V_x)=2$
or $d(V_1)=2$ and $d(V_x)=0$.
None of these is possible. In particular,
if $d(V_1)=4$ then, by (\ref{eq:inside1}) - (\ref{eq:inside4}),
we have
$$
\frac{1}{\xi}(\xi-1\pm\sqrt{-\xi})\in S^* \textrm{ and }
\frac{1}{1-\xi}(-\xi\pm\sqrt{-\xi}) \in S^*,
$$
and thus
$$\eta_{1,2}(x,x)=\frac{x}{\xi}(\xi-1\pm\sqrt{-\xi})\in N^*
\textrm{ and }
\eta_{3,4}(x,x)=\frac{1}{1-\xi}(-\xi\pm\sqrt{-\xi}) \in N^*,
$$
implying that $d(V_x)=0$.
Similarly, if $d(V_1)=2$ then  two of the above expressions
for $\eta_{i}(1,1)$, $i\in\{1,2,3,4\}$, are squares and the other two are non-squares,
which implies that either $\eta_{1,2}(x,x)$ or $\eta_{3,4}(x,x)$ is a square,
and consequently $d(V_x)=2$.

These contradictions
show that the valencies of both $X\la \OO(S^*)\ra$ and
$X\la \OO(N^*)\ra$ are at least $2$, which proves that
$X_\P-\{V_\infty\}$ contains a generalized Petersen graph
$GP((p-1)/4,k)$
as claimed. If $GP((p-1)/4,k)$ is not isomorphic to $GP(n,2)$
with $n=(p-1)/4\equiv{5\pmod 6}$, then
Proposition~\ref{pro:GPG} implies the existence
of a Hamilton cycle in $X_\P-\{V_\infty\}$.
Of course, this cycle contains an edge of the form
$V_xV_y$, $x,y\in S^*$.
Replacing this edge with a path $V_xV_\infty V_y$ gives
a Hamilton cycle in $X_\rho$.
Obviously this Hamilton cycle contains double edges,
and so, by Lemma~\ref{lem:cyclelift}, it lifts to a Hamilton cycle in $X$.
 We may therefore assume that $(p-1)/4\equiv {5\pmod 6}$
and that the generalized Petersen graph in $X_\P-\{V_\infty\}$
is isomorphic to $GP((p-1)/4,2)$.
In this case Proposition~\ref{pro:GPG2} implies
that there exists a Hamilton path in
$X_\P-\{V_\infty\}$  with both endvertices in
$\OO(S^*)$. By joining these two endvertices
with $V_\infty$ we can then
extend this path to  a Hamilton cycle in $X_\rho$
which lifts to a Hamilton cycle in $X$.
This completes the proof of Proposition~\ref{pro:row-8-9}.
\end{proof}



\section{Proof of Theorem~\ref{the:main}}
\label{sec:proof}
\indent

\begin{proofTT}
Let $X$ be a connected vertex-transitive graph of order $pq$,
where $p$ and $q$ are primes and $p\ge q$, other than the
Petersen graph.
If $q\in\{2,p\}$ then $X$ admits a Hamilton cycle
by Proposition~\ref{thm:p2,2p}.
We may therefore assume that $q\not\in\{2,p\}$.
Then $X$ is a generalized orbital graph arising from
one of the actions given in Theorem~\ref{the:main1}.
If $X$ is imprimitive then it admits a Hamilton cycle
by Proposition~\ref{thm:ham-imprimitive}.
We may therefore assume that $X$ is primitive, and so
$X$ is a generalized orbital graph arising from
one of the group actions given in Table~\ref{tab:groups}.
In fact, as explained in Section~\ref{sec:strategy},
we can assume that $X$ is a basic orbital graph arising
from a group action given in Table~\ref{tab:groups}.

If $X$ arises from one of  Rows 1, 2 and 3
of Table~\ref{tab:groups} then it admits a Hamilton cycle
by Proposition~\ref{the:small}.
If $X$ arises from the group action given in Row  4
of Table~\ref{tab:groups} then it admits a Hamilton cycle
by Proposition~\ref{pro:conder}.
If $X$ arises from the group action given in Row  5
of Table~\ref{tab:groups} then it admits a Hamilton cycle
by Propositions~\ref{pro:0} and~\ref{pro:ne0}.
If $X$ arises from the group action given in Row  6
of Table~\ref{tab:groups} then the existence
of a Hamilton cycle follows from Propositions~\ref{pro:row-1-2},
~\ref{pro:row-3-4},~\ref{pro:row-5},
~\ref{pro:row-6-7} and~\ref{pro:row-8-9}.
Finally, if $X$ arises from the group action given in Row  7
of Table~\ref{tab:groups} then the existence
of a Hamilton cycle follows from Proposition~\ref{pro:psl213}.
\end{proofTT}


\subsection*{Acknowledgements}

The authors wish to thank  Marston Conder and
Ademir Hujdurovi\'c
for helpful conversations about the material of this paper.



\end{document}